\documentclass[a4paper,11pt]{article}
\usepackage{amsmath}
\usepackage{geometry,amssymb}
\usepackage{fancyhdr}
\usepackage{bm}
\usepackage{tikz}
\usepackage{scalerel}
\usepackage{geometry}
\usepackage{cite}
\usepackage{indentfirst}
\definecolor{myblue}{rgb}{0.53,0.94,0.92}
\usepackage[bordercolor=white,
linecolor=myblue,
backgroundcolor=myblue,
shadow]
{todonotes}
\usepackage[pdftex,colorlinks=false,plainpages=false]{hyperref}
\hypersetup{
	colorlinks=true,
	linkcolor=black,
	urlcolor=black,
	citecolor=black
}
\numberwithin{equation}{section}
\newtheorem{theorem}{Theorem}[section]
\newtheorem{proposition}{Proposition}[section]
\newtheorem{lemma}{Lemma}[section]
\newtheorem{remarke}{Remark}[section]
\newenvironment{proof}[1]{\mbox{}\\\noindent\textbf{Proof of #1}.\par}{\hfill$\square$\par}
\newenvironment{remark}{\begin{remarke}}{\end{remarke}}
\newcommand{\step}[2]{\noindent\underline{Step {#1}}. {\it #2}.}

\newcommand{\brom}[1]{\uppercase\expandafter{\romannumeral#1}}
\newcommand{\Sch}{Schr{\"o}dinger }
\newcommand{\lot}{{\rm lot}}

\newcommand{\hs}{\hspace{-.1em}}
\newcommand{\hsh}{\hspace{-.05em}}
\newcommand{\phs}{\hspace{.1em}}
\newcommand{\phsh}{\hspace{.05em}}
\newcommand{\h}[1]{\hat{#1}}
\newcommand{\la}{\lambda}
\newcommand{\tta}{\theta}
\newcommand{\varl}{\varLambda}
\newcommand{\varp}{\varPhi}
\newcommand{\vart}{\varTheta}
\newcommand{\varg}{\vg}
\newcommand{\wh}[1]{\widehat{#1}}
\newcommand{\ti}[1]{\tilde{#1}}
\newcommand{\F}{F}
\newcommand{\f}{f}

\newcommand{\err}{\varPsi}
\newcommand{\errz}{\err_0}
\newcommand{\errzt}{\ti{\err}_0}

\newcommand{\mmod}{{\rm Mod}}
\newcommand{\mmodt}{\widetilde{\mmod}}

\newcommand{\m}{\mathbf{M}}
\newcommand{\rest}{\mathcal{R}}
\newcommand{\p}{\h{p}}

\newcommand{\Po}{\mathcal{P}_1}
\newcommand{\Pt}{\mathcal{P}_2}
\newcommand{\one}{\mathbf{1}}
\newcommand{\lf}{\left}
\newcommand{\rg}{\right}

\newcommand{\vv}{v}
\newcommand{\hv}{\h{\vv}}

\newcommand{\xx}{z}
\newcommand{\w}{w}
\newcommand{\wz}{\w_0}
\renewcommand{\wp}{\w^\perp}

\newcommand{\wzt}{\ti{\w}_0}
\newcommand{\W}{W}
\newcommand{\Wp}{\W^\perp}
\newcommand{\hw}{\h{\w}}
\newcommand{\hW}{\wh{\W}}
\newcommand{\vg}{\varGamma}
\newcommand{\ka}{k_1}
\newcommand{\kb}{k_2}
\newcommand{\kk}{\Delta k}
\newcommand{\ao}{\rho_1}
\newcommand{\bo}{\rho_2}
\newcommand{\oo}{\rho}
\newcommand{\les}{\lesssim}
\newcommand{\Ham}{H}

\newcommand{\mHam}{{\it \mathbb{H}}}
\newcommand{\Haml}{H_\la}
\newcommand{\mHaml}{\mathbb{H_\la}}

\newcommand{\mHamp}{\mathbb{H^\perp}}

\newcommand{\mHamlp}{\mathbb{H_\la^\perp}}
\newcommand{\mHamo}{\mHam^{\first}}
\newcommand{\mHamt}{\mHam^{\second}}
\newcommand{\mHamth}{\mHam^{\third}}
\newcommand{\A}{A}
\newcommand{\As}{A^\ast}
\newcommand{\Al}{A_\la}
\newcommand{\Asl}{A_\la^\ast}
\newcommand{\mA}{\mathbb{A}}
\newcommand{\mAs}{\mathbb{A^\ast}}
\newcommand{\mAl}{\mathbb{A_\la}}
\newcommand{\mAsl}{\mathbb{A_\la^\ast}}
\newcommand{\pr}{\partial}
\newcommand{\prt}{\pr_{t}}
\newcommand{\prs}{\pr_{s}}
\newcommand{\prb}{\pr_{b}}
\newcommand{\py}{\pr_y}
\newcommand{\hJ}{{\rm\hat{J}}}
\newcommand{\tJ}{{\rm\tilde{J}}}
\newcommand{\E}{\mathcal{E}}
\newcommand{\R}{{\rm R}}
\newcommand{\B}{B_{\delta}}
\newcommand{\Rw}{\R_{\hw}}
\newcommand{\av}{a}
\newcommand{\varpt}{\ti{\varp}}
\newcommand{\chib}{\chi_{B_1}}
\newcommand{\lchib}{(\varl\chi)_{B_1}}
\newcommand{\Q}{\mathcal{Q}}
\newcommand{\Qfour}{\Q_4}
\newcommand{\Qfive}{\Q_5}
\newcommand{\Qf}{\Q_{F}}

\newcommand{\lgb}{\log b}
\newcommand{\lgba}{\lf|\log b\rg|}
\newcommand{\lgyba}{\lf|\log(\rg.\hs\hs y\sqrt{b}\lf.\hs\hs)\rg|}
\newcommand{\db}{\delta(b^\ast)}
\newcommand{\bflog}{\frac{b^4}{\lgba^2}}
\newcommand{\Eb}{\E_4+\bflog}
\newcommand{\Ebb}{\bigg(\Eb\bigg)}
\newcommand{\sg}{\varSigma}
\newcommand{\sgb}{\sg_b}
\newcommand{\er}{e_{r}}
\newcommand{\etau}{e_{\tau}}
\newcommand{\ex}{e_{x}}
\newcommand{\ey}{e_{y}}
\newcommand{\ez}{e_{z}}
\newcommand{\ea}{E}
\newcommand{\tea}{\ti{E}}
\newcommand{\supscriptscale}{6.8pt}
\newcommand{\sups}[1]{\scaleto{(#1)}{\supscriptscale}}
\newcommand{\first}{\scaleto{(1)}{\supscriptscale}}
\newcommand{\second}{\scaleto{(2)}{\supscriptscale}}
\newcommand{\third}{\scaleto{(3)}{\supscriptscale}}
\def\onez{(1\hs+\hs Z)}
\def\pyzy{\Big(\py\hs\hs+\hs\hs\frac{Z}{y}\Big)}
\def\ptone{\phantom{\first}}

\def\RR{\mathbb{R}}
\def\ZZ{\mathbb{Z}}

\def\SS{\mathbb{S}}
\def\CCC{\mathcal{C}}
\def\SSS{\mathcal{S}}
\def\mf{\mathcal{F}}
\def\d{{\rm d}}
\def\iff{{\rm iff}}
\def\spann{{\rm span}}
\def\iset{\mathcal{I}}
\def\isetst{\mathcal{I}^{\ast}}
\def\sstar{s^{\ast}}

\title{Blowup dynamics for smooth equivariant solutions to energy critical Landau-Lifschitz flow}
\date{}
\author{Jitao Xu\footnote{School of Mathematical Sciences, University of Science and Technology of China, Hefei, Anhui, 230026, PR China, xujt@mail.ustc.edu.cn}, Lifeng Zhao\footnote{School of Mathematical Sciences, University of Science and Technology of China, Hefei, Anhui, 230026, PR China, zhaolf@ustc.edu.cn}}

\begin{document}
	\maketitle
	\thispagestyle{plain}
	\begin{abstract}
	In this paper, we study the energy critical 1-equivariant Landau-Lifschitz flow mapping $\RR^2$ to $\SS^2$ with arbitrary given coefficients $\ao\in\RR,\, \bo>0$. We prove that there exists a codimension one smooth well-localized set of initial data arbitrarily close to the ground state which generates type-\uppercase\expandafter{\romannumeral2} finite-time blowup solutions, and give a precise description of the corresponding singularity formation. In our proof, both the Schr\"odinger part and the heat part play important roles in the construction of approximate solutions and the mixed energy/Morawetz functional. However, the blowup rate is independent of the coefficients.
	\end{abstract}

	\section{Introduction}
	\label{S: introduction}
	
	\subsection{Setting of the problem}
	\label{SS: Setting of the problem}
	We consider the energy-critical Landau-Lifschitz equation
	\begin{equation} \label{eq: LL}
		\begin{cases}
			u_t=\ao u\wedge\Delta u -\bo u\wedge(u\wedge\Delta u), \\
			\lf. u \rg|_{t=0} = u_0 \in \dot{H}^1,
		\end{cases}
		\; (t,x)\in \RR\times\RR^2, \;\; u(t,x)\in\SS^2,
	\end{equation}
	where the exchange constant $\ao\in\RR$ and the Gilbert damping constant $\bo>0$. This equation was proposed by Landau and Lifshitz \cite{landau1992theory} in studying the dispersive theory of magnetization of ferromagnets. It describes the evolution of magnetization in classical ferromagnet, and thus the study on \eqref{eq: LL} is of fundamental importance for the understanding of nonequilibrium magnetism, see \cite{spinwaves} for more details. If $\ao=0$, \eqref{eq: LL} reduces to the harmonic map heat flow \eqref{eq: heat flow}, a model in nematic liquid flow \cite{BBCH}. If $\bo=0$, \eqref{eq: LL} becomes the conservative \Sch map flow \eqref{eq: sch map flow}, which is of much fundamental interest in differential geometry \cite{2002DingWeiyue_Schrodingerflow}.

	\medskip
	For equation $\eqref{eq: LL}$, the associated Dirichlet energy is given by
	\begin{equation} \label{dirichilet energy}
		E(u) = \int_{\RR^2} |\nabla u|^2 dx,
	\end{equation}
	which is dissipative along the flow
	\begin{equation} \label{eq: energy dissipative}
		\frac{d}{dt}\bigg(\int_{\RR^2}|\nabla u|^2\bigg) =-2\bo\int_{\RR^2}|u\wedge\Delta u|^2.
	\end{equation}
	Moreover, the energy \eqref{dirichilet energy} is invariant under the mixed symmetric transformations of the scaling and the rotation
	\begin{equation} \notag
		u(t,x)\to u_{\la,O}(t,x) = O u\bigg(\frac{t}{\la},\frac{x}{\la^2} \bigg), 
		\quad(\la,O)\in\mathbb{R_{+}^{\ast}}\times O(3).
	\end{equation}
	A remarkable feature of \eqref{eq: LL} is that smooth solutions preserve the magnitude $|u(t,x)|=1$ for all $(t,x)\in\RR_{+}^{\ast}\times\RR^2$, once we fix $|u_0|=1$ initially. In particular, there is a specific class of solutions with an additional symmetry, called the $k$-equivariant maps, which take the form
	\begin{equation} \label{eq: rotation}
		u(t,x)=e^{k\tta\R}\begin{bmatrix}
			u_1(t,r)\\ u_2(t,r)\\ u_3(t,r) \end{bmatrix},
		\quad\mbox{with}\quad
		\R=\begin{bmatrix}
			0 & -1 & 0 \\1 & 0 & 0 \\ 0 & 0 & 0
		\end{bmatrix},
	\end{equation}
	where $(r, \tta)$ is the polar coordinate on $\RR^2$, and $k\in \mathbb{Z}^{\ast}$ is the homotopy degree given by
	\begin{equation} \notag
		k=\frac{1}{4\pi} \int_{\RR^2}\big(\pr_1 u\wedge\pr_2 u\big)\cdot u.
	\end{equation}
	A typical stationary solution of \eqref{eq: LL} is the $k$-equivariant harmonic map
	\begin{equation} \notag
		Q_k(r,\tta)=\frac{e^{k\tta\R}}{1+r^{2k}} \begin{bmatrix}
			2r^k\\ 0 \\1\hs-\hs r^{2k}
		\end{bmatrix}, \quad k\in\ZZ.
	\end{equation}
	According to the Bogomol'nyi's factorization\cite{1976Bogomol_StabilityofClassicalSolutions}, $Q_k$ is the minimizers of the Dirichlet energy~\eqref{dirichilet energy} in homotopy-$k$ class with
	\begin{equation} \notag
		E(Q_k)=4\pi|k|.
	\end{equation}
	In other words, $Q_k$ is the ground state of \eqref{eq: LL}. Since this paper is mainly concerned with the 1-equivariant solutions, we use the convention $Q=Q_1$ for the 1-equivariant ground state.

	\subsection{Related geometric flows}
	\label{SS: heat flow Sch map}
	In the past decades, great progress has been made on both the harmonic heat flow problem and the \Sch map problem. For the harmonic heat flow,
	\begin{equation} \label{eq: heat flow}
	\mbox{(Heat flow)}\quad
		\begin{cases}
	 		u_t = \Delta u + |\nabla u|^2 u, \\
			\lf. u \rg|_{t=0} = u_0,
		\end{cases}
		\; (t,x) \in  \RR\times\RR^2, \quad u(t,x) \in \SS^2,
	\end{equation}
	we refer to
	\cite{kung1989heat, eells1964harmonic, lin1998energy, 1990ChenYunmei_HarmonicMapBlowupandGlobalExistence} for existence and uniqueness results. Since 2-dimensional heat flow is energy critical, singularity formation by energy concentration is possible. It is known that concentration implies the bubbling off of a nontrivial harmonic map at a finite number of blowup points, see Struwe \cite{1985Struwe}, Ding and Tian \cite{dingtian}, Qing and Tian \cite{qing1997bubbling}, Topping \cite{topping2004winding}. For $k$-equivariant case, blowup near $Q_k$ for $k \geq 3$ has been ruled out in \cite{2010Gustafson_HarmonicHeatAsymptoticStability}, where the harmonic map is proved to be asymptotically stable. Chang, Ding and Ye \cite{changdingye} found the first example of the finite-time blowup solutions of heat flow.  For $\mathbb D^2$ initial manifold and $\mathbb S^2$ target, van den Berg, Hulshof, and King \cite{BHK} implemented a formal analysis based on the matched asymptotics techniques and predicted the existence of blow-up solutions with quantized rates
	\begin{equation}
	\la_L(t)\approx \frac{C|T-t|^L}{\lf|\log (T-t)\rg|^\frac{2L}{2L-1}},\quad L\in\mathbb N^*. \notag
	\end{equation}
	For $\mathbb R^2$ initial manifold, Rapha{\"e}l and Schweyer \cite{2013Raphael_HarmonicHeatQuantizedBlowup, 2015Raphael_HarmonicHeatQuantizedBlowup} exhibited a set of initial data arbitrarily close to the least energy harmonic map in the energy-critical topology such that the corresponding solutions blow up in finite time with the quantized blow-up rate $\la_L(t)$ for any $L\geq 1$.  The case $L = 1$ corresponds to the stable regime. Without equivariant assumption, Davila, del Pino and Wei \cite{DPJ} constructed a solution  which blows up precisely at finite number of given points if the starting manifold is a bounded domain in $\mathbb R^2$. The profile around each point is close to an asymptotically singular scaling of a 1-corotational harmonic map with blowup rate $\la_L(t)$, $L=1$. This rate was expected to be generic for 1-corotational heat flow, see \cite{BHK}. For more results on harmonic map heat flow, see \cite{LW} and the references therein.

	\medskip
	For the \Sch map problem
	\begin{equation} \label{eq: sch map flow}
	\mbox{(\Sch map)}\quad
		\begin{cases}
			u_t = u \wedge \Delta u, \\
			\lf. u \rg|_{t=0} = u_0,
		\end{cases}
		\; (t,x) \in  \RR\times\RR^2, \quad u(t,x) \in\SS^2,
	\end{equation}
	there are quite many works. The local well-posedness of \Sch map was established in \cite{sulem1986continuous, ding2001local, mcgahagan2007approximation}. When the target is $\mathbb S^2$, the global well-posedness for small data in critical spaces was proved by Bejenaru, Ionescu, Kenig and Tataru \cite{BIKT2}. Their result was extended to the target of K{\"a}hler manifolds by Li \cite{2018Lize2, 2019Lizeh}. The stationary solutions of \Sch flow are harmonic maps. When the energy is less than $4\pi$, 1-equivariant solutions are global in time and scatter, see \cite{BIKT1}. Gustafson, Kang, Tsai, Nakanishi \cite{2007Gustafson_SchrodingerFlow, 2008Gustafson_HarmonicMapAsymptotic, 2010Gustafson_HarmonicHeatAsymptoticStability} showed that harmonic maps are  asymptotically stable in high equivariant classes ($k\leq 3$), which precludes the blowup solutions near the harmonic maps in those cases. However, in 1-equivariant class, Bejenaru and Tataru~\cite{BT} showed that harmonic maps are stable under smooth well-localized perturbations, but unstable under $\dot{H}^1$ topology. Merle, Rapha{\"e}l and Rodnianski \cite{2011Merle_SchMapBlowup} proved the existence of a codimension one set of smooth well-localized initial data arbitrarily close to $Q$ which generates type-\brom{2} blowup solutions, and they figured out the detailed asymptotic behavior of these solutions near the blowup time. Perelman \cite{2014Galina_Blowup} presented another type-\brom{2} blowup solution with different singularity formation.
	
	\medskip
	There are a lot of works devoted to the study of Landau-Lifshitz flow. The global existence of weak solutions or partial regular solutions have been shown in \cite{AS, CF, GH, KO, Mel, Wang}. However, the dynamical behavior is much less studied. The asymptotic stability of ground state harmonic maps in high equivariant classes were proved by Gustafson, Nakanishi and Tsai \cite{2010Gustafson_HarmonicHeatAsymptoticStability}. In 1-equivariant class, the solutions with energy less than $4\pi$ were proved to converge to a constant map in energy space by Li and Zhao \cite{2017LizeAsymptotic}.	The equivariant blowup solution was constructed in \cite{BW} by formal asymptotics and was verified by numerical experiments. However, as the authors said in \cite{BW} `{\it mathematically rigorous justification is required}'. As far as we know, the problem of blowup dynamics of the Landau-Lifschitz equation near $Q$  remains open until now.

	\subsection{Statement of the result}
	\label{SS: main result}
	Our work is in continuation of the investigation of the \Sch map \cite{2011Merle_SchMapBlowup}, the wave map \cite{2012Raphael_WaveMapBlowup}, the harmonic heat flow \cite{2015Raphael_HarmonicHeatQuantizedBlowup}. We establish the existence of smooth 1-equivariant type-\brom{2} blowup solutions to \eqref{eq: LL}, and gives a sharp description on the asymptotic behavior on its singularity formation.

	\begin{theorem}[Existence and description of blowup LL flow]
		\label{th: main th}
		There exists a set of smooth well-localized 1-equivariant initial data with its elements arbitrarily close to the ground state $Q$ in $\dot{H}^1$ topology, such that the corresponding solution to \eqref{eq: LL} blows up in finite time. The singularity formation corresponds to the concentration of the universal bubble of energy in the scale invariant energy space:
		\begin{equation} \label{eq: singularity formation}
			u - e^{\vart(t)\R} Q\lf( \frac{x}{\la(t)} \rg) \to u^{\ast}\in\dot{H}^1\;\;\mbox{as}\;\; t\to T,
		\end{equation}
		for geometrical parameters $\big(\vart(t), \la(t)\big)\in \CCC^1\big(\hspace{.1em}[\hspace{.1em}0,T),\RR\hs\times\hs\RR_{+}^{\ast}\big)$ with their asymptotic behaviors near the blowup time $T$ given by
		\begin{gather}
			\la(t)= C(u_0)\big(1+o(1)\big)\frac{(T-t)}{\lf|\log (T-t)\rg|^2}, \;\;C(u_0)>0,
			\label{eq: lambda}\\			
			\vart(t)\to\vart(u_0)\in\RR 
			\;\;\mbox{as}\;\; t\to T.
			\label{eq: Theta}
		\end{gather}
		Moreover, there holds the propagation of regularity:
		\begin{equation} \label{eq: propagation of regularity}
			\Delta u^{\ast} \in L^2.
		\end{equation}
	\end{theorem}
	\paragraph{Comments on the result}
	\par 1. {\it On the blowup asymptotics}: The overall blowup behavior is similar to that of \Sch map problem \cite{2011Merle_SchMapBlowup}. One of the major distinctions is that, due to the additional heat flow term (also referred to as the Lifscthiz disspation), the approximate solution is completely different from that of \Sch map problem. It behaves like a combination of the approximate solution of  \eqref{eq: heat flow} and \eqref{eq: sch map flow}. However, the appearance of the heat flow term does not deteriorate the error estimate.
	
	\medskip
	\par 2. {\it On the Morawetz estimate}: As in \cite{2011Merle_SchMapBlowup}, a unsigned quadratic term induced by the commutator $\big[\prt, \mHam\big]$ appears in the plain energy identity. It can not be controlled via the energy bounds directly (see \eqref{eq: uncontrollable term}), and thus requires an extra Morawetz functional to create cancellation. The construction of a suitable functional is the core of the analysis. However, the presence of the heat flow term make our situation very complicated, which is the main difficulty of this work. The key observation we have made is the intrinsic structure of the operator $\mA,\mAs$ gathered in Lemma \ref{le: structure}, which enables us to capture the competion of the \Sch part and the heat flow part, and formulate the uncontrollable term. In our construction, the ratio of $\ao, \bo$ also plays a crucial role, based on which the exquisitely-designed coefficients \eqref{eq: cccc} are eventually responsible for the distinct controls of various errors coming from the Morawetz estimate \eqref{eq: morawetz eq}.
		
	\medskip
	\par 3. {\it On the universality of the blowup rate}: The blowup rate \eqref{eq: lambda} is independent of the coefficients $\ao, \bo$, in spite of the evident influence of the latter. In fact, the subdued contribution of the coefficients has been observed by van den Berg and Williams through formal asymptotics \cite{BW}. Indeed, we shall see in Lemma \ref{le: ap solution}  that the coupling coefficients involving $\ao, \bo$ in front of the approximate profiles are consistent with the expressions in the flux computation \eqref{eq: flux computation}. This correspondence produces cancellations in computing the modulation equations, which gives the identical dynamic system for the modulation parameters \eqref{eq: modulation eq} as in \cite{2011Merle_SchMapBlowup, 2013Raphael_HarmonicHeatQuantizedBlowup}. That explains, to some extend, the reason why the blowup rate is irrelevant to the coefficients. 
		
	\medskip
	\par 4. {\it On the codimension one instability}: In our construction, the initial data of the blowup solutions are characterized by specifying the modulation parameters $a,b$ and the radiation term $\w$. The choice of $b_0$ and $\w_0$ resist small perturbations, while $a_0$, representing the time derivative of the phase, is unstable and thus selected accordingly afterward.  as shown in Section \ref{S: the trapped regime}. This regime ensures, in some weak sense, that the solutions evolve from a codimension one set of smooth initial data will blow up in the way we describe in Theorem\ref{th: main th}.

	\paragraph{Notions}
	We use the polar coordinates $(r,\tta)$ and $(y,\tta)$ on $\RR^2$, where by the anticipated scaling transform, we will set $y = r/\la$. We use the convention
	\begin{equation} \notag
		\pr_\tau = \frac{1}{y}\pr_\tta,
		\quad \varl f = y \cdot \nabla_y f.
	\end{equation}
	For any given parameter $b$, we introduce the scales
	\begin{equation} \label{eq: scales}
		B_0 = \frac{1}{\sqrt{b}}, \quad B_1 = \frac{\lgba}{\sqrt{b}}.
	\end{equation}
	The cut-off function $\chi$ is a smooth radially symmetric function defined by
	\begin{equation} \notag
	\chi(x) =\begin{cases}
		1, \quad |x|\leq1, \\
		0, \quad |x|\geq2,
	\end{cases}
	\end{equation}
	with its scaling given by
	\begin{equation} \notag
		\chi_M(x) = \chi\Big(\frac{x}{M}\Big),
	\end{equation}
	provided any large constant $M>0$.
	
	\subsection{Strategy of the proof}
	Let us briefly sketch our approach for proving the Theorem \ref{th: main th}, which follows from the strategy developed in \cite{2011Merle_SchMapBlowup, 2012Raphael_WaveMapBlowup, 2013Raphael_HarmonicHeatQuantizedBlowup}.
	
	\medskip
	\step{1}{Renormalization} We look for the 1-equivariant solution with its energy slightly higher than the ground state $Q$ which takes the form
	\begin{equation} \label{eq: sol form 0}
		u(t)=e^{\vart(t)\R}(Q+\hv)\lf(t,\frac{r}{\la(t)}\rg),
	\end{equation}
	where $\hv$ is a perturbation with small enough Sobolev norm
	\begin{equation} \notag
		\|\hv\|_{\dot{H}^1} \ll 1.
	\end{equation}
	The blowup mechanism suggests $\la\to 0$ as $t\to T$, and thus it is suitable to consider the renormalized function $v(s,y)$ under the self-similar transformation
	\begin{equation} \label{eq: ss transform}
		u(t)=e^{\vart \R}v(s,y), \quad
		\frac{ds}{dt}=\frac{1}{\la(t)^2}, \quad
		\frac{y}{r}=\frac{1}{\la(t)}.
	\end{equation}
	Then from \eqref{eq: LL} we obtain the renormalized equation for $v$:
	\begin{equation} \notag
		\prs v -\frac{\la_s}{\la}\varl v +\vart_s\R v =\ao v\wedge v -\bo v\wedge(v\wedge\Delta v).
	\end{equation}
	To understand the solution's behavior in the vicinity of the ground state, we apply the Frenet basis associated to $Q$, that is $[e_r,e_{\tau},Q]$ with
	\begin{equation} \notag
		\er=\frac{\pr_y Q}{|\pr_y Q|}, \quad
		\etau=\frac{\pr_{\tau}Q}{|\pr_{\tau}Q|}, \quad
		Q(y,\tta)=e^{\tta\R} \begin{bmatrix} \varl\phi(y) \\ 0 \\ Z(y) \end{bmatrix},
	\end{equation}
	where $\phi, \varl\phi, Z$ are given by \eqref{eq: ground state}, \eqref{eq: lambda phi Z}. Then replacing $v$ by $Q+\hv$, we reformulate \eqref{eq: LL} and encounter the following repeated linear pattern in $\hv$:
	\begin{equation} \notag
		\Delta \hv + |\nabla Q|^2 \hv,
	\end{equation}
	which gives rise to the linearized Hamiltonian/\Sch operator
	\begin{equation} \notag
		\Ham = -\Delta +\frac{V(y)}{y^2}, \quad\mbox{with}\quad
		V(y)=\frac{y^4-6y^2+1}{(1+y^2)^2},
	\end{equation}
	and also its vectorial version $\mHam$ \eqref{eq: mHam func}, \eqref{eq: mHam}. Finally, we obtain the following component equations of $\hv$, which is equivalent to \eqref{eq: LL}. ($\lot$ stands for the lower order errors)
	\begin{equation} \label{eq: component equation 0}
	\lf\{\begin{aligned}
	\prs \h\alpha - \frac{\la_s}{\la}\varl\h\alpha
	&= \, \ao \Ham\h\beta - \bo\Ham\h\alpha + \frac{\la_s}{\la}\varl\phi + \lot, \\
	\prs \h\beta - \frac{\la_s}{\la}\varl\h\beta
	&= - \ao \Ham\h\alpha - \bo \Ham\h\beta
	- \vart_s Z\h\alpha + \lot, \\
	\prs \h\gamma - \frac{\la_s}{\la}\varl\h\gamma
	&= - \ao \h\alpha\Ham\h\beta + \ao \h\beta\Ham\h\alpha
	+ \bo \h\alpha\Ham\h\alpha + \bo \h\beta\Ham\h\beta
	- \frac{\la_s}{\la}\varl\phi\h\alpha + \lot,
	\end{aligned}\rg.
	\end{equation}
	Under the convention $\hv=[\er, \etau, Q]\hw$, \eqref{eq: component equation 0} can be rewritten as the frequently used vectorial form
	\begin{equation} \label{eq: vec wz equation 0}
		\prs\hw -\frac{\la_s}{\la}\varl\hw + \vart_s Z\R\hw +\hJ\Big(\,\ao\mHam\hw -\bo\hJ\mHam\hw +\p\,\varl\phi\,\Big) = 0.
	\end{equation}

	\medskip
	\step{2}{Construction of the approximate solution} To characterize the evolution of $\la_s, \vart_s$ appearing in \eqref{eq: component equation 0}, we define two more modulation parameters $a,b$, and claim the following slow modulated ansatz:
	\begin{equation} \notag
		a\approx-\vart_s, \quad
		b\approx-\frac{\la_s}{\la}, \quad
		a_s\approx 0, \quad b_s\approx -(b^2+a^2).
	\end{equation}
	Then we seek for the approximate solution of \eqref{eq: vec wz equation 0} whose leading part is
	\begin{equation} \label{eq: ap solution 0}
		\wz = a\phsh\varp_{1,0} +b\phsh\varp_{0,1} +b^2\phsh S_{0,2},
	\end{equation}
	where $\varp_{1,0}, \varp_{0,1}$ are the first order profiles  responsible for the cancellation of the expressions inside the big brace in \eqref{eq: vec wz equation 0}, while $S_{0,2}$ is the second order profile dominating the third component $\h\gamma$, which is chosen to obey the constraint
	\begin{equation} \notag
		\h\alpha^2 +\h\beta^2 +(1+\h\gamma)^2=1.
	\end{equation}
	Moreover, we set the higher order profiles $\varp_{i,j}$ \eqref{eq: varp asymptotics 2} to improve the precision of the approximate solution. In general, the constructions of these profiles are mainly concerned with solving a typical type of elliptic equations
	\begin{equation} \label{eq: elliptic equation}
		\begin{cases}
			\Ham \varp_{i,j} =E_{i,j}(\varl\phi, \varp_{k,l}), 
			& \mbox{for}\;\; i+j=1, \\
			\Ham \varp_{i,j} =E_{i,j}(\varl\phi, \varp_{k,l}, S_{k,l}), & \mbox{for}\;\; 2\leq i+j\leq 3,
		\end{cases}
	\end{equation}
	with $k\leq i,\, l\leq j,\, k+l<i+j$. Its solution can be obtained by the method of variation of constants with the known resonance $\varl\phi$ of the Hamiltonian $\Ham$. The constructions of $\varp_{1,0}, \varp_{0,1}$ \eqref{eq: varp function} are quite straightforward, while the higher order profiles $\varp_{i,j}$ require further manipulations, owing to their fast growing tails as $y\to+\infty$ induced by $T_1-\varl T_1$ where
	\begin{equation} 
	\lf\{\begin{aligned}
		& T_1=-y\log y+y+O\bigg(\frac{(\log y)^2}{y}\bigg), \\
		& \varl T_1=-y\log y+O\bigg(\frac{(\log y)^2}{y}\bigg).
	\end{aligned}\rg.
	\end{equation}
	To handle this, a radiation $\sgb$ \eqref{eq: sigma b} has been applied to modify the recursive system \eqref{eq: elliptic equation}, as illustrated in \eqref{eq: lambda T1 Sigma b outside}. Furthermore, the presence of $\sgb$ functions have profound influence on the error $\errz$ of the approximate solution, the coupling formulas derived from which eventually leads to the forthcoming modulation equations
	\begin{equation} \label{eq: modulation equation 0}
		a_s = -\frac{2ab}{\lgba} +\lot, \quad
		b_s = -b^2 \bigg(1+\frac{2}{\lgba}\bigg) +\lot,
	\end{equation}
	as well as $a=-\vart_s+\lot$, $b=-\la_s/\la+\lot$. These identities constitute an ODE system, which determines the asymptotics for the modulation parameters. By refined computations on the ODE system, we see given $\la_0>0$, the scaling parameter $\la(t)$ change its sign in finite time $T$, which yields the finite-time blowup.
	
	\medskip
	\step{3}{Control of the radiation} Based on the approximate solution $\wzt=\chib\wz$ localized at the scale $B_1$, we seek for an actual solution, by showing the existence of the radiation $\w$ in the decomposition
	\begin{equation} \notag
		\hw = \wzt + \w.
	\end{equation}
	Recall \eqref{eq: sol form 0}, it remains to prove the existence of $\w$ in suitable Sobolev space that permits the modulation equations \eqref{eq: modulation equation 0}, but this requires the estimate of the form
	\begin{equation} \label{eq: first estimate}
   		\sum_{i=0}^4\int\frac{|\py^i\w|^2}{1+y^{8-2i}} \les \bflog,
	\end{equation}
	which is the key to finish the analysis. This is accomplished by the energy method together with the bootstrap argument. To be specific, we introduce the energies of $\w$, namely $\E_1, \E_2, \E_4$, at the level of $\dot{H}^1, \dot{H}^2, \dot{H}^4$ respectively, and then claim the bootstrap bounds of the energies \eqref{eq: pointwise energy 1}--\eqref{eq: pointwise energy 3} on a small initial interval. By proving that refined bounds \eqref{eq: refined pointwise energy 1}-\eqref{eq: refined pointwise energy 3} hold, we know the radiation $\w$ is trapped in the bootstrap regime, and then the estimate \eqref{eq: first estimate} follows. To show the refined bounds, we derive the following differential inequality for $\E_4$:
	\begin{equation} \label{eq: energy estimate 0}
		\frac{\d}{dt} \bigg(\frac{\E_4}{\la^{6}}\bigg) 
		\leq \frac{b}{\la^8} \bigg[ 2\big(1-d_2+o(1)\big) \E_4 + O\bigg(\bflog\bigg) \bigg],
	\end{equation}
	which is the core of our analysis. The main difficulty is the appearance of the unsigned quadratic terms emerging therefrom:
	\begin{equation} \label{eq: uncontrollable term 0}
		\int\R\mHaml\W_2^0\cdot
		\Big[-\mHamlp\big(G\W^\perp\big)+G\R\W_2^0\,\Big].
	\end{equation}
	This term, however can not be handled using the same Morawetz estimate in \cite{2011Merle_SchMapBlowup}, because of its intrinsic connection with the coefficients $\ao,\bo$. Therefore we construct a new Morawetz term $\m(t)$  using the factorization of $\mHamp$ \eqref{eq: mHamp factori}, with its coefficients involving $\bo/\ao$, to cancel \eqref{eq: uncontrollable term 0}. Analysis behind this cancellation splits into two cases: $|\bo/\ao|\leq 1$ and $|\bo/\ao|>1$, and in both cases the interaction of \eqref{eq: uncontrollable term 0} and Morawetz term will not perturb \eqref{eq: modulation equation 0}, \eqref{eq: energy estimate 0}. This fact suggests the invariance of the blowup speed against different $\ao, \bo$. Finally, we apply \eqref{eq: energy estimate 0} to close the bootstrap arguments, and the description of the blowup speed as well as the behavior of the residue term $\w$ are direct consequences of \eqref{eq: first estimate}, \eqref{eq: modulation equation 0}, which yield Theorem~\ref{th: main th}. 
	
	\paragraph{Structure of the paper}
	In Section \ref{S: 1-equiv flow in frenet basis}, we introduce the Frenet basis, renormalized variables and the linear Hamiltonian $\Ham$, and convert the problem to the vectorial form.
	In Section \ref{S: construction of the ap sol}, we construct an approximate solution $\w_0$ to the renormalized equation, and also its localized version $\wzt$, then give estimates on corresponding error terms.
	In Section \ref{S: the trapped regime}, we are aimed at seeking for an actual solution, so we consider the correction term $\w$, set up the trapped regime, introduce the definitions of energies $\E_1,\E_2,\E_4$, and obtain the modulation equations.
	In Section \ref{S: energy method}, we compute the mixed energy identity/Morawetz formula, and close the bootstrap bounds.
	In Section \ref{S: des on singularity formation}, we exhibit the that the finite-time blowup, and give a precise description of the blowup speed as well as the behavior of the correction term $\w$.
	In Appendix \ref{S: appendix A}, we list some results from \cite{2011Merle_SchMapBlowup} on the interpolation bounds of the correction term $\w$, which are frequently used in Section \ref{S: energy method}.
	In Appendix \ref{S: appendix B}, we give a proof of the Lemma \ref{le: gain of 2d}, which is also a crucial component in Section \ref{S: energy method}.

	\paragraph{Acknowledgements}
	J. Xu would like to thank F. Merle for their fruitful discussions and his kind advices on this work. 
	The contents of the second and third sections were finished when J. Xu was visiting Universit\'e de Cergy-Pontoise with financial support from University of Science and Technology of China, so he thank especially for the nice hospitality of UCP as well as the generous support from USTC.
	L. Zhao would like to thank P. Rapha\"el for suggesting this interesting problem to him
	during his visit in Universit\'e de Nice Sophia-Antipolis in 2018.
	L. Zhao is supported by NSFC Grant of China No. 12271497  and the National Key Research and Development Program of China  No. 2020YFA0713100.

	
	\section{1-equivariant flow in the Frenet basis}
	\label{S: 1-equiv flow in frenet basis}
	
	\subsection{The ground state and Frenet basis}
	\label{SS: ground state and the Frenet basis}
	Let us introduce the geometric set up of the 1-equivariant solution $u$ of the Landau-Lifschitz equation \eqref{eq: LL}. Maps with values in $\SS^2$ will be viewed as maps into $\RR^3$ with image parameterized by Euler angle $(\phi, \tta)$. The ground state $Q$ is a harmonic map of homotopy degree $k=1$ satisfying
	\begin{equation}
	\Delta Q = - |\nabla Q|^2 Q.
	\label{groud state eq}
	\end{equation}
	Its explicit formula is
	\begin{equation} \label{eq: ground state}
	Q(r,\tta) = \begin{bmatrix} \sin(\phi(r))\cos(\tta) \\ \sin(\phi(r))\sin(\tta) \\ \cos(\phi(r))
	\end{bmatrix},
	\quad\mbox{with}\quad \phi(r)=2\arctan(r).
	\end{equation}
	For the ease of notations, we define the following functions
	\begin{equation} \label{eq: lambda phi Z}
	\lf\{\begin{aligned}
	& \varl\phi(r) =r\pr_r\big(2\arctan(r)\big) =\frac{2r}{1+r^2} =\sin(\phi(r)), \\
	& Z(r) =\frac{1-r^2}{1+r^2} =\cos(\phi(r)),
	\end{aligned}\rg.
	\end{equation}
	according to which \eqref{eq: ground state} is rewritten as
	\begin{align}
	Q(r,\tta) = \begin{bmatrix} \varl\phi(r) \cos(\tta) \\ \varl\phi(r) \sin(\tta) \\ Z(r) \end{bmatrix} = e^{\tta \R} \begin{bmatrix} \varl\phi(r) \\ 0 \\ Z(r) \end{bmatrix}.
	\notag
	\end{align}
	In order to study the flow in the vicinity of $Q$,
	we choose a suitable gauge, namely Frenet basis $\big[\er, \etau, Q\big]$, with
	\begin{equation} \label{eq: frenet}
	\begin{gathered}
	\er = \frac{\pr_r Q}{|\pr_r Q|} = e^{\tta \R}
	\begin{bmatrix} Z \\ 0 \\ -\varl\phi
	\end{bmatrix},\quad
	\etau = \frac{\pr_\tau Q}{|\pr_\tau Q|} = e^{\tta \R}
	\begin{bmatrix} \,0\, \\ \,1\, \\ \,0\, \end{bmatrix}.
	\end{gathered}
	\end{equation}
	The action of derivatives and rotations in the Frenet basis are given by the following lemma. The proof follows from direct computations.
	
	\begin{lemma}[Derivation and rotation of Frenet basis \cite{2011Merle_SchMapBlowup}]
		\label{le: frenet basis}
		There holds:\\
		{\rm(i)} Action of derivatives:
		\begin{equation} \label{def: derivation of the Frenet basis}
		\begin{aligned}
			&\lf\{\begin{aligned}
				&\pr_r \er = -(1+Z)Q,\\
				&\pr_r \etau = 0,\\
				&\pr_r Q = (1+Z)\er,
			\end{aligned}\rg. \qquad\quad\; 
			\lf\{\begin{aligned}
				&\varl \er = -\varl\phi\,Q,\\
				&\varl \etau = 0,\\
				&\varl Q = \varl\phi\,\er,
			\end{aligned}\rg.\\
			&\lf\{\begin{aligned}
				&\pr_{\tau} \er =\frac{Z}{r}\etau,\\
				&\pr_{\tau} \etau =-\frac{Z}{r}\er -(1+Z)Q,\\
				&\pr_{\tau} Q =(1+Z)\etau,
			\end{aligned}\rg. \quad 
			\lf\{\begin{aligned}
				&\Delta \er =-\frac{1}{r^2}\er -\frac{2Z(1+Z)}{r}Q,\\
				&\Delta \etau = -\frac{1}{r^2}\etau,\\
				&\Delta Q =-2(1+Z)^2 Q.
				\end{aligned}\rg.
		\end{aligned}
		\end{equation}
		{\rm(ii)} Action of rotations:
		\begin{equation} \notag
			\R\er = Z\etau, \quad \R\etau = -Z\er -\varl\phi\phsh Q,
			\quad \R Q = \varl\phi\,\etau. 
		\end{equation}
		Moreover, the scaling and rotation symmetries yield the two parameters family of the harmonic map
		\begin{equation} \notag
			Q_{\vart,\la}(r) = e^{\vart \R}Q\Big(\frac{r}{\la}\Big),
			\quad (\vart,\la)\in \RR\times\RR_{+}^{\ast},
		\end{equation}
		with the infinitesimal generators:
		\begin{equation}
		\frac{d}{d \la}(Q_{\vart,\la})\big|_{\la=1, \vart=0}= -\varl\phi\, \er,\quad
		\frac{d}{d \vart}(Q_{\vart,\la})\big|_{\la=1, \vart=0}= -\varl\phi\, \etau.
		\notag
		\end{equation}
	\end{lemma}

	\subsection{The component and vectorial formula}
	\label{SS: the component formulas and the vectorial formula}
	Let us introduce two time-dependent geometrical parameters $\vart(t), \la(t)$, and look for solutions of a specific form
	\begin{equation} \label{eq: sol form}
		u = \SSS(Q+\hv)
		= e^{\vart(t)\R}(Q+\hv)\bigg(t,\frac{r}{\la(t)}\bigg),
	\end{equation}
	where we denote $\SSS$ the mixed transformation corresponding to scaling and rotation induced by $\la, \vart$, and the map $\hv$ is a small perturbation near $Q$. We express $\hv$ under component form with respect to the Frenet basis
	\begin{equation} \label{hv hw}
		\hv=\big[\er,\etau,Q\big]\hw,\quad 
		\hw=\big[\h\alpha,\h\beta,\h\gamma\big]^{T}.
	\end{equation}
	As the image of $u$ lies in $\SS^2$, the component functions $\h\alpha,\h\beta,\h\gamma$ obey the constraint
	\begin{equation} \label{eq: component constraint}
		\h\alpha^2+\h\beta^2+(1+\h\gamma)^2=1.
	\end{equation}
	The LL map problem is equivalent to the existence of the component functions $\h\alpha,\h\beta,\h\gamma$, coupling with the dynamical behavior of the geometrical parameters $\la, \vart$. In this subsection we derive the equation for $\h\alpha,\h\beta,\h\gamma$ and its vectorial form. We apply the commonly used self-similar transformation:
	\begin{equation} \label{eq: var sub}
		\frac{ds}{dt} = \frac{1}{\la^2(t)},
		\quad \frac{y}{r} = \frac{1}{\la(t)},
	\end{equation}
	which actually defines a renormalized time $s$ going from certain starting time $s_0$, eventually to $+\infty$, as the original time $t$ runs from $0$ to the lifespan $T$ of $u$. Accordingly, for any given function $v(t,r)$, we use the convention
	\begin{equation} \label{eq: v lambda}
	v_\la(t,r) = v\Big(t,\frac{r}{\la}\Big) = v(s,y),
	\end{equation}
	which leads to
	\begin{equation} \notag
		\pr_r v_{\la} =\frac{1}{\la}(\pr_r v)_{\la}, \quad
		\Delta v_{\la} =\frac{1}{\la^2}(\Delta v)_{\la}.
	\end{equation}
	Also, for the transformation $\SSS$, there holds
	\begin{equation} \notag
		\SSS v =e^{\vart \R} v_{\la}, \quad
		\SSS (v_1 \wedge v_2) = \SSS v_1\wedge\SSS v_2.
	\end{equation}
	Invoking these with \eqref{groud state eq},
	we compute each term in \eqref{eq: LL} for solution \eqref{eq: sol form}:
	\begin{equation} \label{eq: hv eq}
	\lf\{\begin{aligned}
	& \,\,u_t = \prt (\SSS Q) + \prt (\SSS \hv), \\
	& \,\ao u\wedge\Delta u = \frac{\ao}{\la^2}
		\SSS\Big[Q\wedge(\Delta \hv + |\nabla Q|^2 \hv) + \hv\wedge\Delta\hv \Big], \\
	& \hs\hs-\hs\hs \bo u\wedge(u\wedge\Delta u)
	\hs=\hs-\frac{\bo}{\la^2}
			\SSS \Big\{(Q+\hv)\wedge\Big[Q\wedge(\Delta \hv + |\nabla Q|^2 \hv) + \hv\wedge\Delta\hv\Big]\Big\}\\
	& \phantom{-\hs\hs\bo u\wedge(u\wedge}
		= \frac{\bo}{\la^2} \SSS\Big\{
			\big(\er\cdot(\Delta\hv+|\nabla Q|^2\hv)\big)\er
			+ \big(\etau\cdot(\Delta\hv+|\nabla Q|^2\hv)\big)\etau \\
	& \phantom{\hs\hs\bo u\wedge(u\wedge\Delta u)}\quad\quad\quad
		- \hv\wedge\big(Q\wedge(\Delta\hv+|\nabla Q|^2\hv)\big)
			- (Q+\hv)\wedge(\hv\wedge\Delta\hv) \Big\}.
	\end{aligned}\rg.
	\end{equation}

	Considering any given function $v$, from \eqref{eq: var sub}, \eqref{eq: v lambda} and the definition of $\SSS$, we have
	\begin{equation} \notag
		\prt(v_{\la}) =\frac{1}{\la^2}\Big(\prs v-\frac{\la_s}{\la}\varl v\Big)_{\la}, \quad
		\prt(\SSS v) =\frac{1}{\la^2}\SSS\Big(\prs v +\vart_s\R v -\frac{\la_s}{\la}\varl v\Big).
	\end{equation}
	In particular, by Lemma \ref{le: frenet basis}, we have
	\begin{equation} \label{eq: partial t basis}
	\begin{aligned}
		\prt (\SSS \er) & = \frac{1}{\la^2} \SSS\Big(\vart_s Z\etau + \frac{\la_s}{\la}\varl\phi\phsh Q\Big), \\
		\prt (\SSS \etau) & = \frac{1}{\la^2} \SSS\Big(-\vart_s Z\er -\vart_s \varl\phi\phsh Q \Big), \\
		\prt (\SSS Q) & = \frac{1}{\la^2} \SSS\Big(\vart_s\varl\phi \phsh\etau - \frac{\la_s}{\la}\varl\phi\phsh\er\Big).
	\end{aligned}
	\end{equation}
	Thus the time derivative of $\SSS \hv$ is
	\begin{equation} \label{eq: partial t Sv}
	\begin{aligned}
		\prt(\SSS\hv) 
		& = \prt(\h\alpha_{\la})\SSS\er +\prt(\h\beta_{\la})\SSS\etau +\prt(\h\gamma_{\la})\SSS Q \\
		&\quad +\h\alpha_{\la}\prt(\SSS\er) +\h\beta_{\la}\prt(\SSS\etau) +\h\gamma_{\la}\prt(\SSS Q) \\
		& = \frac{1}{\la^2}\bigg\{\bigg[\prs\h\alpha -\frac{\la_s}{\la}\varl\h\alpha -\vart_s\h\beta Z -\frac{\la_s}{\la}\h\gamma\varl\phi\bigg]_{\la}\SSS\er \\
		& \qquad\quad + \bigg[\prs\h\beta -\frac{\la_s}{\la} \varl\h\beta +\vart_s\h\alpha Z +\vart_s\h\gamma \varl\phi\bigg]_{\la}\SSS\etau\\
		& \qquad\quad + \bigg[\prs\h\gamma -\frac{\la_s}{\la}\varl\h\gamma +\frac{\la_s}{\la}\h\alpha\varl\phi -\vart_s\h\beta\varl\phi\bigg]_{\la}\SSS Q\bigg\}.
	\end{aligned}
	\end{equation}
	Next, the linear terms in the forthcoming $\hv$ equation, as exhibited in \eqref{eq: hv eq}, are
	\begin{equation} \label{eq: linear term}
	\begin{aligned}
		& \ao Q\wedge\big(\Delta \hv +|\nabla Q|^2 \hv\big) \\
		& \qquad + \bo\Big[\Big(\er\cdot(\Delta\hv+|\nabla Q|^2\hv)\Big)\er +\Big(\etau\cdot(\Delta\hv+|\nabla Q|^2\hv)\Big)\etau\Big].
	\end{aligned}
	\end{equation}
	Similarly, by $\hv\wedge\hv=0$, the nonlinear terms in \eqref{eq: hv eq} is reformulated to
	\begin{equation}\label{eq: hv nonlinear}
	\begin{split}
	& \ao\hv\wedge\Delta\hv - \bo\hv\wedge\Big(Q\wedge(\Delta\hv+|\nabla Q|^2\hv)\Big) - \bo(Q+\hv)\wedge(\hv\wedge\Delta\hv) \\
	=&\, \ao\hv\wedge(\Delta\hv+|\nabla Q|^2\hv) -\bo\hv\wedge\Big(Q\wedge(\Delta\hv+|\nabla Q|^2\hv)\Big) \\
	&- \bo(Q+\hv)\wedge\Big(\hv\wedge(\Delta\hv+|\nabla Q|^2\hv)\Big).
	\end{split}
	\end{equation}
	Notice here there are repeated linear patterns on $\hv$, namely $\Delta\hv+|\nabla Q|^2\hv$, in both \eqref{eq: linear term} and \eqref{eq: hv nonlinear}. From the property of the ground state \eqref{groud state eq}, there holds $|\nabla Q|^2 = -\Delta Q/Q = 2(1+Z)^2$. Using the derivation of the Frenet basis \eqref{def: derivation of the Frenet basis}, we compute:
	\begin{equation} \label{eq: part linear}
	\begin{aligned}
		&\quad\, \Delta\hv +|\nabla Q|^2\hv \\
		& = \Delta\h\alpha\phsh\er +2\pr_r\h\alpha\phs\pr_r\er + \h\alpha\phsh\Delta\er +\Delta\h\beta\phsh\etau + 2\pr_r\h\beta\phs\pr_r\etau +\h\beta\phsh\Delta\etau  \\
		&\quad +\Delta\h\gamma\phsh Q +2\pr_r\h\gamma\phsh\pr_r Q +\h\gamma\phsh\Delta Q +2(1+Z)^2\big(\h\alpha\phsh\er + \h\beta\phsh\etau +\h\gamma\phsh Q\big)\\
		& = \bigg\{\hs\Delta\h\alpha\hs+\hs\bigg(2\onez^2\hs\hs-\hs\hs\frac{1}{r^2}\bigg)\h\alpha\hs+\hs2\onez\pr_r\h\gamma\bigg\}\phsh\er +\bigg\{\Delta\h\beta \hs+\hs\bigg(2\onez^2 \hs\hs-\hs\hs\frac{1}{r^2}\bigg)\h\beta\bigg\}\phsh \etau \\
		&\quad + \bigg\{\Delta\h\gamma -2\onez\pr_r\h\alpha -\frac{2Z\onez}{r}\h\alpha \bigg\}\phsh Q.
	\end{aligned}
	\end{equation}
	The first and second components on the RHS of \eqref{eq: part linear} inspire us to define the Hamiltonian/\Sch operator
	\begin{equation} \label{eq: sch operator}
		\Ham = -\Delta + \frac{V(r)}{r^2},
	\end{equation}
	where the potential is
	\begin{equation} \label{eq: V}
		V(r)=\frac{1}{r^2}-2(1+Z)^2=\varl Z+Z^2=\frac{r^4-6r^2+1}{(1+r^2)^2}.
	\end{equation}
	To make the formulas brief, we further introduce the following operators, each of which maps the vector $\hw$ to a scalar function of its components:
	\begin{equation} \label{eq: mHam func}
	\lf\{\begin{aligned}
	\,\mHamo(\hw) =&\, \Ham \h\alpha -2(1+Z) \pr_r\h\gamma, \\
	\mHamt(\hw) =&\, \Ham \h\beta, \\
	\mHamth(\hw) =&\, -\Delta \h\gamma +2(1+Z)\pr_r\h\alpha +\tfrac{2Z(1+Z)}{r} \h\alpha.
	\end{aligned}\rg.
	\end{equation}
	Then \eqref{eq: part linear} is actually
	\begin{equation}
	\Delta \hv + |\nabla Q|^2 \hv = -\Big(\mHamo(\hw)\,\er + \mHamt(\hw)\,\etau + \mHamth(\hw)\,Q\Big), \notag
	\end{equation}
	Using \eqref{eq: part linear}, \eqref{eq: sch operator} to treat \eqref{eq: linear term}, we have
	\begin{equation} \label{eq: linear v}
	\begin{aligned}
	& \ao Q\wedge(\Delta \hv + |\nabla Q|^2 \hv) \\
	& \qquad + \bo \Big[\Big(\er\cdot(\Delta\hv+|\nabla Q|^2\hv)\Big)\er
	+\Big(\etau\cdot(\Delta\hv+|\nabla Q|^2\hv)\Big)\etau \Big] \\
	=&\,\Big(\ao\mHamo(\hw) -\bo\mHamo(\hw)\Big)\,\er + \Big(-\ao \mHamo(\hw) -\bo\mHamo(\hw)\Big)\,\etau.
	\end{aligned}
	\end{equation}

	Finally the nonlinear terms \eqref{eq: hv nonlinear} are
	\begin{equation} \notag
	\begin{aligned}
	&\quad \ao\hv\wedge\Delta\hv - \bo\hv\wedge\Big(Q\wedge(\Delta\hv+|\nabla Q|^2\hv)\Big) - \bo(Q+\hv)\wedge(\hv\wedge\Delta\hv) \\
	& =\bigg\{\ao\Big(\h\gamma\mHamt(\hw)-\h\beta\mHamth(\hw)\Big) \\
		&\quad +\bo\bigg( -\big(\h\beta^2+\h\gamma^2+2\h\gamma\big)\mHamo(\hw)+\h\alpha\h\beta\mHamt(\hw) + \h\alpha(1+\h\gamma)\mHamth(\hw) \bigg)\bigg\}\,\er
	\end{aligned}
	\end{equation}
	\begin{equation} \label{eq: hv nonlinear term}
	\begin{aligned}
	& +\bigg\{\ao\Big(\h\alpha\mHamth(\hw)-\h\gamma\mHamo(\hw)\Big)\\
		&\quad +\bo\bigg( -\big(\h\alpha^2+\h\gamma^2+2\h\gamma\big)\mHamt(\hw)+\h\alpha\h\beta\mHamo(\hw) + \h\beta(1+\h\gamma)\mHamth(\hw) \bigg)\bigg\}\, \etau\\
	& +\bigg\{\ao\Big(\h\beta\mHamo(\hw)-\h\alpha\mHamt(\hw)\Big) \\
		&\quad +\bo\bigg( -\big(\h\alpha^2+\h\beta^2\big)\mHamth(\hw)+\h\alpha(1+\h\gamma)\mHamo(\hw) + \h\beta(1+\h\gamma)\mHamt(\hw) \bigg)\bigg\}\, Q.
	\end{aligned}
	\end{equation}
	Inserting \eqref{eq: partial t basis}, \eqref{eq: partial t Sv}, \eqref{eq: linear v}, \eqref{eq: hv nonlinear term} into \eqref{eq: hv eq},
	and projecting them onto $\big\{\SSS\er,\SSS\etau,\SSS Q\big\}$, we obtain the component equations of $\hv$:
	\begin{equation} \label{eq: compon 1}
	\begin{aligned}
	\prs \h\alpha -\frac{\la_s}{\la}&\varl\h\alpha
	=\, \ao(1+\h\gamma)\mHamt(\hw) -\bo(1+\h\gamma)^2\mHamo(\hw)
	+\frac{\la_s}{\la}(1+\h\gamma)\varl\phi
	+\vart_s\h\beta Z \\
	& -\ao\h\beta\mHamth(\hw) +\bo\h\alpha(1+\h\gamma)\mHamth(\hw)
	-\bo\h\beta^2\mHamo(\hw) +\bo\h\alpha\h\beta\mHamt(\hw),
	\end{aligned}
	\end{equation}
	\begin{equation} \label{eq: compon 2}
	\begin{aligned}
	\prs \h\beta - \frac{\la_s}{\la}&\varl\h\beta
	=\, -\ao(1+\h\gamma)\mHamo(\hw) -\bo(1+\h\gamma)^2\mHamt(\hw)
	-\vart_s\h\alpha Z -\vart_s(1+\gamma)\varl\phi \\
	& +\ao\h\alpha\mHamth(\hw) +\bo\h\beta(1+\h\gamma) \mHamth(\hw)
	-\bo\h\alpha^2 \mHamt(\hw) +\bo\h\alpha\h\beta \mHamo(\hw),
	\end{aligned}
	\end{equation}
	\begin{equation} \label{eq: compon 3}
	\begin{aligned}
	\prs \h\gamma &- \frac{\la_s}{\la}\varl\h\gamma
	=\, \ao\h\beta\mHamo(\hw) -\ao\h\alpha\mHamt(\hw)
	-\frac{\la_s}{\la}\h\alpha\varl\phi
	+\vart_s\h\beta\varl\phi \\
	& +\bo\h\alpha(1+\h\gamma)\mHamo(\hw) +\bo\h\beta(1+\h\gamma)\mHamt(\hw)
	 -\bo\h\alpha^2\mHamth(\hw) -\bo\h\beta^2\mHamth(\hw).
	\end{aligned}
	\end{equation}
	These component equations, especially \eqref{eq: compon 1}, \eqref{eq: compon 2}, reveal the structure of a combination of the quasilinear \Sch equation and the quasilinear heat equation.

	\begin{remark}
	{\rm (i)} The linear term on the RHS of \eqref{eq: compon 3} is comparatively small owing to the presence of $\la_s/\la$ and $\vart_s$, which eventually implies the smallness of $\h\gamma$. This is actually reasonable in view of the constraint \eqref{eq: component constraint}, and it will be confirmed in Section~\ref{S: construction of the ap sol}.
	{\rm (ii)} The spatial variable used here is $y=r/\la$ instead of $r$, due to the action of $\SSS$ in \eqref{eq: hv eq}. Accordingly, modifications should be made by replacing $r$ with $y$ in the definitions of \eqref{eq: sch operator}, \eqref{eq: mHam func}, for example, $\Ham_{y} = -\Delta_y + V(y)/y^2$. However, since the potential risk of ambiguity is low, we will still use the original notations.
	\end{remark}

	An essential feature of our analysis is to keep track of the geometric structure of \eqref{eq: LL}, so we rewrite the system into vectorial form. We set $\ez = [0,0,1]^{T}$, and define the rotation related to $\hw$:
	\begin{equation} \label{eq: hJ}
		\hJ = (\ez + \hw)\wedge,
	\end{equation}
	where we note that $\ez\wedge=\R$ is just the usual rotation \eqref{eq: rotation}. The operators defined in \eqref{eq: mHam func} actually yield the vectorial Hamiltonian/ \Sch operator:
	\begin{equation} \label{eq: mHam}
		\mHam\hw =
		\begin{bmatrix} \,\mHamo(\hw)\, \\ \mHamt(\hw) \\ \mHamth(\hw) \end{bmatrix}
		= \begin{bmatrix} \Ham\h\alpha \\ \Ham\h\beta \\ -\hs\hs\Delta \h\gamma \end{bmatrix}
		+2\big(1+Z\big)\hs\begin{bmatrix} \py \h\gamma \\ 0 \\ \py \h\alpha + \frac{Z}{y}\h\alpha \end{bmatrix}.
	\end{equation}
	Moreover, we define the following vector
	\begin{equation} \label{eq: p vec}
		\p =\begin{bmatrix}\vart_s\\ \la_s/\la\\ 0\end{bmatrix}.
	\end{equation}
	Therefore \eqref{eq: compon 1}--\eqref{eq: compon 3} can be rewritten as the following vectorial form
	\begin{equation} \label{eq: vec hw equation}
		\prs\hw -\frac{\la_s}{\la}\varl\hw + \vart_s Z\R\hw +\hJ\Big(\,\ao\mHam\hw -\bo\hJ\mHam\hw +\p\,\varl\phi\,\Big) = 0.
	\end{equation}

	\subsection{The linearized Hamiltonian}
	The linearized Hamiltonian/\Sch operator $\Ham$ \eqref{eq: sch operator} and $\mHam$ \eqref{eq: mHam} naturally arise when computing the Landau-Lifschtiz flow near $\SSS Q$ \eqref{eq: part linear}. In this subsection, we collect some of their properties, which will be of essential importance in the derivation of Laypnouv mononicity. 
	
	The linearized Hamiltonian $\Ham$ admits the following factorization:
	\begin{equation} \label{eq: Ham factori}
		\Ham=\As\A,
	\end{equation}
	with
	\begin{equation} \notag
		\A=-\py+\frac{Z}{y},
		\quad \As=\py+\frac{1+Z}{y},
	\end{equation}
	where $\As$ is the adjoint of $\A$. Given any radially symmetric function $f(y)$, these operators can be reformulated as
	\begin{equation} \label{eq: A reformulation}
		\A f = - \varl\phi \,\py\bigg(\frac{f}{\varl\phi}\bigg), \quad
		\As f = \frac{1}{y\varl\phi} \,\py\big(y\varl\phi f\big).
	\end{equation}
	Thus, their kernels on $\RR_{+}^{\ast}$ are explicit:
	\begin{equation} \notag
	\lf\{\begin{aligned}
		&\A f = 0 & \iff \quad & f\in\spann(\varl\phi), \\
		&\As f = 0 & \iff \quad &  f\in\spann\bigg(\frac{1}{y\varl\phi}\bigg).
	\end{aligned}\rg.
	\end{equation}
	Hence considering $\Ham f = 0$ on $\RR_{+}^{\ast}$, we deduce either $f= C \varl\phi$ or $\A f = \frac{C}{y\varl\phi}$ for any constant $C\in \RR$. In other words, we have
	\begin{equation} \notag
		\Ham f=0 \quad\iff\quad f\in\spann(\varl\phi,\vg),
	\end{equation}
	where $\varl\phi$ is a resonance of $\Ham$ at the origin induced by energy critical scaling invariance, and
	\begin{equation} \notag
		\vg(y)=\varl\phi(y)\int_1^y\frac{dx}{x \lf(\varl\phi(x)\rg)^2}
		=\frac{1}{4(1+y^2)}\Big(y^3+4y\log y-\frac{1}{y}\Big).
	\end{equation}
	The asymptotics for $\varl\phi, \vg$ are
	\begin{equation} \label{eq: lambda phi}
		\varl\phi(y)=\begin{cases}
			2y+O(y^3) &\mbox{as}\;\; y\to0,\\
			\frac{2}{y}+ O\Big(\frac{1}{y^3}\Big)
			&\mbox{as}\;\; y\to+\infty,
		\end{cases}
	\end{equation}
	and
	\begin{equation} \label{eq: gamma}
		\vg(y)=\begin{cases}
			-\frac{1}{4y}+O(y\log y) &\mbox{as}\;\; y\to0,\\
			\frac{y}{4}+ O\Big(\frac{\log y}{y}\Big)
			&\mbox{as}\;\; y\to+\infty.
		\end{cases}
	\end{equation}
	Owing to the Wronskian of $\varl\phi, \vg$ given by
	\begin{equation} \notag
		\varl\phi'\vg -\vg'\varl\phi=-\frac{1}{y},
	\end{equation}
	the variation of constants formula yields a regular solution to the inhomogeneous equation
	\begin{equation} \notag
		\Ham u= f,
	\end{equation}
	which is given by
	\begin{equation} \label{eq: inverse of Ham}
		u(y) = \varl\phi(y)\int_0^y f(x)\vg(x) xdx - \vg(y)\int_0^y f(x)\varl\phi(x) xdx.
	\end{equation}
	In particular, if $f = \varl\phi$, the solution is 
	\begin{equation} \label{eq: T1}
		T_1(y) =\frac{(1-y^4)\log(1+y^2)+2y^4-y^2-4y^2\int_{1}^{y}\frac{\log(1+x^2)}{x}dx}{2y(1+y^2)},
	\end{equation}
	with the asymptotics
	\begin{equation} \label{eq: T1 asymptotics}
		T_1(y)=\begin{cases}
		-\frac{y^3}{4}+O(y^5) &\mbox{as}\;\; y\to0, \\
		-y\log y+y+O\Big(\frac{(\log y)^2}{y}\Big)
		&\mbox{as}\;\; y\to+\infty.
	\end{cases}
	\end{equation}
	To deal with the evolution of the forthcoming radiation \eqref{eq: w}, we extract the dominating part of the vectorial Hamiltonian $\mHam$, which is defined by
	\begin{equation}\label{eq: mHamp}
		\mHamp\begin{bmatrix}
			\alpha\\ \beta\\ \gamma
		\end{bmatrix}
		=\begin{bmatrix}
			\Ham\alpha\\ \Ham\beta\\ 0
		\end{bmatrix}.
	\end{equation}
	The given factorization \eqref{eq: Ham factori} implies the corresponding factorization of $\mHamp$, which is $\mHamp = \mAs\mA$ with
	\begin{equation} \label{eq: mHamp factori}
		\mA\begin{bmatrix} \alpha\\ \beta\\ \gamma\end{bmatrix}
		=\begin{bmatrix} \A\alpha\\ \A\beta\\ 0 \end{bmatrix},
		\quad \mAs\begin{bmatrix} \alpha\\ \beta\\ \gamma\end{bmatrix}
		=\begin{bmatrix} \As\alpha\\ \As\beta\\ 0 \end{bmatrix}.
	\end{equation}
	Moreover, there holds the relation of $\mHamp$ with \eqref{eq: rotation}:
	\begin{equation} \label{eq: relation mHamp}
		\R\mHamp=\mHamp\R, \quad\R\mHam\R=-\mHamp.
	\end{equation}
	
	\section{Construction of the approximate solution}
	\label{S: construction of the ap sol}
	This section is devoted to the construction of the approximate self-similar solution to the vectorial equation \eqref{eq: vec hw equation} with controllable growth in spacial variable $y$. The blowup construction heavily relys on the behavior of the modulation parameters $\la, \vart$, and also their derivatives $\la_s/\la, \vart_s$, which emerged in \eqref{eq: vec hw equation}. To study them, we introduce two more parameters, as well as their anticipated dynamics
	\begin{equation} \label{eq: mod parameter ansatz}
		a\approx-\vart_s, \quad
		b\approx-\frac{\la_s}{\la}, \quad
		a_s\approx 0, \quad b_s\approx -(b^2+a^2).
	\end{equation}
	This approximation will be justified in subsection \eqref{SS: modulation equations}. The layout of our approximate solution resembles a series expansion of $a, b$. More precisely, we have the following lemma.
	
	\begin{lemma}[The approximate solution] \label{le: ap solution}
	Let $M$ be a large enough universal constant, then there exists a universal small constant $b^\ast=b^\ast(M)>0$, such that for any $\CCC^1$ maps
	\begin{equation} \label{eq: a b c1 map}
		a, b:\,[s_0,+\infty)\mapsto(-b^\ast,b^\ast)
	\end{equation}
	with the prior bound
	\begin{equation} \label{eq: a prior bound}
		|a| \lesssim \frac{b}{\lgba},
	\end{equation}
	and the varying scale $B_0, B_1$ given by \eqref{eq: scales}, the following holds: There exist smooth radially symmetric profiles $\varp_{1,0}(y), \varp_{0,1}(y), \varp_{i,j}(b,y)$ and $S_{0,2}(y)\sim |T_1(y)|^2$ admiting the asymptotics
	\begin{equation} \label{eq: varp asymptotics}
		\varp_{1,0},\,\varp_{0,1} =\begin{cases}
		 	O(y^3) &\mbox{{\rm for}}\;\; y\leq 1, \\
		 	O(y\log y) &\mbox{{\rm for}}\;\; 1\leq y\leq 2B_1,
		\end{cases}
	\end{equation}
	and 
	\begin{equation} \label{eq: varp asymptotics 2}
		|\varp_{i,j}| \les \begin{cases}
			y^{2(i+j)+1} &\mbox{{\rm for}}\;\; y\leq 1, \\ 
			\dfrac{1+y^{2(i+j)-3}}{b\lgba}
			&\mbox{{\rm for}}\;\; 1\leq y\leq 2B_1,
		\end{cases}
	\end{equation}
	for non-negative integers $i,j$ in the range of $2\leq i+j\leq 3$, such that
	\begin{equation} \label{eq: wz form}
		\wz= a\phsh\varp_{1,0} +b\phsh\varp_{0,1} +\hs\sum_{2\leq i+j\leq 3}a^{i}b^{j}\varp_{i,j} +b^2 S_{0,2}
	\end{equation}
	is an approximate self-similar solution to \eqref{eq: vec hw equation} in the following sense. Let
	\begin{equation} \label{eq: ap error expand}
		\mmod +\errz
		= \prs\wz -\frac{\la_s}{\la}\varl\wz +\vart_s Z\R\wz +\hJ\Big(\,\ao\mHam\wz -\bo\hJ\mHam\wz +\p\phsh\varl\phi\,\Big),
	\end{equation}
	where $\mmod$ is the modulation vector given by {\rm (}The summations are for $2\leq i+j\leq 3${\rm )}
	\begin{equation} \label{eq: mod vec}
	\begin{aligned}
		\mmod(t) 
		&= a_s\Big(\varp_{1,0}+\sum ia^{i-1}b^{j}\varp_{i,j}\Big) \\
		&\quad +\big(b_s+b^2 +a^2\big)\Big( \varp_{0,1} +2b S_{0,2} +\sum ja^{i}b^{j-1}\varp_{i,j} +\sum a^{i}b^{j}\prb \varp_{i,j} \Big) \\
		&\quad -\Big(\frac{\la_s}{\la}+b\Big)\Big(\varl\phi \big(\ex+O(\wz)\big)+\varl\wz\Big) \\
		&\quad +\big(\vart_s+a\big)\Big(\varl\phi\big(\ey+O(\wz)\big)+ZR\wz\Big),
	\end{aligned}
	\end{equation}
	and $\errz$ is the error satisfying the estimates:\\
	{\rm (i)} Weighted bounds:
	\begin{align}
		& \int_{y\leq 2B_1}\frac{|\pr_y^i\errz^{\first}|^2+|\pr_y^i\errz^{\second}|^2}{y^{6-2i}} 
		\les\bflog, \quad 0\leq i\leq 3,
		\label{eq: weighted bound 1}\\
		& \int_{y\leq 2B_1}\frac{|\pr_y^i\errz^{\third}|^2}{y^{8-2i}}\les\frac{b^6}{\lgba^2}, \quad 0\leq i\leq 4,
		\label{eq: weighted bound 2}\\
		& \int_{y\leq 2B_1}|\Ham\errz^{\first}|^2 +|\Ham\errz^{\second}|^2
		 \les b^4\lgba^2,
		\label{eq: weighted bound 3}\\
		& \int_{y\leq 2B_1}\frac{|\pr_y^i\Ham\errz^{\first}|^2 +|\pr_y^i\Ham\errz^{\second}|^2}{y^{2-2i}} 
		\les\bflog, \quad 0\leq i\leq 1,
		\label{eq: weighted bound 4}\\
		& \int_{y\leq 2B_1}|\A\Ham\errz^{\sups{1}}|^2 +|\A\Ham\errz^{\sups{2}}|^2 \les b^5,
		\label{eq: weighted bound 5}\\
		& \int_{y\leq 2B_1}|\Ham^2\errz^{\sups{1}}|^2 +|\Ham^2\errz^{\sups{2}}|^2 \les \frac{b^6}{\lgba^2}.
		\label{eq: weighted bound 6}
	\end{align}
	{\rm (ii)} Flux computation: Let $\varp_M$ be defined by \eqref{eq: varp M}, then
	\begin{equation} \label{eq: flux computation}
	\lf\{\begin{aligned}
	& \frac{(\Ham\errz^{\first},\varp_M)}{(\varl\phi,\varp_M)} = \frac{2(\ao ab -\bo b^2)}{(\ao^2 +\bo^2)\lgba}\bigg(1+O\Big(\frac{1}{\lgba}\Big)\bigg), \\
	& \frac{(\Ham\errz^{\second},\varp_M)}{(\varl\phi,\varp_M)} = \frac{2(\ao b^2 +\bo ab)}{(\ao^2 +\bo^2)\lgba}\bigg(1+O\Big(\frac{1}{\lgba}\Big)\bigg).
	\end{aligned}\rg.
	\end{equation}
	\end{lemma}
	
	\begin{proof}{Lemma \ref{le: ap solution}}
	We proceed the proof by first constructing the profiles $\varp_{1,0}, \varp_{0,1}, S_{0,2}$, deriving the modulation vector $\mmod$, and then eliminating the growing tails induced by the $\varl T_1-T_1$, constructing the higher order profiles $\varp_{i,j}$ ($2\leq i+j\leq 3$), and finally estimating the error $\errzt$.
	
	\medskip
	\step{1}{Construction of $\varp_{1,0}$, $\varp_{0,1}$} Setting $\hw=\wz$ in \eqref{eq: vec hw equation} gives the RHS of \eqref{eq: ap error expand}.  Assuming the smallness of $a,b$, we see the expression
	\begin{equation} \label{eq: main part in wz eq}
		\hJ\Big(\ao\mHam\wz-\bo\hJ\mHam\wz+\p\phsh\varl\phi\Big)
	\end{equation}
	is the major error to be eliminated in the first place. From the definition of \eqref{eq: hJ}, \eqref{eq: p vec}, we have
	\begin{equation} \label{eq: step1 hv}
	\begin{aligned}
		&\hJ\Big(\ao\mHam\wz -\bo\hJ\mHam\wz +\p\phsh\varl\phi\Big) = \hJ\Bigg\{\hs(\ao-\bo\R)\phsh\mHam\wz -\bo\wz \wedge \mHam\wz -\begin{bmatrix}\,a\,\\ b\\ 0\end{bmatrix}\hs\varl\phi\Bigg\} \\
		&\qquad\qquad\qquad\qquad -\Big(\frac{\lambda_s}{\lambda}+b\Big)\varLambda\phi\big(\ex+O(\wz)\big) +\big(\varTheta_s+a\big)\varLambda\phi\big(\ey+O(\wz)\big).
	\end{aligned}	
	\end{equation}
	We decompose $\wz$ by
	\begin{equation} \label{eq: wz decomposition}
		\wz=\wz^{1} +\wz^{2} +b^2\phsh S_{0,2}, 
		\quad\mbox{with}\quad
		\lf\{\begin{aligned}
			& \wz^{1}=a\phsh\varp_{1,0} +b\phsh\varp_{0,1}, \\
			& \wz^{2}=\sum_{2\leq i+j\leq 3} a^{i}b^{j} \varp_{i,j},
		\end{aligned}\rg.
	\end{equation}
	and further assume the profiles are of the form
	\begin{equation} \label{eq: profile form}
		S_{0,2} = S_{0,2}^{\third}\phsh e_z, \quad
		\varp_{i,j}^{\third} = 0 
		\;\,\mbox{for}\;\, 1\leq i+j\leq 3.
	\end{equation}
	Thus the leading profiles of $\wz$ should be $\wz^1+b^2S_{0,2}$, where $\wz^1$ is mainly responsible for the first and second components, while $b^2S_{0,2}$ for the third one. From these we compute:
	\begin{equation} \label{eq: ao boR term in wz}
	\begin{aligned}
		\quad (\ao-\bo\R)\phsh\mHam\wz
		&= (\ao-\bo\R)\begin{bmatrix}
			\,\mHamo(\wz^{1})\,\\ \mHamt(\wz^{1})\\ 0
		\end{bmatrix} +(\ao-\bo\R)\begin{bmatrix}
			\,\mHamo(\wz^{2})\,\\ \mHamt(\wz^{2})\\ 0
		\end{bmatrix} \\
		&\quad -2b^2\onez\phs\py S_{0,2}^{\third}\phs\big(\ao\ex-\bo\ey\big) \\
		&\quad -\ao b^2\Delta S_{0,2}^{\third}\,\ez +2\ao\onez\sum_{i=1,2}\pyzy(\wz^{i})^{\first}\phsh\ez.
	\end{aligned}
	\end{equation}
	Injecting \eqref{eq: ao boR term in wz} into \eqref{eq: step1 hv}, we let
	\begin{equation} \label{eq: varp function equation}
		(\ao-\bo\R)\begin{bmatrix}
			\,\mHamo(\wz^{1})\,\\ \mHamt(\wz^{1})\\ 0
		\end{bmatrix} -\begin{bmatrix}
			\,a\,\\ b\\ 0
		\end{bmatrix} \varLambda\phi = 0,
	\end{equation}
	which yields the equations for the first-order profiles $\varp_{1,0}, \varp_{0,1}$:
	\begin{equation} \notag
		(\ao-\bo\R)\mHamp\varp_{1,0}=\varl\phi\phsh\ex, \quad
		(\ao-\bo\R)\mHamp\varp_{0,1}=\varl\phi\phsh\ey.
	\end{equation}
	Using \eqref{eq: T1}, we have the following smooth solution
	\begin{equation} \label{eq: varp function}
		\varp_{1,0} = \frac{1}{\ao^2+\bo^2}\begin{bmatrix}\,\ao\, \\ \,\bo\, \\ 0 \end{bmatrix} T_1, \quad 
		\varp_{0,1} = \frac{1}{\ao^2+\bo^2}\begin{bmatrix}-\bo \\ \ao \\ 0 \end{bmatrix} T_1.
	\end{equation}
	Their asymptotics \eqref{eq: varp asymptotics} follows directly from \eqref{eq: T1 asymptotics}. Moreover, there holds
	\begin{equation} \label{eq: varp relation}
		\varp_{0,1}=\R\phsh\varp_{1,0}, \quad 
		\varp_{1,0}=-\R\phsh\varp_{0,1}.
	\end{equation}
	
	\medskip
	\step{2}{Construction of $S_{0,2}$} In view of the constraint \eqref{eq: component constraint}, the anticipated approximate solution should meet 
	\begin{equation} \label{eq: component constraint wz}
		(\wz^{\first})^2 +(\wz^{\second})^2 +(1+\wz^{\third})^2 \approx 0, 
	\end{equation}
	with sufficient precision. From the decomposition \eqref{eq: wz decomposition}, we see  
	\begin{equation} \notag
	\lf\{\begin{aligned}
		& \wz^{\sups{i}} \approx (\wz^{1})^{\sups{i}} = a\phsh\varp_{1,0}^{\sups{i}} +b\phsh\varp_{0,1}^{\sups{i}},
		\quad\mbox{for}\quad i=1,2, \\
		& \wz^{\sups{3}} = b^2S_{0,2}.
	\end{aligned}\rg.
	\end{equation}
	In consideration of \eqref{eq: a prior bound}, the profile $a\phsh\varp_{1,0}$ is negligible compared with $b\phsh\varp_{0,1}$, so that 
	\begin{equation} \notag
		(\wz^{\first})^2 +(\wz^{\second})^2 \approx 
		b^2\Big[ \big(\varp_{0,1}^{\first}\big)^2+\big(\varp_{0,1}^{\second}\big)^2 \Big].
	\end{equation} 
	In order to be compatible with \eqref{eq: component constraint wz},  we let 
	\begin{equation} \label{eq: S02 function} 
		S_{0,2}^{\third} =-\frac{1}{2(\ao^2+\bo^2)}(T_1)^2,
	\end{equation}
	which admits the following asymptotics from \eqref{eq: T1 asymptotics}:
	\begin{equation} \label{eq: S asymptotics}
		S_{0,2}(y) =
		\begin{cases}
			O(y^6) &\mbox{for}\;\; y\leq 1, \\
		 	O\big(y^2(\log y)\big) &\mbox{for}\;\; 1\leq y\leq 2B_1.
		\end{cases}
	\end{equation}
	
	\medskip
	\step{3}{Derivation of the modulaltion vector} In view of \eqref{eq: ap error expand}, we have
	\begin{equation} \label{eq: step2 hv}
	\begin{aligned}
		&\quad \prs\wz -\frac{\la_s}{\la}\varl\wz +\vart Z\R\wz \\
		&= a_s\varp_{1,0} +b_s\varp_{0,1} +\sum\prs(a^{i}b^{j})\varp_{i,j} +\sum a^{i}b^{j}b_s\prb \varp_{i,j} +2bb_sS_{0,2} \\
		&\quad -\Big(\frac{\la_s}{\la}+b\Big)\varl\wz +ab\varl\varp_{1,0} +b^2\varl\varp_{0,1} +\sum a^{i}b^{j+1}\varl \varp_{i,j} +b^3\varl S_{0,2} \\
		&\quad +\big(\vart_s+a\big)Z\R\wz -a^2Z\R\phsh\varp_{1,0} -abZ\R\phsh\varp_{0,1} -\sum a^{i+1}b^{j}Z\R\phsh \varp_{i,j} -ab^2Z\R\varp_{i,j} \\
		&= -\Big(\frac{\la_s}{\la}+b\Big)\varl\wz +\big(\vart_s+a\big)Z\R\wz +a_s\Big(\varp_{1,0} +\sum ia^{i-1}b^{j}\varp_{i,j}\Big) \\
		&\quad +\big(\prs b +a^2 +b^2\big) \Big(\varp_{0,1} +2b\,S_{0,2} +\sum ja^{i}b^{j-1}\varp_{i,j} +\sum a^{i}b^{j}\prb \varp_{i,j} \Big) \\
		&\quad +ab\big(\varl\varp_{1,0}-\varp_{1,0}\big)+b^2\big(\varl\varp_{0,1}-\varp_{0,1}\big) +b^3\big(\varl S_{0,2}-2S_{0,2}\big) \\
		&\quad +ab(1+Z)\phsh\varp_{1,0} -a^2(1+Z)\phsh\varp_{0,1} -2a^2b\,S_{0,2} \\
		&\quad +\sum a^{i}b^{j}\Big(b\varl \varp_{i,j} -aZ\R \varp_{i,j}\Big)  -(a^2+b^2)\sum \Big(ja^{i}b^{j-1}\varp_{i,j} +a^{i}b^{j}\prb \varp_{i,j}\Big).
	\end{aligned}
	\end{equation}
	where \eqref{eq: varp relation} has been applied. The above summations run over the indices $i,j$ for $2\leq i+j\leq 3$. By collecting the terms with coefficients $(\vart_s+a), (\la_s/\la +b), a_s, (b_s+a^2+b^2)$ in \eqref{eq: step1 hv}, \eqref{eq: step2 hv}, we obtain the modulation vector \eqref{eq: mod vec}, and also the error
	\begin{equation} \label{eq: errz expand}
	\begin{aligned}
		\errz &= \hJ\phs\Bigg\{ (\ao-\bo\R)\hs\begin{bmatrix} 
			\,\mHamo(\wz^2)\,\\ \mHamt(\wz^2)\\ 0 
		\end{bmatrix} -2b^2\onez\phs\py S_{0,2}^{\third}\phs\big(\ao\ex-\bo\ey\big) \\
		&\qquad\quad -\bo\wz\hs\wedge\hs\mHam\wz -\ao b^2\Delta S_{0,2}^{\third}\,\ez +2\ao\onez\sum_{i=1,2}\pyzy(\wz^{i})^{\first}\phsh\ez \Bigg\} \\
		&\quad +ab\big(\varl\varp_{1,0}-\varp_{1,0}\big)+b^2\big(\varl\varp_{0,1}-\varp_{0,1}\big) +b^3\big(\varl S_{0,2}-2S_{0,2}\big) \\
		&\quad +ab(1+Z)\phsh\varp_{1,0} -a^2(1+Z)\phsh\varp_{0,1} -2a^2b\,S_{0,2} \\
		&\quad +\hspace{-.7em}\sum_{2\leq i+j\leq 3}\hspace{-.6em}a^{i}b^{j}\Big(b\varl \varp_{i,j} \hs-\hs aZ\R \varp_{i,j}\Big) -(a^2+b^2)\hspace{-.7em}\sum_{2\leq i+j\leq 3}\hspace{-.5em}\Big(ja^{i}b^{j-1}\varp_{i,j} +a^{i}b^{j}\pr_b \varp_{i,j}\Big).
	\end{aligned}	
	\end{equation}
	
	\medskip
	\step{4}{The $\sgb$ function} To create cancellations on the forthcoming growing tails of the second order profiles (see \eqref{eq: lambda T1 T1}, \eqref{eq: lambda T1 Sigma b outside}), we introduce the radiation $\sgb$ defined by
	\begin{equation} \label{eq: sigma b}
		\Ham \sgb = c_b\chi_{\frac{B_0}{4}}\varl\phi -d_b\Ham\Big[(1-\chi_{3B_0})\varl\phi\Big],
	\end{equation}
	where the coefficients are
	\begin{equation} \label{eq: cb}
	\lf\{\begin{aligned}
		& c_b = \frac{4}{\int\chi_{\frac{B_0}{4}}(\varl\phi)^2xdx}
		= \frac{2}{\lgba} \bigg(1+O\Big(\frac{1}{\lgba}\Big)\bigg), \\
		& d_b = c_b\int\chi_{\frac{B_0}{4}}\varl\phi\vg xdx = O\bigg(\frac{1}{b \lgba}\bigg).
	\end{aligned}\rg.
	\end{equation}
	Applying \eqref{eq: inverse of Ham}, we see
	\begin{equation} \notag
	\begin{aligned}
		\sgb(y) = -\vg\int_0^y c_b\chi_{\frac{B_0}{4}}(\varl\phi)^2 +\varl\phi\int_0^y c_b\chi_{\frac{B_0}{4}}\varl\phi\vg -d_b\big(1-\chi_{3B_0}\big)\varl\phi,
	\end{aligned}
	\end{equation}
	which implies
	\begin{equation} \label{eq: sigma b asymptotics}
		\sgb =\begin{cases}
			c_b T_1  &\mbox{for}\;\; y\leq \frac{B_0}{4},\\
			-4\varg &\mbox{for}\;\; y\geq 6B_0.
		\end{cases}
	\end{equation}
	From \eqref{eq: lambda phi}, \eqref{eq: gamma}, there holds the asymptotics: for $y\geq 6B_0$,
	\begin{equation} \notag
		\sgb(y) = -y+O\bigg(\frac{\log y}{y}\bigg).
	\end{equation}
	while for $1\leq y\leq 6B_0$,
	\begin{equation} \notag
	\begin{aligned}
		\sgb(y) & = -c_b\bigg(\frac{y}{4}+O\bigg(\frac{\log y}{y}\bigg)\bigg)\int_0^y\chi_{\frac{B_0}{4}}(\varl\phi)^2xdx \\
		&\quad + c_b\,\varl\phi\int_1^yO(x)dx +O\bigg(\frac{1_{y\geq 3B_0}}{y b\lgba}\bigg) \\
		& = -y \frac{\int_0^y \chi_{\frac{B_0}{4}}(\varl\phi)^2xdx}{\int\chi_{\frac{B_0}{4}}(\varl\phi)^2xdx} + O\bigg(\frac{1+y}{\lgba}\bigg).
	\end{aligned}
	\end{equation}
	Furthermore, the derivatives of $c_b, d_b$ on $b$ are
	\begin{equation} \notag
		\prb c_b =O\bigg(\frac{1}{b\lgba^2}\bigg),
		\quad \prb d_b =O\bigg(\frac{1}{b^2\lgba}\bigg),
	\end{equation}
	and thus
	\begin{equation} \label{eq: b derivative sigma b}
	\begin{aligned}
		\prb \sgb & = \prb c_b\,T_1\,\one_{y\leq \frac{B_0}{4}} +O\bigg(\frac{1}{b^2y\lgba}\bigg)\one_{\frac{B_0}{4}\leq y\leq 6B_0} \\
		& = O\bigg(\frac{y^3}{b\lgba^2}\,\one_{y\leq 1} +\frac{y}{b\lgba}\,\one_{1\leq y\leq 6B_0}\bigg).
	\end{aligned}
	\end{equation}

	\medskip
	\step{5}{Manipulation on the growing tails} Due to the explicit construction \eqref{eq: varp function}, \eqref{eq: S02 function} based on $T_1$, we observe the pattern $\varl T_1-T_1$ appears in each term in the third line on the RHS of \eqref{eq: errz expand}. By \eqref{eq: T1 asymptotics}, the asymptotics of $\varl T_1-T_1$ at the scale of $y\sim 2B_1$ is
	\begin{equation} \label{eq: lambda T1 T1}
		\varl T_1 -T_1=-y+O\bigg(\frac{(\log y)^2}{y}\bigg),
	\end{equation}
	suggesting the potential fast growing tails of the second order profiles (see \eqref{eq: high varp system}, \eqref{eq: varp 2 equation}), which will make the approximate solution out of control. To fix this, we use $\sgb$ to create the following cancellation for $6B_0\leq y\leq 2B_1$:
	\begin{equation} \label{eq: lambda T1 Sigma b outside}
		\varl T_1 -T_1 -\sgb =O\bigg(\frac{(\log y)^2}{y}\bigg).
	\end{equation}
	For $1\leq y\leq 6B_0$, there holds
	\begin{equation} \notag
	\begin{aligned}
		\varl T_1 -T_1 -\sgb & =
		-y\Bigg(1-\frac{\int_0^y\chi_{\frac{B_0}{4}}(\varl\phi)^2xdx}{\int\chi_{\frac{B_0}{4}}(\varl\phi)^2xdx}\Bigg) +O\bigg(\frac{(\log y)^2}{y}\bigg) +O\bigg(\frac{1+y}{\lgb}\bigg) \\
		& \les \frac{y}{\lgba}\Big(1+\lgyba\Big).
	\end{aligned}
	\end{equation}
	These with \eqref{eq: inverse of Ham} yield the bound
	\begin{equation} \label{eq: lambda T1 Sigma b inverse}
		\big|\Ham^{-1}\big(\varl T_1\hs-\hs T_1\hs-\hs\sgb\big)\big|
		= O\bigg(y^5\one_{y\leq 1} +\frac{1+y}{b\lgba}\one_{1\leq y\leq 2B_1}\bigg)
		\les \frac{1+y}{b\lgba}.
	\end{equation}
	Therefore we define the vectorial functions
	\begin{equation} \notag
		\sg_{1,0}\hs=\hs\frac{1}{\ao^2+\bo^2}\hs\begin{bmatrix} \,\ao\,\\ \bo\\ 0 \end{bmatrix}\hs\hs\sgb, \quad
		\sg_{0,1}\hs=\hs\frac{1}{\ao^2+\bo^2}\hs\begin{bmatrix} -\bo \\ \ao \\ 0 \end{bmatrix}\hs\hs\sgb, \quad	
		\sg_{0,2}\hs=\hs \frac{1}{\ao^2+\bo^2}T_1\sgb\,\ez,
	\end{equation}
	corresponding to the cancellations of $\varl\varp_{1,0}-\varp_{1,0}, \varl\varp_{0,1}-\varp_{0,1}, \varl S_{0,2}-S_{0,2}$ respectively. Similar to \eqref{eq: varp relation}, we have
	\begin{equation} \notag
		\sg_{0,1}= \R\sg_{1,0}, \quad \sg_{1,0}= -\R\sg_{0,1}.
	\end{equation}
	Then \eqref{eq: errz expand} can be rewritten as
	\begin{equation} \label{eq: errz expand 2}
	\begin{aligned}
		\errz 
		&= \hJ\phs\Bigg\{ (\ao-\bo\R)\hs\begin{bmatrix} 
			\,\mHamo(\wz^2)\,\\ \mHamt(\wz^2)\\ 0 
		\end{bmatrix} -2b^2\onez\phs\py S_{0,2}^{\third}\phs\big(\ao\ex-\bo\ey\big) \\
		&\qquad\quad  -\bo\wz\hs\wedge\hs\mHam\wz -\ao b^2\Delta S_{0,2}^{\third}\,\ez +2\ao\onez\sum_{i=1,2}\pyzy(\wz^{i})^{\first}\phsh\ez \Bigg\} \\
		&\quad +\sum_{2\leq i+j\leq 4}a^{i}b^{j}\ea_{i,j}+b^3\big(\varl S_{0,2}-2S_{0,2}+\sg_{0,2}\big) \\
		&\quad  +ab\sg_{1,0} +b^2\sg_{0,1} -b^3\sg_{0,2} -2a^2b\phs S_{0,2} -(a^2+b^2)\hspace{-.5em}\sum_{2\leq i+j\leq 3}\hspace{-.5em}a^{i}b^{j}\pr_b \varp_{i,j}.	
	\end{aligned}		
	\end{equation}
	where
	\begin{equation} \label{eq: E wz}
	\begin{aligned}
		\sum_{2\leq i+j\leq 4}\hspace{-.4em} a^{i}b^{j} \ea_{i,j}
		&= ab(1+Z)\phsh\varp_{1,0} -a^2(1+Z)\phsh\varp_{0,1} \\
		&\quad +ab\big(\varl\varp_{1,0}-\varp_{1,0}-\sg_{1,0}\big)+b^2\big(\varl\varp_{0,1}-\varp_{0,1}-\sg_{0,1}\big) \\
		&\quad +\sum_{2\leq i+j\leq 3}\hspace{-.3em} a^{i}b^{j}\Big(b\varl \varp_{i,j} -aZ\R \varp_{i,j}\Big) -(a^2+b^2)\hspace{-.5em}\sum_{2\leq i+j\leq 3}\hspace{-.5em}ja^{i}b^{j-1}\varp_{i,j}.
	\end{aligned}
	\end{equation}

	\medskip
	\step{6}{The $\hJ$ structure} From the definition \eqref{eq: hJ}, we have the identity
	\begin{equation} \label{eq: J ez}
		\hJ\ez =(\ez+\wz)\wedge\ez =-(\ez+\wz)\wedge\wz =-\hJ\wz.
	\end{equation}
	It implies the smallness of the third component under the action of $\hJ$, which is helpful in analyzing the terms inside the big brace in \eqref{eq: errz expand 2}. First, we compute the nonlinear term in \eqref{eq: errz expand 2}. By the decomposition \eqref{eq: wz decomposition}:	
	\begin{equation}
	\begin{aligned}
		\wz\wedge\mHam\wz 
		&= \wz^1\wedge\mHam\wz^1 +\wz^1\wedge\mHam\wz^2 +b^2\wz^1\wedge\mHam S_{0,2} \\
		&\quad +\wz^2\wedge\mHam\wz^1 +\wz^2\wedge\mHam\wz^2 +b^2\wz^2\wedge\mHam S_{0,2} \\
		&\quad +b^2S_{0,2}\wedge\mHam\wz^1 +b^2S_{0,2}\wedge\mHam\wz^2  +b^4S_{0,2}\wedge\mHam S_{0,2}.
	\end{aligned}
	\end{equation}
	From \eqref{eq: varp function}, \eqref{eq: mHam}, we see 
	\begin{equation} \notag
	\begin{aligned}
		\wz^1\wedge\mHam\wz^1 
		& = \wz^1\wedge\begin{bmatrix}
			\mHamo(\wz^1)\\ \mHamt(\wz^1)\\ 0
		\end{bmatrix} +\wz^1\wedge\begin{bmatrix}
			0\\ 0\\ \mHamth(\wz^1)
		\end{bmatrix} \\
		& = \big(a\phsh\varp_{1,0}+b\phsh\varp_{0,1}\big)\wedge\mHamp\big(a\phsh\varp_{1,0}+b\phsh\varp_{0,1}\big) -\mHamth(\wz^1)\big(\ez\wedge\wz^1\big) \\
		& = -2\onez\pyzy(\wz^1)^{\first}\phs\R\wz^1.
	\end{aligned}
	\end{equation}
	Besides, \eqref{eq: profile form} implies $(\wz^2)^{\third}=0$, which together with \eqref{eq: J ez} yields
	\begin{equation} \notag
	\begin{aligned}
		\hJ\big(\wz^1\wedge\mHam\wz^2\big)
		& = \hJ\Bigg\{ \wz^1\wedge\begin{bmatrix}
			\mHamo(\wz^2)\\ \mHamt(\wz^2)\\ 0
		\end{bmatrix} +\wz^1\wedge\begin{bmatrix}
			0\\ 0\\ \mHamth(\wz^2)
		\end{bmatrix} \Bigg\} \\
		& = \hJ\Bigg\{ O\Big(|\wz^1|\,|\mHam\wz^2|\,|\wz|\Big) -2\onez\pyzy(\wz^2)^{\first}\phs\R\wz^1\Bigg\}.
	\end{aligned}
	\end{equation} 
	An analogous identity switching the position of $\wz^1$ and $\wz^2$ holds for $\wz^{2}\wedge\mHam\wz^{1}$. Again from \eqref{eq: profile form}, $S_{0,2}=S_{0,2}^{\third}\,\ez$, and then direct computations give
	\begin{equation} \notag
	\begin{aligned}
		\hJ\Big( b^2\wz^1\wedge\mHam & S_{0,2}  +b^2S_{0,2}\wedge\mHam\wz^1 \Big) \\
		& = \hJ\Big( b^2\Delta S_{0,2}^{\third}\phs\R\wz^1 +2b^2\onez\phsh\py S_{0,2}^{\third}\phs(\wz^1)^{\second}\phsh\wz +b^2S_{0,2}^{\third}\phs\R\mHam\wz^1 \Big).
	\end{aligned}
	\end{equation}
	Combining these, the nonlinear term is actually given by
	\begin{equation} \label{eq: J in out nonlinear}
	\begin{aligned}
		\hJ\big(-\hs\bo\wz\hs\wedge\hs\mHam\wz\big)
		&= \hJ\Bigg\{\, 2\bo\onez\sum_{i+j\leq 3}\pyzy(\wz^i)^{\first}\phs\R\wz^{j}\\
		&\qquad\;\; -\bo b^2\Delta S_{0,2}^{\third}\phs\R\wz^1 -\bo b^2S_{0,2}^{\third}\phs\R\mHam\wz^1 +\sum_{4\leq i+j\leq 7}a^{i}b^{j}R_{i,j} \Bigg\},
	\end{aligned}
	\end{equation}
	where $R_{i,j}$ consists of functions with coefficient $a^{i}b^{j}$ for $4\leq i+j\leq 7$:
	\begin{equation} \notag
	\begin{aligned}
		\sum_{4\leq i+j\leq 7} a^{i}b^{j} R_{i,j}
		& = -\bo\bigg\{\phsh 2b^2\onez\phs\py S_{0,2}^{\third}\phs(\wz^1)^{\second}\phsh\wz +\wz^2\wedge\mHam\wz^2 \\
		&\qquad\qquad +b^2\wz^2\wedge\mHam S_{0,2} +b^2S_{0,2}\wedge\mHam\wz^2 +b^4S_{0,2}\wedge\mHam S_{0,2} \\
		&\qquad\qquad\qquad\quad +O\Big(|\wz^1|\,|\mHam\wz^2|\,|\wz|\Big) +O\Big(|\wz^2|\,|\mHam\wz^1|\,|\wz|\Big) \bigg\}.
	\end{aligned}
	\end{equation}
	Next, for the last two terms in the brace in \eqref{eq: errz expand 2}, using \eqref{eq: J ez} directly, we have
	\begin{equation} \label{eq: J in out ez}
	\begin{aligned}
		\hJ\bigg\{ -\ao b^2\Delta S_{0,2}^{\third}\phsh\ez &\,+2\ao\onez\sum_{i=1,2}\pyzy(\wz^{i})^{\first}\phsh\ez \bigg\} \\
		& = \hJ\bigg\{ \ao b^2\Delta S_{0,2}^{\third}\phsh\wz -2\ao\onez\sum_{i=1,2}\pyzy(\wz^{i})^{\first}\phs\wz \bigg\}.
	\end{aligned}
	\end{equation}
	Furthermore, for any vector $\vv$ with $\vv^{\third}=0$, we observe that
	\begin{equation} \notag
		\vv = \hJ\big(-\hsh\R\vv\big) +(\wz\cdot\vv)\ez -\wz^{\third}\vv.
	\end{equation} 
	In particular, take $\vv=E_{i,j}$, and there holds
	\begin{equation} \label{eq: J in out E}
	\begin{aligned}
		\sum_{2\leq i+j\leq 4} a^{i}b^{j}\ea_{i,j} 
		&= \sum_{2\leq i+j\leq 4}  a^{i}b^{j} \bigg\{ \hJ\phs\big(-\hsh\R\ea_{i,j}\big)+\big(\wz\cdot\ea_{i,j}\big)\ez-b^2S_{0,2}^{\third}\ea_{i,j} \bigg\},
	\end{aligned}
	\end{equation}
	where we note that $E_{i,j}^{\third}=0$ follows from $\varp_{i,j}^{\third}=\sg_{1,0}^{\third}=\sg_{0,1}^{\third}=0$. Now we use \eqref{eq: J in out E} to place $E_{i,j}$ into the big brace in \eqref{eq: errz expand 2}, and apply \eqref{eq: J in out nonlinear}, \eqref{eq: J in out ez}. Then the error becomes
	\begin{equation} \label{eq: errz expand 3}
	\begin{aligned}
		\errz 
		&= \hJ\phs\Bigg\{ (\ao-\bo\R)\hs\hs\begin{bmatrix} 
			\,\mHamo(\wz^2)\,\\ \mHamt(\wz^2)\\ 0 
		\end{bmatrix}\ +\sum_{2\leq i+j\leq 7}\hspace{-.4em}a^{i}b^{j}\tea_{i,j} \Bigg\} \\
		&\quad  +ab\sg_{1,0} +b^2\sg_{0,1} -b^3\sg_{0,2} -2a^2b\phs S_{0,2} -(a^2+b^2)\hspace{-.5em}\sum_{2\leq i+j\leq 3}\hspace{-.5em}a^{i}b^{j}\pr_b \varp_{i,j} \\
		&\quad +\sum_{2\leq i+j\leq 4} a^{i}b^{j} \Big[\big(\wz\cdot\ea_{i,j}\big)\ez-b^2S_{0,2}^{\third}\ea_{i,j}\Big] +b^3\big(\varl S_{0,2}-2S_{0,2}+\sg_{0,2}\big),
	\end{aligned}
	\end{equation}
	where
	\begin{equation} \label{eq: Et wz}
	\begin{aligned}
		\sum_{2\leq i+j\leq 7}\hspace{-.3em}a^{i}b^{j}\tea_{i,j}
		& = -\hspace{-.3em}\sum_{2\leq i+j\leq 4} a^{i}b^{j}\R\ea_{i,j} -2b^2\onez\phs\py S_{0,2}^{\third}\phs\big(\ao\ex\hs-\hs\bo\ey\big) \\
		& -2\onez\hs\hs\sum_{i+j\leq 3}\hspace{-.3em}\pyzy(\wz^i)^{\first}\big(\ao\hs-\hs\bo\R\big)\wz^{j} -\bo b^2S_{0,2}^{\third}\phs\R\mHam\wz^1  \\
		& +b^2\Delta S_{0,2}^{\third}\big(\ao\hs-\hs\bo\R\big)\wz^1 +2\ao b^2\onez\pyzy(\wz^1)^{\first}S_{0,2}^{\third}\,\wz \\
		& +\ao\bigg[ b^2\Delta S_{0,2}^{\third} -2\onez\pyzy(\wz^2)^{\first}\bigg]\big(\wz^2\hs+\hs b^2S_{0,2}\big) \\
		& +\sum_{4\leq i+j\leq 7}a^{i}b^{j}R_{i,j}.
	\end{aligned}
	\end{equation}
	
	\medskip
	\step{7}{Construction of $\varp_{i,j}$} We choose suitable $\varp_{i,j}$ by eliminating the error \eqref{eq: Et wz}. More precisely, we let 
	\begin{equation} \notag
		(\ao-\bo\R)\hs\hs\begin{bmatrix} 
			\,\mHamo(\wz^2)\,\\ \mHamt(\wz^2)\\ 0 
		\end{bmatrix} +\sum_{2\leq i+j\leq 3}\hspace{-.4em}a^{i}b^{j}\tea_{i,j} = 0,
	\end{equation}
	which, in view of \eqref{eq: profile form}, yields the following linear system
	\begin{equation} \label{eq: high varp system}
		\begin{bmatrix}
			\ao & \bo \\ -\bo & \ao 
		\end{bmatrix} \lf[\begin{aligned}
			\,\Ham\varp_{i,j}^{\first}\, \\ \,\Ham\varp_{i,j}^{\second}\,
		\end{aligned}\rg]
		= -\lf[\begin{aligned}
			\,\tea_{i,j}^{\first}\, \\ \,\tea_{i,j}^{\second}\,
		\end{aligned}\rg],
		\quad\mbox{for}\quad 2\leq i+j\leq 3.
	\end{equation}
	It can also be solved using \eqref{eq: inverse of Ham}, and then the constructions of $\varp_{i,j}$ are obtained. We consider first the case $i+j=2$. From \eqref{eq: E wz}, \eqref{eq: Et wz}, we see
	\begin{equation} \label{eq: varp 2 equation}
	\begin{aligned}
		-\sum_{i+j=2} \hspace{-.2em} a^{i}b^{j} \tea_{i,j} 
		&= a^2 \bigg\{ \hs\onez\varp_{1,0} +2\ao(1+Z)\pyzy\varp_{1,0}^{\first}\big(\ao\hs-\hs\bo\R\big)\varp_{1,0}^{\ptone} \bigg\} \\
		& +ab \bigg\{ \onez\varp_{0,1} +\big(\varl\varp_{0,1}-\varp_{0,1}-\sg_{0,1}\big) \\
		&\qquad\qquad\qquad +2\ao\onez\sum_{\substack{i+k=1\\j+l=1}}\pyzy\varp_{i,j}^{\first} \big(\ao\hs-\hs\bo\R\big)\varp_{k,l}^{\ptone} \bigg\}\\
		& +b^2 \bigg\{ 2\onez\py S_{0,2}^{\third}\big(\ao\ex\hs-\hs\bo\ey\big) -\big(\varl\varp_{1,0}-\varp_{1,0}-\sg_{1,0}\big) \\
		&\qquad\qquad\qquad\qquad +2\ao\onez\pyzy\varp_{0,1}^{\first}\big(\ao\hs-\hs\bo\R\big)\varp_{0,1}^{\ptone} \bigg\}.
	\end{aligned}
	\end{equation}
	This together with the asymptotics for $\varp_{1,0}, \varp_{0,1}, S_{0,2}$ and \eqref{eq: lambda T1 Sigma b inverse} implies the bound
	\begin{equation} \label{eq: varp 2 asymptotics}
		\varp_{i,j} =O\bigg(y^5\one_{y\leq 1}+\frac{1+y}{b\lgba}\one_{y\geq 1}\bigg) \les \frac{1+y}{b\lgba},
	\end{equation}
	and also the crude bound
	\begin{equation} \notag
		|\varp_{i,j}|\les 1+y^3.
	\end{equation}
	Moreover, from \eqref{eq: b derivative sigma b}, we have the estimate for the derivatives on $b$:
	\begin{equation} \notag
		\prb \varp_{i,j} = O\bigg(\frac{y^5}{b\lgba^2}\,\one_{y\leq 1} +\frac{y^3}{b\lgba}\,\one_{1\leq y\leq 6B_0}\bigg).
	\end{equation}
	Next for $i+j=3$, from \eqref{eq: E wz}, \eqref{eq: Et wz} again, we have
	\def\ione{i}
	\def\itwo{j}
	\def\jone{k}
	\def\jtwo{l}
	\def\isub{{\ione,\itwo}}
	\def\jsub{{\jone,\jtwo}}
	\newcommand{\varssum}[2]{\sum_{\substack{\ione+\jone={#1}\\ \itwo+\jtwo={#2}}}\hspace{-.2em}\bigg[\pyzy\varp_{\isub}^{\first}\phs S_{\jsub}^{\ptone} +\pyzy S_{\isub}^{\first}\phs\varp_{\jsub}^{\ptone}\bigg]}
	\begin{equation} \notag
	\begin{aligned}
		-\sum_{i+j=3} \hspace{-.2em} a^{i}b^{j} \tea_{i,j} 
		& = \sum_{i+j=2}\hspace{-.3em} a^{i}b^{j}\Big(b\varl\R \varp_{i,j} +aZ \varp_{i,j}\Big) +(a^2+b^2)\hspace{-.3em}\sum_{i+j=2}\hspace{-.3em}ja^{i}b^{j-1}\R \varp_{i,j} \\
		& -2\onez\hs\hs\sum_{i+j=3}\hspace{-.3em}\pyzy(\wz^i)^{\first}\big(\ao\hs-\hs\bo\R\big)\wz^{j} \\
		& -\bo b^2S_{0,2}^{\third}\phs\R\mHam\wz^1 +b^2\Delta S_{0,2}^{\third}\big(\ao\hs-\hs\bo\R\big)\wz^1.
	\end{aligned}
	\end{equation}
	This together with \eqref{eq: varp function}, \eqref{eq: S02 function}, \eqref{eq: varp 2 asymptotics} gives for $i+j=3$ that
	\begin{equation} \label{eq: varp 3 asymptotics}
		\varp_{i,j} =O\bigg(y^7\one_{y\leq 1}+\frac{1+y^3}{b\lgba}\one_{y\geq 1}\bigg) \les \frac{1+y^3}{b\lgba},
	\end{equation}
	the crude bound
	\begin{equation} \notag
		|\varp_{i,j}|\les 1+y^5,
	\end{equation}
	and the derivatives
	\begin{equation} \notag
		\prb \varp_{i,j} = O\bigg(\frac{y^7}{b\lgba^2}\,\one_{y\leq 1} +\frac{y^5}{b\lgba}\,\one_{1\leq y\leq 6B_0}\bigg).
	\end{equation}
	
	\medskip
	\step{8}{Estimates on the error} Due to the choice of $\varp_{i,j}$, the error \eqref{eq: errz expand 3} boils down to
	\begin{equation} \label{eq: errz expand 4}
	\begin{aligned}
		\errz 
		&= ab\sg_{1,0} +b^2\sg_{0,1} -b^3\sg_{0,2} -2a^2b\phs S_{0,2} -(a^2+b^2)\hspace{-.5em}\sum_{2\leq i+j\leq 3}\hspace{-.5em}a^{i}b^{j}\pr_b \varp_{i,j} \\
		&\quad +\sum_{2\leq i+j\leq 4} a^{i}b^{j} \big(\wz\cdot\ea_{i,j}\big)\ez +b^3\big(\varl S_{0,2}-2S_{0,2}+\sg_{0,2}\big) \\
		&\quad -b^2\phsh S_{0,2}^{\third}\hs\sum_{2\leq i+j\leq 4} a^{i}b^{j}\ea_{i,j} +\hJ\phs\bigg( \sum_{4\leq i+j\leq 7} \hs\hs a^{i}b^{j}\tea_{i,j} \bigg).	
		\end{aligned}
	\end{equation}
	For each line of the RHS of \eqref{eq: errz expand 4}, we estimate them in the region $[0, 2B_1]$ by previous asymptotics. From \eqref{eq: sigma b asymptotics}, \eqref{eq: S asymptotics}, there holds
	\begin{equation} \label{eq: main error of errz}
	\begin{aligned}
		&\quad ab\sg_{1,0} +b^2\sg_{0,1} -b^3\sg_{0,2} -2a^2b\phs S_{0,2} \\
		& = O\big(b^2 \sgb\big)(\ex+\ey) +O\big(b^3 T_1\sgb\big)\ez + O\big(ab^2 S_{0,2}\big) \\
		& = b^2\,O\bigg(\frac{y^3}{\lgba}\one_{y\leq 1} +\frac{y\log y}{\lgba}\one_{1\leq y\leq \frac{B_0}{4}} +y\one_{y\geq \frac{B_0}{4}}\bigg)\big(\ex+\ey\big) \\
		&\quad +b^2\,O\bigg(\frac{y^6}{\lgba}\one_{y\leq 1} +\frac{y^2(\log y)^2}{\lgba}\one_{1\leq y\leq \frac{B_0}{4}} +y^2\log y\one_{y\geq \frac{B_0}{4}}\bigg)\ez \\
		&\quad +a^2b\,O\Big(y^2(\log y)^2\one_{y\geq 1}\Big).
	\end{aligned}
	\end{equation}
	By the $b$ derivatives estimates, we have
	\begin{equation} \notag
		(a^2+b^2)\hspace{-.3em}\sum_{2\leq i+j\leq 3}a^{i}b^{j}\prb \varp_{i,j}
		= O\bigg(\frac{b^3 y^5}{\lgba^2}\one_{y\leq 1} +\frac{b^3 y^3}{\lgba}\one_{1\leq y\leq 6B_0}\bigg).
	\end{equation}
	For the second line, \eqref{eq: S02 function} induces the cancellation:
	\begin{equation} \notag
		\Big(b\big(\varl\varp_{0,1} \hs-\hs\varp_{0,1} \hs-\hs\sg_{0,1}\big)\cdot b\,\varp_{0,1}\Big)\phsh\ez +b^3\big(\varl S_{0,2}-2S_{0,2}+\sg_{0,2}\big) =0,
	\end{equation}
	which together with explicit formula \eqref{eq: E wz} gives
	\begin{equation} \notag
	\begin{aligned}
		&\quad \sum_{2\leq i+j\leq 4} a^{i}b^{j} \big(\wz\cdot\ea_{i,j}\big)\ez +b^3\big(\varl S_{0,2}-2S_{0,2}+\sg_{0,2}\big) \\
		& = ab^2\, O\bigg(y^6\one_{y\leq 1} +\frac{y^2\log y}{\lgba}\Big(\hs1\hs+\hs\lgyba\hs\Big)\one_{1\leq y\leq 6B_0} +\frac{y^2\log y}{\lgba}\one_{y\geq 6B_0}\bigg)\,\ez.
	\end{aligned}
	\end{equation}
	For the third line, from \eqref{eq: E wz}, \eqref{eq: Et wz}, and the crude bounds of $\varp_{i,j}$, we have
	\begin{equation} \notag
		\Big|b^2\phsh S_{0,2}^{\third}\hs\sum_{2\leq i+j\leq 4} a^{i}b^{j}\ea_{i,j}\Big| +\Big|\sum_{4\leq i+j\leq 7}a^{i}b^{j}\tea_{i,j} \Big|  
		\les b^4 \Big(y^9\one_{y\leq 1} +y^5(\log y)^2\one_{y\geq 1} \Big).
	\end{equation}
	Collecting these estimates, we see \eqref{eq: main error of errz} is the dominate part of $\errz$, which yields the bound
	\begin{equation} \label{eq: weighted bound 1 part}
		\int_{y\leq 2B_1}\frac{|\errz^{\first}|^2+|\errz^{\second}|^2}{y^{6}}
		\les\frac{b^4}{\lgba^2}, \quad
		\int_{y\leq 2B_1}\frac{|\errz^{\third}|^2}{y^{8}}
		\les\frac{b^6}{\lgba^2}.
	\end{equation}
	This yields \eqref{eq: weighted bound 1}, \eqref{eq: weighted bound 2} with $i=0$. The cases $i\geq 1$ are similar. After the action of $\Ham$, we have
	\begin{equation} \notag
	\begin{aligned}
		&\quad ab\Ham\sg_{1,0} +b^2\Ham\sg_{0,1} = O\big(b^2 \sgb\big)(\ex+\ey) \\
		& = b^2\,O\bigg(\frac{y}{\lgba}\one_{y\leq 1} +\frac{1}{y\lgba}\one_{1\leq y\leq 6B_0} \bigg)\big(\ex+\ey\big),
	\end{aligned}
	\end{equation}
	which still leads to the same bounds as \eqref{eq: weighted bound 1 part}. So do most of the other terms, except the $b^4\R\varl\varp_{0,3}$ coming from the last term of \eqref{eq: Et wz}, which, from \eqref{eq: varp 3 asymptotics}, can be bounded by
	\begin{equation} \notag
		\int_{y\leq 2B_1}\hs\hs\hs |\Ham\big(b^4\R\varl\varp_{0,3}\big)|^2
		\les b^8 \bigg(\int_{y\leq 1}\hs\hs|\Ham(y^7)|^2 +\int_{1\leq y\leq 2B_1}\hs\frac{|\Ham(y^3)|^2}{b^2\lgba^2}\bigg) 
		 \les b^4\lgba^2,
	\end{equation}
	and thus \eqref{eq: weighted bound 3} follows. The estimates \eqref{eq: weighted bound 4}--\eqref{eq: weighted bound 6} are obtained similarly. The details are left to readers. Finally, for the flux computations, we observe from \eqref{eq: sigma b} that 
	\begin{equation} \notag
	\begin{aligned}
		(\mHamp\errz,\varp_M) 
		& = \big(ab\mHamp\sg_{1,0} +b^2\mHamp\sg_{0,1},\varp_M\big) +C(M)\frac{b^6}{\lgba^2} \\
		& = \frac{ab\phsh c_b\big(\varl\phi,\varp_M\big)}{\ao^2+\bo^2} \begin{bmatrix} \ao\\ \bo\\ 0 \end{bmatrix} 
		+\frac{b^2\phsh c_b\big(\varl\phi,\varp_M\big)}{\ao^2+\bo^2} \begin{bmatrix} -\bo\\ \ao\\ 0 \end{bmatrix} +C(M)\frac{b^6}{\lgba^2},
	\end{aligned}
	\end{equation}
	which together with \eqref{eq: cb} yields the desired \eqref{eq: flux computation}. This concludes the proof.
	\end{proof}
	\begin{remark}
		The construction of $S_{0,2}$ \eqref{eq: S02 function} may seem to be unnatural, but it actually corresponds to the constraint \eqref{eq: component constraint}, ensuring that LL flow maps from $\RR^2$ to $\SS^2$. On the contrary, if we simply let $\wz^{\third}=S_{0,2}^{\third}=0$, the error arising therefrom will make $\w^{\third}=\gamma$ {\rm (}see \eqref{eq: decomposition}, \eqref{eq: w}{\rm )} out of control, and destroy the corresponding bounds in Lemma~{\rm \ref{le: bounds for gamma}}. 
		\end{remark}
	
	\subsection{Localization of the profiles}
	\label{SS: Localization}
	The aim of this subsection is to localize the approximate solution constructed in Lemma \ref{le: ap solution} within the spacial scale $y\leq 2B_1$.
	
	\begin{lemma}[Localization]
	\label{le: localized ap solution}
	Under the assumption of Lemma {\rm \ref{le: ap solution}}, we define the localized profile 
	\begin{equation} \label{eq: localized ap sol}
		\wzt = \chib\wz,
	\end{equation}
	and $\varpt_{i,j} = \chib\varp_{i,j}$ accordingly. Then \eqref{eq: localized ap sol} is an approximate solution to \eqref{eq: vec hw equation} in the sense that
	\begin{equation} \label{eq: localized error}
	\begin{aligned}
		\errzt +\mmodt(t) = \prs\wzt -\frac{\la_s}{\la}\varl\wzt
		+\vart_s Z\R\wzt
		+\tJ\Big(\, \ao\mHam\wzt -\bo\tJ\mHam\wzt
		+\p\,\varl\phi \,\Big),
	\end{aligned}
	\end{equation}
	where $\tJ := \big(\ez+ \wzt\big)\wedge$. The localized modulation vector $\mmodt$ is given by
	\begin{equation} \label{eq: localized mod vec}
   	\begin{aligned}
   		\mmodt & = \chib\mmod +\big(b_s+b^2+a^2\big)\,O\Big(\frac{\wz}{b}\Big)\one_{y\sim B_1} -\Big(\frac{\la_s}{\la}+b\Big)O(\wz)\one_{y\sim B_1} \\
   		&\quad +(\vart_s+a)\varl\phi\phsh\big(\ey+O(\wzt)\big)\one_{y\gtrsim B_1} -\Big(\frac{\la_s}{\la}+b\Big)\varl\phi\phsh\big(\ex+O(\wzt)\big)\one_{y\gtrsim B_1},
   	\end{aligned}
   	\end{equation}
	while the localized error $\errzt$ satisfies the following estimates:
	\begin{align}
		& \int_{y\leq 2B_1}\frac{|\pr_y^i\errzt^{\first}|^2+|\pr_y^i\errzt^{\second}|^2}{y^{6-2i}} 
		\les\bflog, \quad 0\leq i\leq 3,
		\label{eq: localized weighted bound 1}\\
		& \int_{y\leq 2B_1}\frac{|\pr_y^i\errzt^{\third}|^2}{y^{8-2i}}\les\frac{b^6}{\lgba^2}, \quad 0\leq i\leq 4,
		\label{eq: localized weighted bound 2}\\
		& \int_{y\leq 2B_1}|\Ham\errzt^{\first}|^2 +|\Ham\errzt^{\second}|^2
		\les b^4\lgba^2,
		\label{eq: localized weighted bound 3}\\
		& \int_{y\leq 2B_1}\frac{|\pr_y^i\Ham\errzt^{\first}|^2 +|\pr_y^i\Ham\errzt^{\second}|^2}{y^{2-2i}} \les\bflog, \quad 0\leq i\leq 1,
		\label{eq: localized weighted bound 4}\\
		& \int_{y\leq 2B_1}|\A\Ham\errzt^{\sups{1}}|^2 +|\A\Ham\errzt^{\sups{2}}|^2 \les b^5,
		\label{eq: localized weighted bound 5}\\
		& \int_{y\leq 2B_1}|\Ham^2\errzt^{\sups{1}}|^2 +|\Ham^2\errzt^{\sups{2}}|^2 \les \frac{b^6}{\lgba^2}.
		\label{eq: localized weighted bound 6}\\
		& \frac{(\Ham\errzt^{\first},\varp_M)}{(\varl\phi,\varp_M)} = \frac{2(\ao ab -\bo b^2)}{(\ao^2 +\bo^2)\lgba}\bigg(1+O\Big(\frac{1}{\lgba}\Big)\bigg),
		\label{eq: localized flux computation 1}\\
		& \frac{(\Ham\errzt^{\second},\varp_M)}{(\varl\phi,\varp_M)} = \frac{2(\ao b^2 +\bo ab)}{(\ao^2 +\bo^2)\lgba}\bigg(1+O\Big(\frac{1}{\lgba}\Big)\bigg).
		\label{eq: localized flux computation 2}
	\end{align}
	\end{lemma}
	
	\begin{proof}{Lemma {\rm\ref{le: localized ap solution}}}
	From the localization \eqref{eq: localized ap sol} and the error of the approximate solution \eqref{eq: ap error expand}, we compute:
	\begin{equation} \label{eq: localization expand}
	\begin{aligned}
		&\quad \prs\wzt -\frac{\la_s}{\la}\varl\wzt
		+\vart_s Z\R\wzt +\tJ\Big(\, \ao\mHam\wzt -\bo\tJ\mHam\wzt
		+\p\phsh\varl\phi \,\Big) \\
		& = \prs\chib\wz -\frac{\la_s}{\la}\lchib\wz
		+ \chib\Big(\prs\wz-\frac{\la_s}{\la}\varl\wz +\vart_s ZR\wz\Big)\\
		&\quad +\Big((\ez\hs\hs+\hs\hs\wz)\hs-\hs(1\hs\hs-\hs\hs\chib)\wz\Big)\wedge\Big(\ao\mHam(\chib\wz)\hs-\hs\bo(\ez\hs+\hs\wzt)\hs\wedge\hs\mHam(\chib\wz)\hs+\hs\p\phsh\varl\phi\Big) \\
		& =\chib\big(\errz + \mmod\big) \\
		&\quad +\prs\chib\wz -\hs\frac{\la_s}{\la}\lchib\wz +(1-\chib)\tJ\big(\p\phsh\varl\phi\big)\\
		&\quad -\chib(1-\chib)\wz\wedge\Big(\ao\mHam\wz-\bo(\ez+\wz)\wedge\mHam\wz+\p\phsh\varl\phi\Big)\\
		&\quad +\bo\chib(1-\chib)\tJ\big(\wz\wedge\mHam\wz\big)\\
		&\quad -\tJ\Big\{\hs\big(\ao\hs-\hs\bo(\ez+\wzt)\hs\wedge\hs\hs\big)\Big(2\py\chib\py\wz +\Delta_y\wz +2(1+Z)\py\chib \,\ey\wedge\wz\hs\Big)\hs\Big\}.
	\end{aligned}
	\end{equation}
	For the second line of RHS of \eqref{eq: localization expand}, we reformulate them according to the anticipated modulation equations
	\begin{equation} \notag
	\begin{aligned}
		&\quad \prs\chib\wz -\frac{\la_s}{\la}\lchib\wz
		+(1-\chib)\phsh\tJ\big(\p\phsh\varl\phi\big)\\
		&= \big(b_s+b^2+a^2\big)\,O\Big(\frac{\wz}{b}\Big)\one_{y\sim B_1} -\Big(\frac{\la_s}{\la}+b\Big)O(\wz)\one_{y\sim B_1} \\
		&\quad +O\big(b^2y\log y\big)(\ex+\ey)\one_{y\sim B_1} +O\big(b^3y^2(\log y)^2\big)\ez \one_{y\sim B_1} +O\big(b\phsh y^{-1}\big)(\ex+\ey)\one_{y\gtrsim B_1} \\
		&\quad +(\vart_s+a)\varl\phi\big(\ey+O(\wzt)\big)\phsh\one_{y\gtrsim B_1} -\Big(\frac{\la_s}{\la}+b\Big)\varl\phi\big(\ex+O(\wzt)\big)\phsh\one_{y\gtrsim B_1}.
	\end{aligned}
	\end{equation}
	To treat the third line of the RHS of \eqref{eq: localization expand}, from the cancellation \eqref{eq: varp function equation}, which is
	\begin{equation} \notag
		\big(\ao-\bo\R\big)\begin{bmatrix}
			\,\mHamo(\wz^{1})\,\\ \mHamt(\wz^{1})\\ 0
		\end{bmatrix} -\begin{bmatrix}
			\,a\,\\ b\\ 0
		\end{bmatrix} \varLambda\phi =0,
	\end{equation}
	we have the estimate
	\begin{equation} \notag
	\begin{aligned}
		&\quad -\chib(1\hs-\hs\chib)\wz\wedge\Big(\ao\mHam\wz-\bo(\ez+\wz)\wedge\mHam\wz+\varl\phi\,\p\Big) \\
		&= O\Big(|\wz|\cdot\big(|\mHam(\wz^{2})|+b^2|\mHam S_{0,2}| +|\wz|\,|\mHam\wz|\big)\Big) \one_{y\sim B_1} \\
		&\qquad -\big(\vart_s+a\big)\varl\phi\phs\big(\ex\wedge\wz\big)\one_{y\sim B_1} -\Big(\frac{\la_s}{\la}+b\Big)\varl\phi\phs\big(\ey\wedge\wz\big)\one_{y\sim B_1} \\
		&= O\big(b^3\phsh y^2\log y\big)\one_{y\sim B_1} -\big(\vart_s+a\big)\varl\phi\phs O(\wz)\one_{y\sim B_1} -\Big(\frac{\la_s}{\la}+b\Big)\varl\phi\phs O(\wz)\one_{y\sim B_1}.
	\end{aligned}
	\end{equation}
	Moreover, the fourth line of \eqref{eq: localization expand} is in fact
	\begin{equation} \notag
	\begin{aligned}
		& \bo\chib(1\hs-\hs\chib) \tJ\big(\wz\hs\wedge\hs\mHam\wz\big) \\ 
		& \qquad\qquad = O\Big(|\wz|\phsh|\mHam\wz| +|\wz|^2\phsh|\mHam\wz|\Big) \one_{y\sim B_1} = O\big(b^2(\log y)^2\big)\phsh\one_{y\sim B_1}. 
	\end{aligned}
	\end{equation}
	Finally for the last line of \eqref{eq: localization expand}, we estimate it by brute force
	\begin{equation} \notag
	\begin{aligned}
		&\quad -\tJ\Big\{\hs\big(\ao\hs-\hs\bo(\ez+\wzt)\hs\wedge\hs\big)\Big(2\py\chib\py\wz +\Delta_y\wz +2(1\hs+\hs Z)\py\chib\phsh\ey\hs\wedge\hs\wz\Big)\hs\Big\} \\
		& = O\bigg(\frac{1}{B_1} |\py\wz| + \frac{1}{B_1^2} |\wz| +\frac{1}{B_1(1+y^2)} |\wz| \bigg)\,\one_{y\sim B_1} 
		= O\bigg(\frac{b^{\frac{3}{2}}\log y}{\lgba}\bigg)\phsh\one_{y\sim B_1}.
	\end{aligned}
  	\end{equation}
	Injecting these computations into \eqref{eq: localization expand}, we obtain the localized error
	\begin{equation} \label{eq: errzt expand}
	\begin{aligned}
		\errzt 
		&= \chib\errz +O\big(b^2y\log y\big)(\ex+\ey)\one_{y\sim B_1} \\
		&\quad +O\big(b^3y^2(\log y)^2\big)\ez \one_{y\sim B_1} +O\big(b\phsh y^{-1}\big)(\ex+\ey)\one_{y\gtrsim B_1} \\
		&\quad +O\big(b^3\phsh y^2\log y\big)\one_{y\sim B_1} +O\big(b^2(\log y)^2\big)\phsh\one_{y\sim B_1} +O\bigg(\frac{b^{\frac{3}{2}}\log y}{\lgba}\bigg)\phsh\one_{y\sim B_1} \\
		&= \chib\errz +O\big(b^2y\log y\big)(\ex+\ey)\one_{y\sim B_1}\\
		&\quad +O\big(b^3y^2(\log y)^2\big)\ez \one_{y\sim B_1} +O\big(b\phsh y^{-1}\big)(\ex+\ey)\one_{y\gtrsim B_1},
	\end{aligned}
	\end{equation}
	and also the explicit formula of the localized modulation vector \eqref{eq: localized mod vec}. Note from \eqref{eq: errzt expand} that most of the additional errors are supported in the region $[B_1,2B_1]$ (labelled by $\one_{y\sim B_1}$), and thus we can easily check they do not perturb the estimate in Lemma \ref{le: ap solution}, in particular thanks to the choice $B_1$. The details are left to readers.
	\end{proof}

	\section{The Trapped Regime}
	\label{S: the trapped regime}
	We now aim at seeking for an actual blowup solution to \eqref{eq: LL}. To this end, we in this section study the solution of the form \eqref{eq: decomposition}. More precisely, we set up the bootstrap regime, give the orthogonality condition, and compute the modulation equations of the geometrical parameters.
	
	\subsection{Bootstrap setup and orthogonality conditions}
	\label{SS: bootstrap setup and orthogonality conditions}
	In this subsection, we describe the set of initial data leading to the blowup scenario of Theorem \ref{th: main th}, and give the corresponding bootstrap statement. For any initial data $u_0\in\dot{H}^1$ with
	\begin{equation}
		\|\nabla u_0-\nabla Q\|_{L^2}\ll 1,
	\end{equation}
	we may decompose the corresponding solution $u(t)$ in a small time interval $[0,t_1]$ by
	\begin{equation} \label{eq: decomposition}
		u(t) = e^{\vart \R}\big(Q+\hv\big)_{\la},
		\quad\mbox{with}\quad
		\lf\{\begin{matrix}
			\hv=\big[\er,\etau,Q\big]\hw,\\
			\hspace{-2em}\hw=\wzt+\w,
		\end{matrix}\rg.
	\end{equation}
	where $\wzt$ is the localized approximate solution, and $\w$ is a correction term called the radiation. The existence of this decomposition at any fixed time follows from the variational characterization of $Q$, see for example \cite{martelmerle2002stabilityofblowup, 2013Raphael_HarmonicHeatQuantizedBlowup}. Then by the standard modulation theory, the solution can be modulated such that there exists the geometrical parameter maps $\la,\vart,a,b\in \CCC^1([0,t_1],\RR)$, and the radiation
	\begin{equation} \label{eq: w}
		\w=\begin{bmatrix}\alpha\\ \beta\\ \gamma\end{bmatrix}
	\end{equation}
	satisfying the orthogonality conditions
	\begin{equation} \label{eq: orthogonality}
		(\alpha, \varp_M) = (\alpha, \Ham\varp_M) = (\beta, \varp_M) = (\beta, \Ham\varp_M) = 0.
	\end{equation}
	Here $\varp_M$ is a localized replacement of $\varl\phi$ defined, for sufficiently large universal constant $M\gg 1$, by
	\begin{equation} \label{eq: varp M}
		\varp_M = \chi_M\varl\phi -c_M \Ham\big(\chi_M\varl\phi\big),
	\end{equation}
	with
	\begin{equation} \notag
		c_M = \frac{(\chi_M\varl\phi,\,T_1)}{(\Ham(\chi_M\varl\phi),\,T_1)} \sim c_{\chi}M^2\big(1+o(1)\big).
	\end{equation}
	It admits the following non-degenerate and orthogonal properties
	\begin{equation} \label{eq: varp M properties}
	\lf\{\begin{aligned}
		&\, \|\varp_M\|_{L^2}= 4\log M \big(1+o(1)\big),\\
		&\, (\varl\phi,\varp_M) =(T_1, \Ham\varp_M) =4\log M\big(1+o(1)\big),\\
		&\, (T_1, \varp_M) =(\varl\phi, \Ham\varp_M) =0.
	\end{aligned}\rg.
	\end{equation}
	Indeed, to study the orthogonality \eqref{eq: orthogonality}, we define the vectorial function 
	\begin{equation} \notag
		\mf(\la,\vart,a,b,u) = \Big[(\alpha,\varp_M),\,(\beta,\varp_M),\,(\alpha,\Ham\varp_M),\,(\beta,\Ham\varp_M)\Big].
	\end{equation}
	Using the explicit construction of $\wzt$ in Lemma~\ref{le: localized ap solution}, we compute the following derivatives at point $P$ where $(\la,\vart,a,b,u)=(1,0,0,0,Q)$
	\begin{equation} \notag
		\begin{aligned}
		&\,\frac{\pr}{\pr\la}\Big|_{P}\Big(e^{-\vart R}u_{\frac{1}{\la}}\Big)=\varl\phi\,\er,\\
		&\,\frac{\pr}{\pr\vart}\Big|_{P}\Big(e^{-\vart R}u_{\frac{1}{\la}}\Big)=-\varl\phi\,\etau,
		\end{aligned} \quad\quad
		\begin{aligned}
		&\,\frac{\pr}{\pr a}\Big|_{P}\big(\wzt\big)=\varpt_{1,0},\\
		&\,\frac{\pr}{\pr b}\Big|_{P}\big(\wzt\big)=\varpt_{0,1},
		\end{aligned}
	\end{equation}
	from which follows the non-degeneracy of the Jacobian of $\mf$ at $P$:	\begin{equation} \label{eq: Jacobian}
	\lf|\begin{matrix}
	(\varl\phi,\varp_M)&0&(\varl\phi,\Ham\varp_M)&0\\
	0&(\varl\phi,\varp_M)&0&(\varl\phi,\Ham\varp_M)\\[5pt]
	\big(\varp_{1,0}^{\first},\varp_M\big)&
	\big(\varp_{1,0}^{\second},\varp_M\big)&
	\big(\varp_{1,0}^{\first},\Ham\varp_M\big)&
	\big(\varp_{1,0}^{\second},\Ham\varp_M\big)\\[5pt]
	\big(\varp_{0,1}^{\first},\varp_M\big)&
	\big(\varp_{0,1}^{\second},\varp_M\big)&
	\big(\varp_{0,1}^{\first},\Ham\varp_M\big)&
	\big(\varp_{0,1}^{\second},\Ham\varp_M\big)
	\end{matrix}\rg|=\frac{\big(1\hs+\hs o(1)\big)}{\ao^2+\bo^2}(\varl\phi,\varp_M)^4 \neq 0. \notag
	\end{equation}
	According to the implicit function theorem, there exist a constant $\delta>0$, a neighborhood $V_{P}$ of the point $P$, and a unique $\CCC^1$ geometrical parameter map 
	\begin{equation} \notag
		(\la,\vart,a,b):\big\{u\in \dot{H}^1:\|u-Q\|_{\dot{H}^1}<\delta\big\}\to V_{P},
	\end{equation}
	such that $\mf(\la,\vart,a,b,u)=0$. In view of \eqref{eq: decomposition}, this ensures the orthogonality of the radiation \eqref{eq: orthogonality}. Moveover, by the regularity of LL flow, the map $(\la,\vart,a,b)$ is $\CCC^1$ function of $t$, concluding our claim. Then we may measure the regularity of the $\w$ by the following Sobolev norms associated to the linear operator \eqref{eq: mHamp}: the energy norm
	\begin{equation} \label{def: E1}
		\E_1=\int\big|\nabla \w\big|^2 +\Big|\frac{\w}{y}\Big|^2,
	\end{equation}
	and higher-order energy norm
	\begin{equation} \label{def: E2 E4}
		\E_2=\int|\mHamp\w|^2,  \quad
		\E_4=\int|(\mHamp)^2\w|^2.
	\end{equation}
	We now make the following assumptions on the initial data $u_0$, which describe a codimension one set (see Proposition~\ref{pro: trapped regime} and its proof in subsection \ref{pf: trapped regime}) close to the ground state $Q$ in $\dot{H}^1$:
	\begin{itemize}
	\item Initial bound on the modulation parameters:
		\begin{equation} \label{eq: initial b a}
			0< b(0) < b^{\ast}(M) \ll 1,
			\quad a(0) \leq \dfrac{b(0)}{4 \lf|\lgb(0)\rg|}.
		\end{equation}
		\item Initial data of the scaling parameter: (Up to a fixed scaling, we can always assume this; thus this assumption is not compulsory)
		\begin{equation} \label{eq: initial lambda}
			\la(0) = 1.
		\end{equation}
		\item Initial energy bounds:
		\begin{equation} \label{eq: initial energy}
		\lf\{\begin{aligned}
			& 0 < \E_1(0) < \db \ll 1, \\
			& 0 < \E_2(0) +\E_4(0) < b(0)^{10}.
		\end{aligned}\rg.
		\end{equation}
	\end{itemize}
	where $\db$ denotes a generic constant related to $b^{\ast}$ which satisfies 
	\begin{equation} \label{eq: b inf}
		\db\to0,\;\;\mbox{as}\;\; b^{\ast}\to0.
	\end{equation} 
	The propagation of regularity of $u(t)$ ensures that these bounds can be preserved in a small time interval $[0, t_1)$. This suggests given a universal large constant $K$, independent of $M$ in Lemma \ref{le: ap solution}, we may assume on $[0,t_1]$ the following bounds:
	\begin{itemize}
		\item Pointwise bounds on the modulation parameters:
		\begin{equation} \label{eq: pointwise a b}
			0 < b(t) \leq K b^{\ast}(M),
			\quad a(t) \leq \dfrac{b(t)}{\lf|\lgb(t)\rg|}.
		\end{equation}
		\item Pointwise energy bounds:
		\begin{align}
			& \E_1(t) \leq K\phsh \db, 
			\label{eq: pointwise energy 1}\\
			& \E_2(t) \leq K\phsh b(t)^2 \lf|\lgb(t)\rg|^6, 
			\label{eq: pointwise energy 2}\\
			& \E_4(t) \leq K\phsh\dfrac{b(t)^4}{\lf|\lgb(t)\rg|^2}.
			\label{eq: pointwise energy 3}
		\end{align}
	\end{itemize}
	\vspace{-.35em}
	The core of our analysis is the following proposition, which yields the contraction of the blowup regime.
	
	\begin{proposition}[Trapped regime]\label{pro: trapped regime}
		Assume that $K$ in \eqref{eq: pointwise a b}--\eqref{eq: pointwise energy 3} has been chosen large enough, independent of $M$. Then for any large enough rescaled initial time $s_0$, and any initial data $\big(\w,\la,\vart,b\big)(0)$ satisfying \eqref{eq: initial b a}, \eqref{eq: initial lambda}, \eqref{eq: initial energy}, there exists
		\begin{equation}
			a(0)=a\Big(b(0),\w(0)\Big) \in
			\lf[-\frac{b(0)}{4\lf|\lgb(0)\rg|},\,\frac{b(0)}{4\lf|\lgb(0)\rg|}\rg],
		\end{equation}
		such that \eqref{eq: decomposition} the corresponding solution to \eqref{eq: LL} satisfies for $t\in[0,t_1]$ that:		\begin{itemize}
			\item Refined bounds on the modulation parameters:
			\begin{equation} \label{eq: refined pointwise a b}
				0 < b(t) \leq \dfrac{K}{2}\phsh b^{\ast}(M),
				\quad a(t) \leq \dfrac{b(t)}{2\lf|\lgb(t)\rg|}.
			\end{equation}
			\item Refined energy bounds:
			\begin{align}
				& \E_1(t) \leq \dfrac{K}{2}\db, 
				\label{eq: refined pointwise energy 1}\\
				& \E_2(t) \leq \dfrac{K}{2}b(t)^2\lf|\lgb(t)\rg|^6, 
				\label{eq: refined pointwise energy 2}\\
				& \E_4(t) \leq K(1-\eta) \dfrac{b(t)^4}{\lf|\lgb(t)\rg|^2},
				\label{eq: refined pointwise energy 3}
			\end{align}
			where $\eta\in(0,1)$ is a universal constant independent of $M$.
		\end{itemize}
	\end{proposition}
	This proposition implies the solution $u$ is trapped in the regime \eqref{eq: refined pointwise a b}, \eqref{eq: refined pointwise energy 1}--\eqref{eq: refined pointwise energy 3} , and thus the bounds can be maintained within the lifespan of $u$. The proof is given in subsection~\ref{SS: closing the bootstrap}.
	
	\subsection{Equation for the radiation}
	Recall from \eqref{eq: decomposition} that we have the decomposition
	\begin{equation} \notag
		\hw = \wzt + \w,
	\end{equation}
	Injecting this into the vectorial formula~\eqref{eq: vec hw equation}, and applying \eqref{eq: localized error}, we obtain the equation for the radiation
	\begin{equation} \label{eq: w equation}
		\prs\w -\frac{\la_s}{\la}\varl\w +\ao\hJ\mHam\w -\bo\hJ^2\mHam\w +\f = 0,
	\end{equation}
	where
	\begin{equation} \label{eq: f term}
		\f = \mmodt +\errzt +\rest.
	\end{equation}
	The modulation vector $\mmodt$, the error $\errzt$ are given by \eqref{eq: localized error}, and the term $\rest$ contains the residue involving the phase derivative and cross terms of $\wzt$ and $\w$:
	\begin{equation} \label{eq: rest}
		\rest = \vart_s Z\R\w -\bo\hJ\lf(\w\wedge\mHam\wzt\rg) +\w\wedge\Big(\ao\mHam\wzt -\bo\tJ\mHam\wzt +\varl\phi\,\p\Big).
	\end{equation}
	We rewrite $\w$ in terms of the original coordinate $(r,t)$ by
	\begin{equation} \label{eq: relation w W}
		\W(t,r) = \w(s,y),
	\end{equation}
	and then \eqref{eq: w equation} is equivalent to	\begin{equation} \label{eq: W equation}
		\prt\W+\ao\hJ_{\la}\mHaml\W-\bo\hJ_{\la}^2\mHaml\W+\F=0,
	\end{equation}
	where
	\begin{equation} \label{eq: F}
		\hJ_\la=\big(\ez+\hW\big)\wedge,
		\quad\F=\frac{1}{\la^2}\,\f_{\la}.
	\end{equation}
	Here the rescaled vectorial Hamiltonian $\mHaml$ is defined by
	\begin{equation} \label{eq: mHaml}
		\mHaml\begin{bmatrix}
			\alpha_{\la}\\ \beta_{\la}\\ \gamma_{\la}
		\end{bmatrix}=\begin{bmatrix}
			\Ham_{\la}\alpha_{\la}\\
			\Ham_{\la}\beta_{\la}\\
			-\Delta_r\gamma_{\la}
		\end{bmatrix}
		+\frac{2(1+Z_\la)}{\la}\begin{bmatrix}
			-\pr_r\gamma_{\la}\\ 0\\
			\pr_r\alpha_{\la}+\frac{Z_\la}{r}\alpha_{\la}
		\end{bmatrix},
	\end{equation}
	where the rescaled scalar Hamiltonian is $\Ham_{\la} = -\Delta_r +V_{\la}/r^2$. The notation of $\la$ subscript is given by \eqref{eq: v lambda}, from which $Z_{\la}(r)=Z(r/\la)$, $\alpha_{\la}(t,r)=\alpha(t,r/\la)$, and so forth. Similar to \eqref{eq: Ham factori}, $\Haml$ admits the factorization $\Haml=\Asl\Al$ with
	\begin{equation} \label{eq: Haml factori}
		\Al=-\pr_r+\frac{Z_{\la}}{r},
		\quad \Asl=\pr_r+\frac{1+Z_{\la}}{r}.
	\end{equation}
	In the forthcoming analysis, we define the dominating part of $\mHaml$ by
	\begin{equation} \label{def: operator mHamlp}
		\mHamlp\begin{bmatrix}
			\alpha_{\la}\\ \beta_{\la}\\ \gamma_{\la}
		\end{bmatrix}=\begin{bmatrix}
			\Haml\alpha_{\la}\\ \Haml\beta_{\la}\\ 0
		\end{bmatrix},
	\end{equation}
	which, from \eqref{eq: Haml factori}, admits the factorization $\mHamlp=\mAsl\mAl$ with
	\begin{equation} \label{eq: mA}
		\mAl\begin{bmatrix}
			\alpha_{\la}\\ \beta_{\la}\\ \gamma_{\la}
		\end{bmatrix}=\begin{bmatrix}
			\Al\alpha_{\la}\\ \Al\beta_{\la}\\ 0
		\end{bmatrix}, \quad
		\mAsl\begin{bmatrix}
			\alpha_{\la}\\ \beta_{\la}\\ \gamma_{\la}
		\end{bmatrix}=\begin{bmatrix}
			\Asl\alpha_{\la}\\ \Asl\beta_{\la}\\ 0
		\end{bmatrix}.
	\end{equation}
	In view of the scaling $y/r=1/\la$, there holds
	\begin{equation} \label{eq: relation Ham mHam}
		\mHam\w=\la^2\mHaml\W, \quad\mHamp\w=\la^2\mHamlp\W,
	\end{equation}
	and from \eqref{eq: relation mHamp}, we have
	\begin{equation} \label{eq: relation mHamlp}
		\R\mHamlp=\mHamlp\R, \quad\R\mHaml\R=-\mHamlp.
	\end{equation}

	\subsection{Modulation equations}
	\label{SS: modulation equations}
	This subsection is devoted to derivation of the modulation equations of $(\la,\vart,a,b)$, which constitute a ODE system that determines the blowup dynamics. We will use the equation \eqref{eq: w equation} and the orthogonality conditions \eqref{eq: orthogonality}.
	
	\begin{proposition}[Modulation equations] \label{pro: modulation equation}
		Assume \eqref{eq: pointwise a b}--\eqref{eq: pointwise energy 3}. Then there hold the following estimates on the modulation parameters:
		\begin{equation} \label{eq: modulation eq}
			\lf\{\begin{aligned}
				&\,\Big|a_s +\frac{2ab}{\lgba}\Big|+\bigg|b_s +b^2\bigg(1+\frac{2}{\lgba}\bigg)\bigg|
				\les \frac{1}{\sqrt{\log M}}\bigg(\sqrt{\E_4}+\frac{b^2}{\lgba}\bigg), \\
				&\,\big|a+\vart_s\big| +\Big|b+\frac{\la_s}{\la}\Big|\les C(M)\,b^3.
			\end{aligned}\rg.
		\end{equation}
	\end{proposition}
	
	\begin{remark}
	{\rm (i)} The expression in companion with $b_s$ in the first identity is not $a^2+b^2$, different from what is suggested by $\mmodt$. This is a direct result the smallness of $a$ \eqref{eq: pointwise a b} and the flux computations \eqref{eq: localized flux computation 1}, \eqref{eq: localized flux computation 2}. {\rm (ii)} The $\sqrt{\log M}$ on the denominator is crucial in the following analysis, especially when closing the bootstrap bound of $\E_4$ and specifying an appropriate initial data of $a$.
	\end{remark}
	
	\begin{proof}{Proposition {\rm\ref{pro: modulation equation}}}
	\label{pf: modulation equation}
	We project \eqref{eq: w equation} onto $\{\ex, \ey\}$ to obtain the equations for $\alpha, \beta$, and then take their $L^2$ inner product with $\Ham\varp_M$, $\varp_M$ defined by \eqref{eq: varp M properties} respectively. Computing each term in the resulting formulas by the interpolation bounds in Appendix \ref{S: appendix A}, \eqref{eq: pointwise a b}--\eqref{eq: pointwise energy 3}, and also \eqref{eq: localized mod vec}, \eqref{eq: localized flux computation 1}, \eqref{eq: localized flux computation 2},  we obtain
	\begin{equation} \label{eq: almost modulation eq 2}
	\begin{gathered}
		\ao\phsh a_s -\bo\big(b_s\hs+\hs b^2\hs+\hs a^2\big) +\frac{2}{\lgba}\big(\ao ab-\bo b^2\big) \les \frac{1}{\sqrt{\log M}}\bigg(\sqrt{\E_4}+\frac{b^2}{\lgba^2}\bigg), \\
		\bo\phsh a_s +\ao\big(b_s\hs+\hs b^2\hs+\hs a^2\big) +\frac{2}{\lgba}\big(\ao ab+\bo b^2\big) \les \frac{1}{\sqrt{\log M}}\bigg(\sqrt{\E_4}+\frac{b^2}{\lgba^2}\bigg), \\
		\big|a+\vart_s\big| +\Big|b+\frac{\la_s}{\la}\Big| =C(M)\,b^3.
	\end{gathered}
	\end{equation}
	Then \eqref{eq: modulation eq} follows by simple cancellations. For more details, reader can refer to \cite{2011Merle_SchMapBlowup}. Here we only compute the first line in \eqref{eq: almost modulation eq 2}. We first define 
	\begin{equation} \label{eq: U}
		U(t) =|a_s|+\big|b_s\hs+\hs b^2\hs+\hs a^2\big|+\big|a+\vart_s\big|+\Big|b+\frac{\la_s}{\la}\Big|.
	\end{equation} 
	Then taking the inner product of the first component of \eqref{eq: w equation} with $\Ham\varp_M$, we obtain the system
	\begin{equation} \label{eq: modulation expand}
	\begin{aligned}
		0 & = \big(\prs\alpha,\Ham\varp_M\big)
			-\frac{\la_s}{\la}\big(\varl\alpha,\Ham\varp_M\big)\\
		&\quad +\ao\Big((\hJ\mHam\w)^{\first},\Ham\varp_M\Big)
			-\bo\Big((\hJ^2\mHam\w)^{\first},\Ham\varp_M\Big)\\
		&\quad +\big(\rest,\Ham\varp_M\big) +\big(\mmodt,\Ham\varp_M\big) +\big(\errzt,\Ham\varp_M\big).
	\end{aligned}
	\end{equation}
	The time derivative term vanishes thanks to \eqref{eq: orthogonality}. For the second term, by commuting $\Ham, \varl$, we split it into
	\begin{equation} \label{eq: modulation expand term 2}
		\big(\varl\alpha,\Ham\varp_M\big)
		=\big(\varl\Ham\alpha,\varp_M\big)+2\big(\Ham\alpha, \varp_M\big)-\bigg(\frac{\varl V}{y^2}\alpha,\varp_M\bigg).
	\end{equation}
	Again using \eqref{eq: orthogonality}, the second term vanishes. From \eqref{eq: ip bound 3}, \eqref{eq: pointwise energy 3}, we have
	\begin{equation} \notag
		\big|\big(\varl\Ham\alpha,\varp_M\big)\big|
		\les C(M) \bigg(\int\frac{\lf|\py\Ham\alpha\rg|^2}{1+y^4} \bigg)^{\frac{1}{2}}
		\les C(M)\sqrt{\E_4},
	\end{equation}
	where $C(M)$ is a universal constant appearing in \eqref{eq: ip bound 3} induced by the coercivity \eqref{eq: Ham2 coercivity}. The constants in the following formulas are similar. For the last term in \eqref{eq: modulation expand term 2}, an analogous estimate holds, giving the bound
	\begin{equation} \label{eq: modulation comp 1}
		\Big|\frac{\la_s}{\la}\Big|
		\,\big|\big(\varl\alpha,\Ham\varp_M\big)\big|
		\les C(M)\phsh b \sqrt{\E_4}.
	\end{equation}
	For the third term in \eqref{eq: modulation expand}, from \eqref{eq: mHam},
	\begin{equation} \label{eq: modulation comp 2}
	\begin{aligned}
		\Big|\Big((\hJ\mHam\w)^{\first},\Ham\varp_M\Big)\Big|
		& \les \Big|\Big((1+\h\gamma)\Ham\beta,\Ham\varp_M\Big)\Big|\\
		&\quad +\bigg|\bigg(\h\beta\Big(\hs-\hs\Delta\gamma +2\onez\Big(\py\hs+\hs\frac{Z}{y}\Big)\alpha \Big),\Ham\varp_M\bigg)\bigg|,
	\end{aligned}
	\end{equation}
	where by commuting $\Ham$, we split the first one into
	\begin{equation} \notag
	\begin{aligned}
		\Big|\Big((1+\h\gamma)\Ham\beta,\Ham\varp_M\Big)\Big|
		& \les \Big|\Big((1+\h\gamma)\Ham^2\beta, \varp_M\Big)\Big| \\
		& \quad +\Big|\Big(\py\h\gamma\phsh\py\Ham\beta,\varp_M\Big)\Big| +\Big|\Big(\Delta\h\gamma\Ham\beta,\varp_M\Big)\Big|,
	\end{aligned}
	\end{equation}
	Using \eqref{eq: ip bound 20}, \eqref{eq: ip bound 22}, we have	\begin{equation} \notag
		\Big|\Big((1+\h\gamma)\Ham^2\beta, \varp_M\Big)\Big|
		\les \big\|1\hs+\hs\h\gamma\big\|_{L^\infty}\big\|\Ham^2\beta\big\|_{L^2}\|\varp_M\|_{L^2}
		\les \sqrt{\log M}\sqrt{\E_4}.
	\end{equation}
	For the remaining, by \eqref{eq: ip bound 3}, \eqref{eq: ip bound 16}, there holds the better bound
	\begin{equation} \notag
	\begin{aligned}
		& \Big|\Big(\py\h\gamma\phsh\py\Ham\beta,\varp_M\Big)\Big| +\Big|\Big(\Delta\h\gamma\Ham\beta,\varp_M\Big)\Big| \\
		&\quad \les C(M)\phsh \bigg\{ \Big\|\frac{\py\h\gamma}{y}\Big\|_{L^\infty} \Big\|\frac{\py\Ham\beta}{y(1+\lf|\log y\rg|)}\Big\|_{L^2} \\
		&\qquad\qquad\qquad + \Big\|\frac{\Delta\h\gamma}{y(1+y)(1+\lf|\log y\rg|)}\Big\|_{L^2} \, \Big\|\frac{\Ham\beta}{y(1+y)(1+\lf|\log y\rg|)}\Big\|_{L^2} \bigg\} 
		\les C(M)b\sqrt{\E_4}.
	\end{aligned}
	\end{equation}
	The other term on the RHS of \eqref{eq: modulation comp 2} can be treated similarly. Recall from the constraint \eqref{eq: component constraint} that $|\ez+\hw|=1$, which implies $|\hJ\vv|\leq |\vv|$ for any vector $\vv$ under the Frenet basis. Thus the fourth term in \eqref{eq: modulation expand} shares the same bound as above, so that
	\begin{equation} \label{eq: modulation comp 3}
		\Big|\Big((\hJ\mHam\w)^{\first},\Ham\varp_M\Big)\Big| +\Big|\Big((\hJ^2 \mHam\w)^{\first},\Ham\varp_M\Big)\Big|
		\les \Big(\sqrt{\log M}+ C(M)b\Big) \sqrt{\E_4}.
	\end{equation}
	We consider the term involving $\rest$ in \eqref{eq: modulation expand}. From \eqref{eq: rest} and Cauchy-Schwartz, we get
	\begin{equation}  \label{eq: modulation comp expand 4}
	\begin{aligned}
		\big|\big(\rest,\Ham\varp_M\big)\big|
		& \les C(M) \bigg\{\big\|\vart_s Z \R\w\big\|_{L^2(y\leq2M)} +\big\|\hJ\big(\w\wedge\mHam\wzt\big)\big\|_{L^2(y\leq2M)} \\
		&\qquad\qquad\qquad\quad +\Big\|\w\wedge\Big(\ao\mHam\wzt-\bo\tJ\mHam\wzt +\varl\phi\phsh\p\Big)\Big\|_{L^2(y\leq2M)} \bigg\}.
	\end{aligned}
	\end{equation}
	By \eqref{eq: ip bound 3}, the phase derivative term is bounded by
	\begin{equation} \notag
		\big\|\vart_s Z\R\w\big\|_{L^2(y\leq2M)}
		\les C(M)\big(b +U(t)\big)\sqrt{\E_4}.
	\end{equation}
	Similarly, from the construction of $\wzt$, the rest of the terms in \eqref{eq: modulation comp expand 4} are bounded by
	\begin{equation} \notag
	\begin{aligned}
		&\quad \big\|\hJ\big(\w\wedge\mHam\wzt\big)\big\|_{L^2(y\leq2M)} +\Big\|\w\wedge\Big(\ao\mHam\wzt-\bo\tJ\mHam\wzt +\varl\phi\,\p\Big)\Big\|_{L^2(y\leq2M)}\\
		& \les 
		 C(M)\sqrt{\E_4} \,\bigg( \big\|\mHam\wzt\big\|_{L^\infty(y\leq2M)} +\big\|\ao\mHam\wzt-\bo\tJ\mHam\wzt +\varl\phi\phsh\p\big\|_{L^\infty(y\leq2M)} \bigg) \\
		& \les C(M)\big(b+U(t)\big)\sqrt{\E_4},
	\end{aligned}
	\end{equation}
	and hence
	\begin{equation} \label{eq: modulation comp 4}
		\big|\big(\rest,\Ham\varp_M\big)\big|\les C(M)\big(b+U(t)\big)\sqrt{\E_4}.
	\end{equation}
	It remains to consider the contributions of $\mmodt$ and $\errzt$ in \eqref{eq: modulation expand}. From \eqref{eq: localized mod vec}, we observe on the support of $\varp_M$, which is $\{y\leq 2M\}$, that $\mmodt$ is given by
	\begin{equation} \notag
	\begin{aligned}
		\mmodt
		& = a_s\phsh\varp_{1,0} +\big(b_s\hs+\hs a^2\hs+\hs b^2\big)\varp_{0,1} \\
		&\quad -\Big(b+\frac{\la_s}{\la}\Big)\varl\phi\,\ex +(a+\vart_s)\varl\phi\,\ey + C(M)b\, U(t),
	\end{aligned}
	\end{equation}
	which together with \eqref{eq: varp M properties} yields that
	\begin{equation} \label{eq: modulation comp 5}
	\begin{aligned}
		\big(\mmodt^{\first}\hs\hs\hs,\Ham\varp_M\big)
		& = a_s \big(\varp_{1,0}^{\first},\Ham\varp_M\big)
			+\big(b_s\hs+\hs a^2\hs+\hs b^2\big)\big(\varp_{0,1}^{\first},\Ham\varp_M\big) + C(M)b\,U(t) \\
		& = \frac{1}{\ao^2+\bo^2}\Big[\ao\phsh a_s -\bo\big(b_s\hs+\hs a^2\hs+\hs b^2\big)\Big]\big(\varl\phi,\varp_M\big) + C(M)b\,U(t).
	\end{aligned}
	\end{equation}
	Injecting \eqref{eq: modulation comp 1}, \eqref{eq: modulation comp 2}, \eqref{eq: modulation comp 3}, \eqref{eq: modulation comp 4}, \eqref{eq: modulation comp 5} into \eqref{eq: modulation expand}, applying the flux computation \eqref{eq: flux computation} and also the non-degeneracy \eqref{eq: varp M properties}, we obtain
	\begin{equation} \notag
	\begin{aligned}
		& \frac{1}{\ao^2+\bo^2}\Big[\ao\phsh a_s -\bo\big(b_s\hs+\hs a^2\hs+\hs b^2\big) +\frac{2}{\lgba}\big(\ao ab\hs-\hs\bo b^2\big)\Big] \\
		& \qquad\qquad = \frac{1}{(\varl\phi,\varp_M)}\bigg[ \Big(\sqrt{\log M}+C(M)\big(b\hs+\hs U(t)\big)\Big)\sqrt{\E_4} +\frac{b^2}{\lgba^2} \bigg],
	\end{aligned}
	\end{equation}
	where the term involving $b+U(t)$ is negligible due to the smallness of $b$ and the absorption to the LHS. This yields the desired first line in \eqref{eq: almost modulation eq 2}. The other two identities can be obtained by similar computations.
	\end{proof}
	\begin{remark} \label{re: non convergence of theta}
		The modulation equations~\eqref{eq: modulation eq} build up a closed ODE system for the geometrical parameters $(\la,\vart,a,b)$, determining their asymptotic behavior as the rescaled time $s\to+\infty$, see \eqref{eq: b asymptotics}. However, a direct integration on \eqref{eq: modulation eq} can only gives
		\begin{equation} \label{eq: rough a b asymptotics}
			b(s)\sim \frac{1}{s}, \quad
			a(s)\sim (\log s)^2 +C_1, \quad
			\vart(s)\leq \int_{s_0}^{s}|a|\,d\sigma+C_2,
		\end{equation}
		insufficient to ensure the smallness of $a(s)$ and the convergence of the phase \eqref{eq: Theta}. To overcome this, a refined bound for $a(s)$ will be derived in Section {\rm \ref{S: des on singularity formation}}.
	\end{remark}

	\section{Energy method}\label{S: energy method}
	In this section, we consider the fourth order energy $\E_4$ of $\w$, and derive the mixed energy identity/Morawetz estimate, which is the core of our analysis. It actually helps us close the refined bounds in Proposition \ref{pro: trapped regime}. The overall strategy is called the energy method, see \cite{2011Merle_SchMapBlowup} for related arguments.
	
	\subsection{The energy identity}
	\label{SS: the Energy Identity}
	In this subsection, we study the evolution of $\E_4$ by investigating its equivalent $\int|\hJ_{\lambda}\mHaml\W_2|^2$ (see \eqref{eq: plain energy}). First, we define the suitable second derivative of $\W$:
	\begin{equation} \label{eq: W2}
		\W_2=\hJ_\la\mHaml\W.
	\end{equation} 
	Recalling \eqref{eq: relation w W}, we let 
	\begin{equation} \label{eq: w2}
		\w_2=\hJ\mHam\w=\lambda^2\W_2
	\end{equation} 
	be its equivalent acting on the rescaled variable $(s,y)$. Moreover, to ease the notation, we denote the rotation with respect to $\hw$ by 
	\begin{equation}
		\Rw=\hw\wedge=\hW\wedge,
	\end{equation}
	and as a consequence,
	\begin{equation} \label{eq: hJ lambda}
		\hJ_{\la}=\R+\Rw=\hJ.
	\end{equation}
	Using \eqref{eq: hJ lambda}, \eqref{eq: W equation}, we compute the equation for $\W_2$:
	\begin{equation} \label{eq: W2 equation}
		\lf\{\begin{aligned}
		\prt\W_2 &= -\ao\hJ_\la\mHaml\W_2+\bo\hJ_\la\mHaml\hJ_\la\W_2-\hJ_\la\mHaml\F+\R\big[\prt,\mHaml\big]\W+\Q_1,\\
		\Q_1 &= \prt\hW\wedge\mHaml\W+\Rw\big[\prt,\mHaml\big]\W.
	\end{aligned}\rg.
	\end{equation}
	In consideration of the quadratic nature of the third component, the dominate part of $\W$ is actually
	\begin{equation} \notag
		\Wp=-\R^2\W=\big[\alpha_\la,\beta_\la,0\big]^T,
		\quad \W^{\third} =\W-\Wp=\gamma_\la\phsh\ez.
	\end{equation}
	Similarly, we decompose $\W_2$ into
	\begin{equation}	 \label{eq: W2 decomposition}
		\W_2^0=\R\mHaml\W^\perp,
		\quad \W_2^1=\W_2-\W_2^0,
	\end{equation}
	and also $\w_2^0=\R\mHam\w^\perp, \w_2^1=\w_2-\w_2^0$ accordingly. From \eqref{eq: W equation}, $\Wp$ satisfies the equation
	\begin{equation} \label{eq: Wp equation}
		\lf\{\begin{aligned}
		\prt\W^\perp &= -\ao\W_2^0+\bo\R\W_2^0-\F^\perp+\Q_2^1,\\
		\Q_2^1 &= \ao\R^2\W_2^1-\bo\R^2\Rw\W_2^0-\bo\R^2\hJ_\la \W_2^1,
		\end{aligned}\rg.
	\end{equation}
	while $\W_2^0$ satisfies the equation
	\begin{equation} \label{eq: W20 equation}
		\lf\{\begin{aligned}
		\prt\W_2^0 &= -\ao\R\mHaml\W_2^0-\bo\mHaml\W_2^0- \R\mHaml\F^\perp+\Q_2^2,\\
		\Q_2^2 &= \R\big[\prt,\mHaml\big]\Wp+\R\mHaml\Q_2^1.
		\end{aligned}\rg.
	\end{equation}
	We define two more functions, whose behaviors play importance roles on the forthcoming energy estimate:
	\begin{equation} \label{eq: L G}
		G(t,r)=\frac{b(\varl V)_\la}{\la^2 r^2},
		\quad L(t,r)=\frac{b(\varl Z)_\la}{\la^2 r}.
	\end{equation}
	As we shall see later, $G$ appears in an unsigned quadratic term, and $L$, containing the information on the structure of $\mA, \mAs$, will be used to obtain the mixed energy identity. We now derive the plain energy identity at the level of $\E_4$.
	\begin{lemma}[Plain energy identity] \label{le: plain energy}
	Under the above definitions, there holds
	
	\begin{equation} \label{eq: plain energy}
	\begin{aligned}
		\frac{1}{2}\frac{d}{dt}\int\big|\hJ_\la\mHaml\W_2\big|^2
		=& -\ao\hs\int\hJ_\la\mHaml\W_2\cdot\big(\hJ_\la \mHaml\big)^2\W_2 \\
		& +\bo\hs\int\hJ_\la\mHaml\W_2\cdot\big(\hJ_\la \mHaml\big)^2\hJ_\la\W_2 \\
		& +\int\hJ_\la\mHaml\W_2\cdot\big(\hJ_\la\mHaml\big) \Q_1-\hs\int\hJ_\la\mHaml\W_2\cdot\big(\hJ_\la \mHaml\big)^2\F \\
		& +\Q_3+\hs\int\R\mHaml\W_2^0\cdot\Big[-\mHamlp\big(G \W^\perp\big)+G\R\W_2^0\Big],
	\end{aligned}
	\end{equation}
	where
	\begin{equation} \label{eq: Q3}
	\begin{aligned}
		\Q_3 =& -\Big(b+\frac{\la_s}{\la}\Big)\int\R\mHaml \W_2^0\cdot\bigg[-\mHamlp\bigg(\frac{(\varl V)_\la}{\la^2 r^2}\W^\perp\bigg) +\frac{(\varl V)_\la}{\la^2 r^2}\R\W_2^0\, \bigg] \\
		& +\int\R\mHaml\W_2^0\cdot\bigg(\R\mHaml\R\big[\prt,\mHaml\big]\gamma\,\ez+\Rw\mHaml\R\big[\prt,\mHaml\big]\W \\
		& \qquad\qquad+\prt\hW\wedge\mHaml\W_2^0+\Rw\big[\prt,\mHaml\big]\W_2^0+\big[\prt,\hJ_\la\mHaml\big]\W_2^1\bigg) \\
		& +\int \R\mHaml\W_2^1\cdot\Big(\hJ_\la \mHaml \R \big[\prt,\mHaml\big]\W+\big[\prt,\hJ_\la\mHaml\big]\W_2\Big) \\
		& +\int\Rw\mHaml\W_2\cdot\Big(\hJ_\la\mHaml\R \big[\prt,\mHaml\big]\W+\big[\prt,\hJ_\la\mHaml\big]\W_2\Big).
	\end{aligned} 
	\end{equation}
	\end{lemma}
	
	\begin{remark} 
	{\rm (i)} $\Q_3$ will be proved to be small errors in Proposition~\ref{pro: mixed energy estimate}. 
	{\rm (ii)} The majority of the RHS of \eqref{eq: plain energy} can be estimated by the interpolation bounds listed in Lemma \ref{S: appendix A} directly, but the last quadratic term is unsigned, and can only be controlled via \eqref{eq: LZ LV}, \eqref{eq: ip bound 3}, \eqref{eq: ip bound 29}, \eqref{eq: ip bound 31}, by the following rough bound
	\vspace{-.3em}
	\begin{equation} \label{eq: uncontrollable term}
	\begin{aligned}
		& \int\R\mHaml\W_2^0\cdot\Big[-\mHamlp\big(G\W^\perp\big)+G\R \W_2^0\,\Big]\\
		\les &\,\frac{b}{\la^8}\bigg(\int\big|\R\mHam\w_2^0\big|^2\bigg)^{\frac{1}{2}} \bigg(\int\Big|\mHamp\Big(\frac{\w^\perp}{1+y^4}\Big)\Big|^2 +\frac{|\w_2^0|^2}{1+y^8} \bigg)^{\frac{1}{2}} \les \frac{b}{\la^8}\Ebb.
	\end{aligned}
	\end{equation}
	It is of the same size as the leading terms in \eqref{eq: plain energy} (the first two lines on the RHS), and the implicit constant in \eqref{eq: uncontrollable term} will deteriorate the forthcoming energy estimate. To treat this, we introduce the Morawetz type formula in the next subsection.
	\end{remark}
	
	\begin{proof}{Lemma {\rm\ref{le: plain energy}}} 
	\label{pf: plain energy}
	Injecting \eqref{eq: W2 equation} into the time derivative of $\int|\hJ_{\la}\mHaml\W_2|^2$, we obtain an analogue of the RHS of \eqref{eq: plain energy}. The last line of the latter is given by \eqref{eq: uncontrollable originate} as below, different from that of \eqref{eq: plain energy}. Thus it remains to do some algebraic computation. Using \eqref{eq: hJ lambda}, \eqref{eq: W2 decomposition}, we compute
	\begin{equation} \label{eq: uncontrollable originate}
	\begin{aligned}
		&\quad \int\hJ_\la\mHaml\W_2\cdot
			\Big(\hJ_\la\mHaml\R\big[\prt,\mHaml\big]\W+\big[\prt,\hJ_\la\mHaml\big]\W_2 \Big) \\
		& = \int\R\mHaml\W_2^0 \cdot \Big(\hJ_\la\mHaml\R\big[\prt,\mHaml\big]\W +\big[\prt,\hJ_\la\mHaml\big]\W_2\Big) \\
		&\quad +\int\R\mHaml\W_2^1 \cdot\Big(\hJ_\la\mHaml\R \big[\prt,\mHaml\big]\W +\big[\prt,\hJ_\la\mHaml\big]\W_2 \Big) \\
		&\quad +\int\Rw\mHam\W_2 \cdot\Big(\hJ_\la\mHaml\R \big[\prt,\mHaml\big]\W +\big[\prt,\hJ_\la\mHaml\big]\W_2 \Big).
	\end{aligned}
	\end{equation}
	The first term on the RHS of \eqref{eq: uncontrollable originate} can be further splited into
	\begin{equation} \label{eq: uncontrollable small originate}
	\begin{aligned}
		& \int\R\mHaml\W_2^0 \cdot \Big(\hJ_\la\mHaml\R\big[\prt,\mHaml\big]\W +\big[\prt,\hJ_\la\mHaml\big]\W_2 \Big)\\
		=&\, \int\R\mHaml\W_2^0 \cdot
			\Big(\R\mHaml\R\big[\prt,\mHaml\big]\W +\big[\prt,\hJ_\la\mHaml\big]\W_2^0\Big) \\
		& +\int\R\mHaml\W_2^0\cdot \Big(\Rw\mHaml\R\big[\prt,\mHaml\big]\W +\big[\prt,\hJ_\la\mHaml\big]\W_2^1\Big),
	\end{aligned}
	\end{equation}
	where the involved commutators are
	\begin{equation} \notag
		\big[\prt,\hJ_\la\mHaml\big]\W
		=\R\big[\prt,\mHaml\big]\W+\prt\hW\wedge\mHaml\W+\Rw\big[\prt,\mHaml\big]\W,
	\end{equation}
	and
	\begin{equation} \label{eq: commutator2}
	\begin{aligned}
		\big[\prt,\mHaml\big]\W \hs &= -\hs\frac{\la_s}{\la^3}
		\Bigg\{\bigg[\frac{(\varl V)_\la}{r^2}\alpha_\la-\frac{2}{\la}\Big(\hs1\hs\hs+\hs\hs Z\hs\hs+\hs\hs\varl Z\hs\Big)_\la\pr_r\gamma_\la\,\bigg]\ex
		+\bigg[\frac{(\varl V)_\la}{r^2}\beta_\la\bigg]\ey\\
		&\qquad\quad\;\; +\hs\bigg[\frac{2}{\la}\Big(\hs1\hs\hs+\hs\hs Z\hs\hs+\hs\hs\varl Z\hs\Big)_\la\pr_r\alpha_\la +\hs\frac{2}{\la r}\Big(\hs Z\hs\hs+\hs\hs Z^2\hs\hs+\hs\hs\varl Z\big(1\hs+\hs 2Z\big)\hs\Big)_{\la}\hs\hs\alpha_\la\hs\bigg]\ez\hs\Bigg\}.
	\end{aligned}
	\end{equation}
	From \eqref{eq: relation mHamlp}, \eqref{eq: commutator2}, we have
	\begin{equation} \label{eq: commutator3}
	\lf\{\begin{aligned}
		&\,\R\mHaml\R\big[\prt,\mHaml\big]\W^\perp=\frac{\la_s}{\la}\,\mHamlp\bigg(\frac{(\varl V)_\la}{\la^2 r^2} \W^\perp\bigg)\approx -\mHamlp\big(G\Wp\big),\\
		&\,\R\big[\prt,\mHaml\big]\W_2^0=-\frac{\la_s}{\la}\,\frac{(\varl V)_\la}{\la^2 r^2} \R\W_2^0\approx G\R\W_2^0.
	\end{aligned}\rg.
	\end{equation}
	Using these, we rewrite the first term on the RHS of \eqref{eq: uncontrollable small originate} as
	\begin{equation} \label{eq: uncontrollable small originate 2}
	\begin{aligned}
		& \int\R\mHam_\la\W_2^0 \cdot \Big(\R\mHaml\R\big[\prt,\mHaml\big]\W+\big[\prt,\hJ_\la\mHaml\big]\W_2^0\Big) \\
		=&\, \hs\int\R\mHaml\W_2^0\cdot\Big[-\mHamlp\big(G \W^\perp\big)+ G\R\W_2^0\,\Big] \\
		& -\Big(b+ \frac{\la_s}{\la}\Big)\int\R\mHaml\W_2^0 \cdot\bigg[-\mHamlp\bigg(\frac{(\varl V)_\la}{\la^2 r^2} \W^\perp\bigg)+\frac{(\varl V)_\la}{\la^2 r^2}\R \W_2^0\bigg]  \\
		& +\int\R\mHaml\W_2^0 \cdot\Big(\R\mHaml\R \big[\prt,\mHaml\big] \gamma\,\ez+\prt\hW\wedge\mHaml\W_2^0+\Rw\big[\prt,\mHaml\big]\W_2^0 \Big).
	\end{aligned}
	\end{equation}
	Inserting this into \eqref{eq: uncontrollable small originate} yields the last line of \eqref{eq: plain energy}. By collecting the rest, we obtain \eqref{eq: Q3}, and thus complete the proof.
	\end{proof}
	
	\subsection{Morewatz correction}
	\label{SS: morawetz correction}
	The aim of this subsection is to modify the plain energy \eqref{eq: plain energy}, by adding an extra Morawetz term to the energy functional. The resulting formula is called the mixed energy/Morawetz identity, which, containing no uncontrollable terms, is appropriate to close the bootstrap.    It is notable that the newly-added Morawetz term heavily rely on the coefficients $\ao, \bo$ in \eqref{eq: LL}. 
	
	\medskip
	Indeed, to characterize the relative size of the absolute values of $\ao, \bo$, we define their ratio by
	\begin{equation} \notag
		\oo=\bigg(\frac{\ao}{\bo}\bigg)^2.
	\end{equation}
	Recall that  $\ao\in\RR$ and $\bo>0$, we see $\oo$ is well-defined and takes value in $[0,+\infty)$. We further introduce the following conditional parameters depending on $\oo$:
	\begin{equation} \label{eq: ks}
		\ka(\oo)=\begin{cases}
			1,\quad \oo\geq 1, \\0, \quad \oo<1,
		\end{cases}
		\quad \kb(\oo)=1-\ka(\oo),
		\quad \kk(\oo)=\ka(\oo)-\kb(\oo).
	\end{equation}
	It is easy to see $\kb\in\{0,1\}$ and $\kk\in\{-1,1\}$. Now we introduce the extra term, and derive the Morawetz type formula. The results are collected in the following lemma.
	
	\begin{lemma}[Morawetz correction]\label{le: morawetz}
	The desired Morawetz term is \vspace{-.3em}
	\begin{equation} \label{eq: morawetz}
	\begin{aligned}
		\m(t)= & \,c_1\hs\int\mHaml\W_2^0\cdot G\W^\perp +c_2\hs\int\R\mAl\W_2^0\cdot L\W_2^0 \\
		& +c_3\hs\int\R\mHaml\W_2^0\cdot G\W^\perp -c_4\hs\int\mAl\W_2^0\cdot L\W_2^0,
	\end{aligned}
	\end{equation}
	where the coefficient in front of the integrals are \vspace{-.3em}
	\begin{equation} \label{eq: cccc}
	\begin{aligned}
		\big(c_1,c_2,&\, c_3,c_4\big)=\\
		&\bigg( \frac{\ao}{\ao^2+\bo^2},\,\frac{2\ao(\ao^2-\bo^2)}{(\ao^2+\bo^2)(\kk\ao^2+\bo^2)},\,\frac{\bo}{\ao^2+\bo^2},\,\frac{2\ao^2\bo(1+\kk)}{(\ao^2+\bo^2)(\kk\ao^2+\bo^2)} \bigg).
	\end{aligned}
	\end{equation}
	Moreover, there holds the boundedness 
	\begin{equation} \label{eq: morawetz bounded}
		\big|\m(t)\big|\les \frac{\db}{\la^6}\Ebb.
	\end{equation}
	Its time derivative (Morawetz type identity) is given by
	\begin{equation} \label{eq: morawetz eq}
	\begin{aligned}
		\frac{d}{dt}\,\m(t) & =
			\int\R\mHaml\W_2^0\cdot\Big[-\mHamlp\big(G\W^\perp\big)+G\R \W_2^0\,\Big] +\frac{b}{\la^8}\,O\bigg(\frac{\E_4}{\sqrt{\log M}}+\bflog\bigg) \\
			&\quad -2\ao c_2\ka\int\mHaml\W_2^0\cdot L\mAl\W_2^0 +2\ao c_2\kb\int\mAl\mHaml\W_2^0\cdot L\W_2^0.
	\end{aligned}
	\end{equation}
	\end{lemma}
	
	\begin{remark}
	{\rm (i)} The first term on the RHS of \eqref{eq: morawetz eq} is just the uncontrollable quadratic term \eqref{eq: uncontrollable term}. 
	{\rm (ii)} The coefficients $c_i$ in \eqref{eq: cccc} are well-defined. Indeed, by the positivity of the Gilbert damping $\bo>0$, we see $\ao^2+\bo^2>0$. Moreover, the vanishing $\kk\ao^2+\bo^2=0$ implies $\kk=-1$ and $\ao^2=\bo^2$, from which $\oo=\ao^2/\bo^2=1$, and thus $\kk=1$, contradiction. This ensures the denominators in the definitions of $c_i$ are non-zero.
	\end{remark}
	Assume Lemma \ref{le: morawetz} holds for now. We subtract \eqref{eq: morawetz eq} from the plain energy \eqref{eq: plain energy}, and obtain the following mixed energy/Morawetz identity
	\begin{equation} \label{eq: mixed energy identity}
	\begin{aligned}
		&\quad \frac{d}{dt}\bigg\{\frac{1}{2}\int\big|\hJ_\la\mHaml\W_2\big|^2-\m(t)\bigg\} \\
		& = -\ao \int \hJ_\la \mHaml \W_2 \cdot\big(\hJ_\la \mHaml\big)^2 \W_2+ \bo \int \hJ_\la \mHaml \W_2 \cdot\big(\hJ_\la \mHaml\big)^2 \hJ_\la \W_2 \\
		&\quad +2\ao c_2\ka\int\mHaml\W_2^0\cdot L\mAl\W_2^0 - 2\ao c_2\kb\int\mAl\mHaml\W_2^0\cdot L\W_2^0 \\
		&\quad +\Q_3 +\int \hJ_\la \mHaml \W_2 \cdot \big(\hJ_\la \mHaml\big)\Q_1 -\int\hJ_\la\mHaml\W_2\cdot\big(\hJ_\la\mHaml\big)^2\F \\
		&\quad +\frac{b}{\la^8}\,O\bigg(\frac{\E_4}{\sqrt{\log M}} +\bflog\bigg).
	\end{aligned}
	\end{equation}
	
	Note that the uncontrollable term has been cancelled. The identity \eqref{eq: mixed energy identity} is the key to derive the mixed energy/Morawetz estimate. Now before giving a proof for Lemma~\ref{le: morawetz}, we introduce the following brief lemma on the structure of $\mA,\mAs$, and their relation with the functions $G, L$ \eqref{eq: L G}.
	
	\begin{lemma}[Structure of $\mA, \mAs$]\label{le: structure}
	There holds the identities
	\begin{gather}
		\int\mAl\mHaml\W_2^0\cdot L\W_2^0 +\int\mHaml\W_2^0\cdot L\mAl\W_2^0 =\int\mHaml\W_2^0\cdot G\W_2^0,
		\label{eq: structure identity 1}\\
		\int\R\mAl\mHaml\W_2^0\cdot L\W_2^0 =\int\R\mHaml\W_2^0\cdot G\W_2^0,
		\label{eq: structure identity 2}\\
		\int\R\mHaml\W_2^0\cdot L\mAl\W_2^0 =0,
		\label{eq: structure identity 3}
	\end{gather}
	and the non-positivity
	\begin{equation} \label{eq: 43 negativity}
		\int \mHaml \W_2^0 \cdot L \mAl \W_2^0 \leq 0.
	\end{equation}
	\end{lemma}
	
	\begin{proof}{Lemma \ref{le: structure}}
	Recalling the function $V$ \eqref{eq: V}, we see $\varl V=\varl(\varl Z+Z^2)$. Applying \eqref{eq: mA}, we compute for any radially symmetric function $f(r)$ with $r=\la y$ that
	\begin{equation} \label{eq: Structure comp 1}
	\begin{aligned}
		\mAsl(Lf)&+L\mAl f=\bigg(\pr_r L+\frac{2Z_\la+1}{r}L\bigg)f\\
		=&\;\frac{b}{\la^4}\bigg[\py\bigg(\frac{\varl Z}{y}\bigg)+\frac{(2Z+1)\varl Z}{y^2}\bigg]\,f
		=\,\frac{b\big(\varl(\varl Z+Z^2)\big)(y)}{\la^4 y^2}\,f =Gf.
	\end{aligned}
	\end{equation}
	This together with the commutativity $\mAl\R=\R\mAl$ leads to
	\begin{align} 
		\int\mAl\mHaml\W_2^0\cdot L\W_2^0 &= -\int\mHaml\W_2^0\cdot L\mAl\W_2^0 +\int\mHaml\W_2^0\cdot G\W_2^0, 
		\label{eq: Structure comp 2} \\
		\int\R\mAl\mHaml\W_2^0\cdot L\W_2^0 &= -\int\R\mHaml\W_2^0\cdot L\mAl\W_2^0 +\int\R\mHaml\W_2^0\cdot G \W_2^0,
		\label{eq: Structure comp 3}
	\end{align}
	where \eqref{eq: Structure comp 2} is just \eqref{eq: structure identity 1}. Moreover, by a similar computation on $\mAl(Lf)+L\mAsl f$ as \eqref{eq: Structure comp 1}, and the fact $L(t,r)<0$ for any $r>0$, we have
	\begin{equation}
	\begin{aligned}
		&\quad \int\mHaml\W_2^0\cdot L\mAl\W_2^0 =\int\mAl\W_2^0\cdot\mAl\big(L\mAl\W_2^0\big)\\
		& = \int\mAl\W_2^0\cdot\bigg[-L\mAsl\mAl\W_2^0+\bigg(\frac{2Z_\la+1}{r}L-\pr_r L\bigg)\mAl\W_2^0\bigg] \\
		& = \int\mAl\W_2^0\cdot\bigg(-L\mHaml\W_2^0+\frac{2L}{r}\mAl\W_2^0\bigg)
		= \int\frac{L}{r}\,\big|\mAl\W_2^0\big|^2 \leq 0,
	\end{aligned}
	\end{equation}
	which is the non-positivity \eqref{eq: 43 negativity}. Note here we have used $\mHaml\W_2^0\cdot X=\mHamlp\W_2^0\cdot X$, for any vector $X$ under Frenet basis with the third component $X^{\third}=0$. An analogous argument can be formulated to show the vanishing
	\begin{equation} \notag
		\int\R\mHaml\W_2^0\cdot L\mAl\W_2^0=\int\frac{L}{r}\,\R\mAl\W_2^0\cdot\mAl\W_2^0=0,
	\end{equation} 
	which is \eqref{eq: structure identity 3}. Combining this with  \eqref{eq: Structure comp 3} yields \eqref{eq: structure identity 2}. This completes the proof.
	\end{proof}
	\vspace{.5em}
	The structural bonus illutrated in Lemma \ref{le: structure} will be found essential for the following analysis. We now use it to prove Lemma \ref{le: morawetz}.
	
	\begin{proof}{Lemma \ref{le: morawetz}}
	We first show the boundedness \eqref{eq: morawetz bounded}, and then prove \eqref{eq: morawetz eq} by computing the time derivatives of each term in \eqref{eq: morawetz}, formulating the uncontrollable term, and finally estimating the errors via  Appendix \ref{S: appendix A}.
	
	\medskip
	\step{1}{Control of $\m(t)$} From the lossy logarithmic bound \eqref{eq: ip bound 5}, we have 
	\begin{equation} \notag
	\begin{aligned}
		\big|\m(t)\big|
		& \les \frac{b}{\la^6}\bigg(\int|\mHam\w_2^0|\,\frac{|\w^\perp|}{1+y^4}+\int|\mA\w_2^0|\,\frac{|\w_2^0|}{1+y^3}\bigg) \\
		& \les \frac{b}{\la^6}\bigg(\int|\mHam\w_2^0|^2+\int\frac{|\mA\w_2^0|^2}{y^2(1+y^2)}\bigg)^{\frac{1}{2}} \bigg(\int\frac{|\wp|^2}{1+y^8}+\int\frac{1+\lf|\log y\rg|^2}{1+y^4}|\w_2^0|^2\bigg)^{\frac{1}{2}} \\
		& \les \frac{b}{\la^6}\lgba^C \Ebb
		\les \frac{\db}{\la^6}\Ebb.
	\end{aligned}
	\end{equation}
	
	\medskip
	\step{2}{Computations on time derivatives} For the first integral in \eqref{eq: morawetz}, we apply the equations \eqref{eq: Wp equation}, \eqref{eq: W20 equation} to compute
	\begin{equation} \label{eq: morawetz comp 1}
	\begin{aligned}
		&\quad \prt\int\mHaml\W_2^0\cdot G\W^\perp
		= \prt\int\W_2^0\cdot\mHaml\big(G\W^\perp\big)\\
		& = -\ao\int\R\mHaml\W_2^0\cdot\mHamlp\big(G\W^\perp\big) -\bo\int\mHaml\W_2^0\cdot\mHamlp\big(G\W^\perp\big) \\
		&\quad -\ao\int\mHaml\W_2^0\cdot G\W_2^0 -\bo\int\R\mHaml\W_2^0\cdot G\W_2^0 +\int\W_2^0\cdot\big[\prt, \mHaml\big]\big(G\W^\perp\big) \\
		&\quad +\int\mHaml\W_2^0\cdot\big(\prt G\big)\W^\perp +\int\mHaml\Q_2^2\cdot G\W^\perp +\int\mHaml\W_2^0\cdot G\Q_2^1 \\
		&\quad -\int\R\mHaml\F^\perp\cdot\mHaml\big(G\W^\perp\big)
		-\int\mHaml\W_2^0\cdot G\F^\perp.
	\end{aligned}
	\end{equation}
	The terms involving the commutator $[\prt,\mHaml]$ and $\prt G$ are anticipated tiny errors, so are the ones with $\F^\perp, \Q_2^1$ and $\Q_2^2$. The attention should be focused on the quadratic terms in $\W$ with the coefficients $\ao,\bo$. The third term in \eqref{eq: morawetz} can be treated similarly
	\begin{equation} \label{eq: morawetz comp 3}
	\begin{aligned}
		&\quad \prt\int\R\mHaml\W_2^0\cdot G\W^\perp
		= \prt\int\R\W_2^0\cdot\mHaml\big(G\W^\perp\big) \\
		& = \ao\int\mHaml\W_2^0\cdot\mHamlp\big(G\W^\perp\big) -\bo\int\R\mHaml\W_2^0\cdot\mHamlp\big(G\W^\perp\big) \\
		&\quad -\ao\int\R\mHaml\W_2^0\cdot G\W_2^0 +\bo\int\mHaml\W_2^0\cdot G\W_2^0 +\int\R\W_2^0\cdot\big[\prt,\mHaml\big]\big(G\W^\perp\big)\\ 
		&\quad +\int\R\mHaml\W_2^0\cdot(\prt G)\W^\perp +\int\R\mHaml\Q_2^2\cdot G\W^\perp +\int\R\mHaml\W_2^0\cdot G\Q_2^1 \\
		&\quad +\int\mHaml\F^\perp\cdot\mHaml\big(G\W^\perp\big) -\int\R\mHaml\W_2^0\cdot G\F^\perp.
	\end{aligned}
	\end{equation}
	Next, we treat the second and the last integrals in \eqref{eq: morawetz}. Straightforward computations from \eqref{eq: Wp equation}, \eqref{eq: W20 equation} give
	\begin{equation} \label{eq: morawetz comp 2 1}
	\begin{aligned}
		&\quad \prt\int\R\mAl\W_2^0\cdot L\W_2^0 \\
		& = \ao\int\mAl\mHaml\W_2^0\cdot L\W_2^0 -\bo\int\R\mAl\mHaml\W_2^0\cdot L \W_2^0 \\
		&\quad -\ao\int\mAl\W_2^0\cdot L\mHaml\W_2^0 -\bo\int\R\mAl\W_2^0\cdot L\mHaml \W_2^0 \\
		&\quad +\int\R\big[\prt,\mAl\big]\W_2^0\cdot L\W_2^0 +\int\R\mAl\W_2^0 \cdot(\prt L)\W_2^0 \\
		&\quad +\int \R\mAl\Q_2^2\cdot L\W_2^0 +\int\R\mAl\W_2^0\cdot L\Q_2^2\\
		&\quad +\int\mAl\mHaml\F^\perp\cdot L\W_2^0 -\int\mHaml\F^\perp\cdot L\mAl\W_2^0,
	\end{aligned}
	\end{equation}
	where from the vanishing \eqref{eq: structure identity 3} and the identity \eqref{eq: structure identity 2}, the first two lines on the RHS of \eqref{eq: morawetz comp 2 1} are actually
	\begin{equation} \notag
	\begin{aligned}
		\ao\bigg(\int\mAl\mHaml\W_2^0\cdot L\W_2^0 -\int\mHaml\W_2^0\cdot L\mAl\W_2^0\bigg) -\bo\int\R\mHaml\W_2^0\cdot G\W_2^0.
	\end{aligned}
	\end{equation}
	We rewrite the expression in the big parenthesis by $1=\ka+\kb$, separating it into two identical formulas, and then apply \eqref{eq: structure identity 1} twice on each of them:	\begin{equation} \notag
	\begin{aligned}
		&\quad \ao\bigg(\int\mAl\mHaml\W_2^0\cdot L\W_2^0 -\int\mHaml\W_2^0\cdot L\mAl\W_2^0\bigg) \\
		& = \ao\ka\bigg(\int\mHaml\W_2^0\cdot G\W_2^0 -2\int\mHaml\W_2^0\cdot L\mAl\W_2^0\bigg) \\
		&\quad +\ao\kb\bigg(\hs-\hs\int\mHaml\W_2^0\cdot G\W_2^0 +2\int\mAl\mHaml\W_2^0\cdot L\W_2^0\bigg).
	\end{aligned}
	\end{equation}
	Now by $\kk=\ka-\kb$, we gather the above integrals to get
	\begin{equation} \label{eq: morawetz comp 2}
	\begin{aligned}
		&\quad \prt\int\R\mAl\W_2^0\cdot L\W_2^0 \\
		& = \kk\ao\int\mHaml\W_2^0\cdot G\W_2^0 -\bo\int\R\mHaml\W_2^0\cdot G\W_2^0 +2\ao\kb\int\mAl\mHaml\W_2^0\cdot L\W_2^0 \\
		&\quad -2\ao\ka\int\mHaml\W_2^0\cdot L\mAl\W_2^0 +\int\R\big[\prt,\mAl\big]\W_2^0\cdot L\W_2^0 \\
		&\quad +\int\R\mAl\W_2^0\cdot (\prt L)\W_2^0 +\int \R\mAl\Q_2^2\cdot L\W_2^0 +\int\R\mAl\W_2^0\cdot L\Q_2^2\\
		&\quad +\int\mAl\mHaml\F^\perp\cdot L\W_2^0 -\int\mHaml\F^\perp\cdot L\mAl\W_2^0.
	\end{aligned}
	\end{equation}
	Finally, for the last term of \eqref{eq: morawetz}, we repeat the analysis done for \eqref{eq: morawetz comp 2} using Lemma \ref{le: structure} to obtain
	\begin{equation} \label{eq: morawetz comp 4}
	\begin{aligned}
		&\quad \prt\int\mAl\W_2^0\cdot L\W_2^0 \\
		& = -\ao\int\R\mHaml\W_2^0\cdot G\W_2^0 -\bo\int\mHaml\W_2^0\cdot G\W_2^0 \\
		&\quad +\int\big[\prt,\mAl\big]\W_2^0\cdot L\W_2^0 +\int\mAl\W_2^0\cdot(\prt L)\W_2^0 +\int\mAl\Q_2^2 \cdot L\W_2^0 \\
		&\quad +\int\mAl\W_2^0\cdot L\Q_2^2 -\int\R\mAl\mHaml\F^\perp\cdot L\W_2^0 -\int\R\mHaml\F^\perp\cdot L\mAl\W_2^0.
	\end{aligned}
	\end{equation}
	
	\medskip
	\step{3}{Numerology on the coefficients} Injecting \eqref{eq: morawetz comp 1}, \eqref{eq: morawetz comp 3}, \eqref{eq: morawetz comp 2}, \eqref{eq: morawetz comp 4} into \eqref{eq: morawetz}, we obtain the formula
	\begin{equation} \label{eq: morawetz dt}
	\begin{aligned}
		\frac{d}{dt}\,\m(t) 
			&= C_1\int\R\mHaml\W_2^0\cdot\mHamlp\big(G\W^\perp\big) +C_2\int\mHaml\W_2^0\cdot\mHamlp\big(G\W^\perp\big) \\
			&\quad +C_3\int\R\mHaml\W_2^0\cdot G\W_2^0 +C_4\int\mHaml\W_2^0\cdot G\W_2^0 \\
			&\quad -2\ao c_2\ka\int\mHaml\W_2^0\cdot L\mAl\W_2^0 +2\ao c_2\kb\int\mAl\mHaml\W_2^0\cdot L\W_2^0 \\
			&\quad +\Qfour+\Qfive+\Qf,
	\end{aligned}
	\end{equation}
	where the coefficients $C_i$ for $1\leq i\leq 4$ are given by
	\begin{equation} \notag
		\lf\{\begin{aligned}
			& C_1 =-(\ao c_1+\bo c_3)=-1,\\
			& C_2 =\ao c_3-\bo c_1=0,\\
			& C_3 =-\ao c_3+\ao c_4-\bo c_1-\bo c_2=0,\\
			& C_4 =-\ao c_1+\ao c_2\kk+\bo c_3+\bo c_4=1.
		\end{aligned}\rg.
	\end{equation}
	In view of the first two lines of \eqref{eq: morawetz dt}, this formulates the desired uncontrollable term, thanks to the very choice of $c_i$ \eqref{eq: cccc}. Moreover, the $\Q$ terms on the last line of \eqref{eq: morawetz dt} are given explicitly by \eqref{eq: Qfour expand}, \eqref{eq: Qfive expand}, \eqref{eq: Qf expand}. We estimate them respectively in the rest of this proof.
	
	\medskip
	\step{4}{Estimating $\Qfour$} By above computation, the $\Q_4$ term is given by
	\begin{equation} \label{eq: Qfour expand}
	\begin{aligned}
		\Qfour & = c_1\hs\int\W_2^0\cdot\big[\prt,\mHaml\big]\big(G \W^\perp \big) +c_3\int\R\W_2^0\cdot\big[\prt,\mHaml\big]\big(G\W^\perp\big) \\
		&\quad +c_2\int\R\big[\prt,\mAl\big]\W_2^0\cdot L\W_2^0	-c_4\int\big[\prt,\mAl\big]\W_2^0\cdot L\W_2^0\\
		&\quad +c_1\int\mHaml\W_2^0\cdot\big(\prt G\big)\W^\perp +c_3\int\R\mHaml\W_2^0\cdot\big(\prt G\big)\W^\perp \\
		&\quad +c_2\int\R\mAl\W_2^0\cdot\big(\prt L\big)\W_2^0 -c_4\int\mAl\W_2^0\cdot\big(\prt L\big)\W_2^0,
	\end{aligned}
	\end{equation}
	which contains the commutators and $\prt G, \prt L$. From the explicit formulas of $Z, V$, there holds
	\begin{equation} \label{eq: LZ LV}
		\Big|\frac{\varl Z}{y}\Big|\les\frac{1}{1+y^3},\qquad
		\Big|\frac{\varl V}{y^2}\Big|\les\frac{1}{1+y^4},
	\end{equation}
	which together with \eqref{eq: commutator3} implies
	\begin{equation} \label{eq: commutator4}
		\lf\{\begin{aligned}
			& \R\big[\prt,\mHaml\big]\big(G\Wp\big)=\frac{b^2}{\la^8}\,O\Big(\frac{1}{1+y^8}\Big)\R\wp \\
			& L\big[\prt,\mAl\big]\W_2^0=\frac{b^2}{\la^8}\, O\Big(\frac{1}{1+y^8}\Big)\,\w_2^0.
		\end{aligned}\rg.
	\end{equation}
	Using these with \eqref{eq: ip bound 3}, \eqref{eq: ip bound 29}, we can estimate the first term in \eqref{eq: Qfour expand} by Cauchy-Schwartz
	\begin{equation} \notag
	\begin{aligned}
		&\quad \Big|\int\W_2^0\cdot\big[\prt,\mHaml\big]\big(G\W^\perp\big)\Big|
		\les \frac{b^2}{\la^8}\int \frac{|\w_2^0|}{1+y^4}\cdot\frac{|\w^\perp|}{1+y^4} \\
		&\quad \les \frac{b^2}{\la^8} \bigg(\int\frac{\big|\w_2^0\big|^2}{(1+y^4)(1+\lf|\log y\rg|^2)}\bigg)^{\frac{1}{2}} \bigg(\int\frac{|\w^\perp|^2}{(1+y^8)(1+\lf|\log y\rg|^2)}\bigg)^{\frac{1}{2}} \\
		&\quad \les \frac{b\,\db}{\la^{8}}\bigg(\E_4+\bflog\bigg).
	\end{aligned} 
	\end{equation}
	The extra $b^2$ given by \eqref{eq: commutator4} ensures the same bounds for the other terms on the first and second lines of \eqref{eq: Qfour expand}. For the rest two lines, by the definition \eqref{eq: L G} we see
	\begin{equation} \notag
	\lf\{\begin{aligned}
		& \prt G =\frac{b^2}{\la^4 r^2}\Big(O(\varl V)+O(\varl^2 V)\Big)(y) =\frac{b^2}{\la^6}\, O\Big(\frac{1}{1+y^4}\Big), \\
		& \prt L = \frac{b^2}{\la^4 r}\Big(O(\varl Z)+O(\varl^2 Z)\Big)(y) =\frac{b^2}{\la^6}\, O\Big(\frac{1}{1+y^3}\Big),
	\end{aligned}\rg.
	\end{equation}
	which also bring $b^2$ as the above one, so the estimates are very similar. Applying the interpolation bounds in Lemma \ref{S: appendix A} fews more times, we obtain
	\begin{equation} \label{eq: Q4 estimate}
		|\Qfour| \les \frac{b\,\db}{\la^8}\bigg(\E_4+\bflog\bigg).
	\end{equation}
	
	\medskip
	\step{5}{Estimating $\Qfive$} We consider the $\Qfive$ term
	\begin{equation} \label{eq: Qfive expand}
	\begin{aligned}
		\Qfive & = c_1\int\mHaml\W_2^0\cdot G\Q_2^1 +c_3\int\R\mHaml\W_2^0\cdot G\Q_2^1 +c_2\int\R\mAl\Q_2^2\cdot L\W_2^0 \\
		&\quad -c_4\int\mAl\Q_2^2\cdot L\W_2^0 +c_1\int\mHaml\Q_2^2 \cdot G\W^\perp +c_3\int\R\mHaml\Q_2^2\cdot G\W^\perp \\
		&\quad +c_2\int\R\mAl\W_2^0\cdot L\Q_2^2 -c_4\int\mAl\W_2^0\cdot L\Q_2^2,
	\end{aligned}
	\end{equation}
	Observing from the definition \eqref{eq: Wp equation}, \eqref{eq: W20 equation} and the commutators \eqref{eq: commutator3}, we have the brute identities
	\begin{equation} \label{eq: Q2 expand}
		\lf\{\begin{aligned}
			& \Q_2^1=\frac{1}{\la^2}\Big(O(\w_2^1)+O\big(\hw\wedge\w_2^0\big)\Big), \\
			& \Q_2^2=\frac{1}{\la^4}\bigg[O\Big(\frac{b}{1+y^4}\Big)\R\wp +O\big(\mHam\w_2^1\big) +O\Big(\mHam\big(\hw\wedge\w_2^0\big)\Big)\bigg].
		\end{aligned}\rg.
	\end{equation}
	From \eqref{eq: ip bound 8}, \eqref{eq: ip bound 20}, there holds the smallness of $\hw$:
	\begin{equation}	 \label{eq: hw L infinity estimate}
		\|\hw\|_{L^{\infty}}
		\leq \|\wzt\|_{L^{\infty}} + \|\w^\perp\|_{L^{\infty}} + \|\gamma\|_{L^{\infty}} \les\db.
	\end{equation}
	Using this and the estimates on $\w_2^1$ \eqref{eq: ip bound 30}, we have
	\begin{equation}
	\begin{aligned}
		\Big|\int\mHaml\W_2^0\cdot G\Q_2^1\Big|
		& \les \frac{b}{\la^8}\int\big|\mHam\w_2^0\big|\,\frac{1}{1+y^4}\Big(|\w_2^1|+\big|\hw\wedge\w_2^0\big|\Big) \\
		& \les \frac{b\,\db}{\la^8}\bigg(\E_4+\bflog\bigg).
	\end{aligned}
	\end{equation}
	For the rest terms involving $\Q_2^2$ in \eqref{eq: Qfive expand}, we estimate last term on the first line for instance. From \eqref{eq: Structure comp 1}, there holds
	\begin{equation} \notag
		\int\R\mAl\Q_2^2\cdot L\W_2^0
		= \hs\int\R\Q_2^2\cdot \mAsl\big(L\W_2^0\big)
		= -\hs\int\R\Q_2^2\cdot L\mAl\W_2^0 +\hs\hs\int\R\Q_2^2\cdot G\W_2^0,
	\end{equation}
	and thus it can be controlled using \eqref{eq: Q2 expand}, \eqref{eq: ip bound 3}, \eqref{eq: ip bound 29}, \eqref{eq: ip bound 33}, we obtain
	\begin{equation} \notag
	\begin{aligned}
		\Big|\int\R\mAl\Q_2^2 \,\cdot &\, L\W_2^0\Big|
		\les \int\big|\R\Q_2^2\big|\cdot\Big(\big|L\mAl\W_2^0\big|+\big|G\W_2^0\big|\Big) \\
		& \les \frac{b}{\la^8}\int\bigg(\frac{b|\wp|}{1+y^4}+\big|\mHam\w_2^1\big| +\Big|\mHam\big(\hw\wedge\w_2^0\big)\Big|\bigg)\cdot\bigg(\frac{\big|\mA\w_2^0\big|}{1+y^3}+\frac{\big|\w_2^0\big|}{1+y^4}\bigg) \\
		& \les \frac{b\,\db}{\la^8}\bigg(\E_4+\bflog\bigg).
	\end{aligned}
	\end{equation}
	The terms involving $\mHaml\Q_2^2$ can be estimated using the self-adjointness of $\mHaml$ and Lemma~\ref{S: appendix A}. Therefore we obtain
	\begin{equation} \label{eq: Q5 estimate}
		|\Qfive| \les \frac{b\,\db}{\la^8}\bigg(\E_4+\bflog\bigg).
	\end{equation}
	
	\medskip
	\step{6}{Estimating $\Qf$} The $\Qf$ term is given by
	\begin{equation} \label{eq: Qf expand}
	\begin{aligned}
	\Qf = & -c_1 \int\R\mHaml\F^\perp\cdot\mHaml\big(G\W^\perp\big) +c_3 \int\mHaml\F^\perp\cdot\mHaml\big(G\W^\perp\big) \\
	& +c_2\int\mAl\mHaml\F^\perp\cdot L\W_2^0 +c_4\int\R\mAl\mHaml\F^\perp\cdot L\W_2^0 \\
	& -c_1\int\mHaml\W_2^0\cdot G\F^\perp -c_3\int\R\mHaml\W_2^0\cdot G\F^\perp \\
	& -c_2\int\mHaml\F^\perp\cdot L\mAl\W_2^0 -c_4\int\R\mHaml \F^\perp\cdot L\mAl\W_2^0.
	\end{aligned}
	\end{equation}
	We claim the estimate:
	\begin{equation} \label{eq: fp bound claim}
	\begin{aligned}
		& \int\frac{|\f^\perp|^2}{1+y^8} +\int\frac{1+\lf|\log y\rg|^2}{1+y^4}|\mHam\f^\perp|^2\\
		&\qquad\qquad\qquad\qquad +\int\frac{1+\lf|\log y\rg|^2}{1+y^2}|\mA\mHam\f^\perp|^2 \les \frac{\E_4}{\log M}+ \bflog,
	\end{aligned}
	\end{equation}
	and then from \eqref{eq: LZ LV}, and \eqref{eq: ip bound 3}, \eqref{eq: ip bound 29}, \eqref{eq: ip bound 31}, it follows that
	\begin{equation} \notag
	\begin{aligned}
		|\Qf| & \les \frac{b}{\la^{8}} \Bigg( \int\frac{|\w^\perp|^2}{(1+y^8)(1+\lf|\log y\rg|^2)} +\int\frac{|\w_2^0|^2}{(1+y^4)(1+\lf|\log y\rg|^2)} \\
		& \qquad\qquad\quad\;\; +\int|\mHam\w_2^0|^2 +\int\frac{|\mA\w_2^0|^2}{(1+y^2)(1+\lf|\log y\rg|^2)} \Bigg)^{\frac{1}{2}} \Ebb^{\frac{1}{2}} \\
		& \les \frac{b}{\la^8}\bigg(\frac{\E_4}{\sqrt{\log M}}+\bflog\bigg),
	\end{aligned}
	\end{equation}
	which is the desired bound for $\Qf$. This together with \eqref{eq: Q4 estimate}, \eqref{eq: Q5 estimate} concludes the proof. Recall \eqref{eq: f term}, $\f$ consists of three separate terms, namely $\errzt$, $\mmodt$, $\rest(t)$, and thus it remains to prove \eqref{eq: fp bound claim} for each one of them. The estimates for $\errzt$ follow directly from \eqref{eq: localized weighted bound 1}, \eqref{eq: localized weighted bound 4}, \eqref{eq: localized weighted bound 5}, and we are left to handle the other two. 
	
	\medskip
	\step{7}{Contribution of $\mmodt$} From \eqref{eq: modulation eq}, we see
	\begin{equation} \label{eq: U bound}
		U(t) \les \frac{1}{\sqrt{\log M}}\bigg(\sqrt{\E_4}+\bflog\bigg),
	\end{equation}
	which together with \eqref{eq: mod vec}, \eqref{eq: localized mod vec} yields 
	\begin{equation} \notag
		\mmodt^{\perp} = U(t) \phs O\bigg(\frac{\wzt}{b} +b\phsh y^5\one_{y\leq 1} +\frac{b\phsh y^3}{\lgba}\one_{1\leq y\leq 6B_0} +\frac{1}{y}\one_{y\gtrsim B_1}\bigg),
	\end{equation}
	and also
	\begin{equation} \label{eq: mHam mmodt}
	\begin{aligned}
		\mHam\mmodt 
		& = \chib\Bigg(\frac{a_s}{\ao^2+\bo^2}\begin{bmatrix} \ao\\ \bo\\ 0 \end{bmatrix}\varl\phi +\frac{b_s+b^2+a^2}{\ao^2+\bo^2}\begin{bmatrix} -\bo\\ \ao\\ 0 \end{bmatrix}\varl\phi\Bigg) \\
		&\quad +U(t)\phs O\bigg(b\phsh y^3\one_{y\leq 1} +\frac{b\phsh y}{\lgba}\one_{1\leq y\leq 6B_0} +\frac{\log y}{y}\one_{1\leq y\leq 2B_1} +\frac{1}{y^3}\one_{y\gtrsim B_1}\bigg) \\
		&\quad +U(t)\phs O\bigg(\one_{y\leq 1} +\frac{\log y}{y^2}\one_{1\leq y\leq 2B_1}\bigg)\phsh\ez.
	\end{aligned}
	\end{equation}
	Therefore we have
	\begin{equation} \notag
		\int\frac{\big|\mmodt^\perp\big|^2}{1+y^8} +\int\frac{1+\lf|\log y\rg|^2}{1+y^4}\big|\mHam\mmodt^\perp\big|^2 \les \frac{\E_4}{\log M}+ \bflog. 
	\end{equation}
	Moreover, the cancellation
	\begin{equation} \notag
		\A\Ham T_1 =\A\varl\phi =0
	\end{equation}
	makes the first profiles in $\wzt/b$ vanish, implying the refined bound
	\begin{equation} \notag
		\int\frac{1+\lf|\log y\rg|^2}{1+y^2}\big|\mA\mHam\mmodt^\perp\big|^2 \les \db\Ebb,
	\end{equation}
	which concludes \eqref{eq: fp bound claim} for $\mmodt$. 
	
	\medskip
	\step{8}{Contribution of $\rest$} By the definition \eqref{eq: rest}, we separate $\rest$ into
	\begin{equation} \label{eq: rest expand}
	\rest=\rest_1+\rest_2, \quad\mbox{where}\quad
	\lf\{\begin{aligned}	
		&\,\rest_1=\w\wedge\Big(\ao\mHam\wzt -\bo\R\mHam\wzt+\varl\phi\,\p\Big), \\
		&\,\rest_2=-\bo\hJ\big(\w\wedge\mHam\wzt\big) -\bo\w\wedge\big(\wzt\wedge\mHam\wzt\big) \\ 
		&\phantom{\,\rest_2=} -\bo\w\wedge\big(\w\wedge\mHam\wzt\big) +\vart_s Z\R\w.
	\end{aligned}\rg.
	\end{equation}
	By the given approximate solution \eqref{eq: varp asymptotics}, \eqref{eq: varp asymptotics 2}, \eqref{eq: wz decomposition}, we compute
	\begin{equation} \notag
	\begin{aligned}
		&\quad \ao\mHam\wzt -\bo\R\mHam\wzt +\varl\phi\phs\p \\
		& = \chib\bigg(\hs\big(\ao-\bo\R\big)\mHam\wz^2 +\begin{bmatrix} a+\vart_s\\ b+\frac{\la_s}{\la}\\ 0 \end{bmatrix}\hs\hs\varl\phi\hs\bigg) \\
		&\quad +\Big(\hs-2\py\chib\py\wz -\Delta\chib\wz -2\onez\phsh\py\chib\big(\ey\wedge\wz\big) +(1\hs-\hs\chib)\p\phs\varl\phi\Big) \\
		& = O\bigg(b^{2}y^3\one_{y\leq1} +\frac{b^{\frac{3}{2}}}{\lgba}\one_{y\sim B_1} +\frac{b}{y}\one_{y\gtrsim B_1}\bigg),
	\end{aligned}
	\end{equation}
	which gives the brute force estimate for $\mHam\rest_1$:
	\begin{equation} \label{eq: mHam R1 expand}
	\begin{aligned}
		&\quad \mHam\bigg[\w\wedge\Big(\ao\mHam\wzt -\bo\R\mHam\wzt+\varl\phi\,\p\Big)\bigg] \\
		& \les \bigg(|\Ham\w^{\perp}| +|\Delta\gamma| +\frac{\big|\big(\py+\frac{1}{y}\big)\w\big|}{1+y^2} \bigg) \bigg(b^{2}\phsh y^3\one_{y\leq1}\hs +\hs\frac{b^{\frac{3}{2}}}{\lgba}\one_{y\sim B_1}\hs +\hs\frac{b}{y}\one_{y\gtrsim B_1}\bigg) \\
		&\quad +\bigg(|\py\w|+\frac{|\w|}{y}\bigg) \bigg(b^{2}\phsh y^2\one_{y\leq1} +\frac{b^{\frac{3}{2}}}{y\lgba}\one_{y\sim B_1} +\frac{b}{y^2}\one_{y\gtrsim B_1}\bigg). 
	\end{aligned}
	\end{equation}
	This together with interpolation bounds in Appendix~\ref{S: appendix A} yields
	\begin{equation} \notag
		\int\frac{|\rest_1^\perp|^2}{1+y^8} +\int\frac{1+\lf|\log y\rg|^2}{1+y^4}|\mHam\rest_1^\perp|^2 +\int\frac{1+\lf|\log y\rg|^2}{1+y^2}|\mA\mHam\rest_1^\perp|^2 \les b^2\Ebb.
	\end{equation}
	Now we treat the terms involving $\rest_2$. By the construction of $\wzt$ in Lemma~\ref{le: localized ap solution}, we have	\begin{equation} \label{eq: wzt brute estimate}
	\lf\{\begin{aligned} 
		& \wzt = b\phs O\Big(y^3\one_{y\leq 1} +y\log y\,\one_{1\leq y\leq 2B_1}\Big), \\
		& \mHam\wzt = b\phs O\bigg(y\one_{y\leq 1} +\frac{\log y}{y}\,\one_{1\leq y\leq 2B_1}\bigg).
	\end{aligned}\rg.
	\end{equation}
	The bounds of the first two terms in $\rest_2$ \eqref{eq: rest expand} easily follows using Lemma~\ref{S: appendix A}. To control the third one, we apply \eqref{eq: hw L infinity estimate} to get
	\begin{equation} \notag
		\int\frac{\big|\w\wedge(\w\wedge\mHam\wzt)\big|^2}{1+y^8} \les \db\Ebb.
	\end{equation}
	Moreover, we compute by brute force that for $y\geq 1$,
	\begin{equation} \notag
	\begin{aligned}
		\big|\mHam\big(\w\hs\wedge\hs(\w\hs\wedge\hs\mHam\wzt)\big)\big|
		& \les b\, O\bigg(\frac{\log y}{y}\bigg(|\py^2\w|\phsh|\w| + |\py\w|^2 +\frac{|\py\w|\phsh|\w|}{y} +\frac{|\w|^2}{y^2}\bigg)\one_{1\leq y\leq 2B_1}\hs\bigg),
	\end{aligned}
	\end{equation}
	which together with \eqref{eq: ip bound 3}, \eqref{eq: ip bound 14}, \eqref{eq: ip bound 16}, \eqref{eq: ip bound 21} implies
	\begin{equation} \notag
		\int_{y\geq 1}\frac{1+\lf|\log y\rg|^2}{1+y^4}\big|\mHam\big(\w\wedge(\w\wedge\mHam\wzt)\big)\big|^2 \les b\,\db\Ebb,
	\end{equation}
	The $y\leq 1$ case and the estimate for $\mA\mHam\big(\w\wedge(\w\wedge\mHam\wzt)\big)$ can be obtained in the same way. For the last term in $\rest_2$ involving $\vart_s$, by \eqref{eq: modulation eq}, we have
	\begin{equation} \label{eq: phase second derivative}
	\begin{aligned}
		|\vart_s Z\R\w| & \les \big(|a|+U(t)\big)\phsh |\w^\perp|, \\
		\big|\mHam\big(\vart_s Z\R\w\big)\big| & \les \big(|a|+U(t)\big)\phsh \bigg(|\mHamp\w| +\frac{|\py\w^\perp|}{1+y^2} +\frac{|\w^\perp|}{y(1+y^2)}\bigg), \\
		\big|\mA\mHam\big(\vart_s Z\R\w\big)\big| & \les \big(|a|+U(t)\big)\phsh\bigg(|\mA\mHamp\w|+\frac{|\py^2\w^\perp|}{1+y^2} +\frac{\py\w^\perp}{y(1+y^2)} +\frac{|\w^\perp|}{y^2(1+y^2)}\bigg),
	\end{aligned}
	\end{equation}
	together with \eqref{eq: U bound}, \eqref{eq: ip bound 3} yields the bounds for $\vart_s Z\R\w$. In summary, we obtain
	\begin{equation} \notag
		\int\frac{|\rest_2^\perp|^2}{1+y^8} +\hs\hs\int\hs\frac{1+\lf|\log y\rg|^2}{1+y^4}|\mHam\rest_2^\perp|^2 +\hs\hs\int\hs\frac{1+\lf|\log y\rg|^2}{1+y^2}|\mA\mHam\rest_2^\perp|^2 \les b\,\db\Ebb.
	\end{equation}
	This concludes the proof.
	\end{proof}

	\subsection{Energy estimate}
	\label{SS: energy estimate}
	The aim of this subsection is to prove the following proposition, which is the heart of our analysis.
	
	\begin{proposition}[Mixed energy/Morawetz estimate]
	\label{pro: mixed energy estimate}
		Assume Lemma \ref{le: morawetz} and Appendix \ref{S: appendix A}. Then there exists a universal constant $d_2 \in (0,1)$ independent of $M$, such that the following differential inequality holds
		\begin{equation} \label{eq: energy estimate}
		\begin{aligned}
			& \frac{d}{dt}\,\Bigg\{\frac{1}{\la^6}\bigg[\E_4+\db\Ebb\bigg]\Bigg\} \\
			& \qquad\qquad\qquad\qquad \leq\frac{b}{\la^8}\bigg[ 2\bigg(1-d_2+\frac{C}{\sqrt{\log M}}\bigg)\E_4 +O\bigg(\bflog\bigg)\bigg],
		\end{aligned}
		\end{equation}
		where $\db$ is the infinitesimal defined by \eqref{eq: b inf}.
	\end{proposition}
	
	\begin{proof}{Proposition \ref{pro: mixed energy estimate}}
	This proof is based on the mixed energy/Morawetz identity \eqref{eq: mixed energy identity}. We proceed by first controlling the mixed energy functional, and then estimating each term on the RHS of  \eqref{eq: mixed energy identity} respectively.
	
	\smallskip
	\step{1}{Control of the functional}	Recall \eqref{eq: relation mHamp}, \eqref{eq: hJ lambda}, \eqref{eq: hw L infinity estimate}, there holds
	\begin{equation} \label{eq: JHw2 expand}
	\begin{aligned}
		\hJ \mHam \w_2
		& = (\R+\Rw)\phsh\mHam\phsh(\R+\Rw)\phsh\mHam\w \\
		& = -\big(1+\db\big)(\mHamp)^2\w^\perp +O\bigg(\Ham\bigg(\frac{\py\gamma}{1+y^2}\bigg)\bigg) +O\Big(\R\mHam\Rw\big(\mHam\w\big)\Big).
	\end{aligned}
	\end{equation}
	From \eqref{eq: ip bound 16}, \eqref{eq: ip bound 19}, the second term on the RHS of \eqref{eq: JHw2 expand} is $L^2$ bounded by
	\begin{equation} \notag
	\begin{aligned}
		\int\Big|O\bigg(\Ham\bigg(\frac{\py\gamma}{1+y^2}\bigg)\bigg)\Big|^2
		& \les \int\frac{|\py^3\gamma|^2}{1+y^4}+\int\frac{\big(\py^2\gamma-\frac{\py\gamma}{y}\big)^2}{y^2(1+y^4)}+\int\frac{|\py^2\gamma|^2}{1+y^6}+\int\frac{|\py\gamma|^2}{1+y^8} \\
		& \les \sum_{0\leq i\leq 3}\int\frac{|\py^i\gamma|^2}{y^2(1+y^{6-2i})(1+\lf|\log y\rg|^2)} +\int\frac{|\A\py\gamma|^2}{y^4(1+\lf|\log y\rg|^2)} \\
		& \les\, \db\Ebb.
	\end{aligned}
	\end{equation}
	The same $L^2$ bound holds also for the last term in \eqref{eq: JHw2 expand} by \eqref{eq: ip bound 3}, \eqref{eq: ip bound 7}, \eqref{eq: ip bound 11}, \eqref{eq: ip bound 21}, \eqref{eq: ip bound 22}, \eqref{eq: ip bound 24}. In view of the definition \eqref{def: E2 E4}, we have
	\begin{equation} \label{eq: JHw2 L2}
		\int\big|\hJ_\la\mHaml\W_2\big|^2
		=\frac{1}{\la^{6}} \int\big|\hJ\mHam\w_2\big|^2
		=\frac{1}{\la^{6}} \bigg[\E_4+\db\Ebb\bigg].
	\end{equation}
	Moreover, by \eqref{eq: morawetz bounded}, the morawetz term does not perturb the bound
	\begin{equation} \label{eq: energy morawetz size}
		\frac{1}{2}\int \big|\hJ_\la\mHaml\W_2\big|^2- \m(t) = \frac{1}{2\la^{6}} \bigg[\E_4+\db\Ebb\bigg].
	\end{equation}

	\medskip
	\step{2}{The quasilinear terms with coefficients $\ao, \bo$} We consider the first two lines of the RHS of \eqref{eq: mixed energy identity}:
	\begin{equation} \label{eq: quasilinear expand}
	\begin{aligned}
		& -\ao \int \hJ_\la \mHaml \W_2 \cdot\big(\hJ_\la \mHaml\big)^2 \W_2+ \bo \int \hJ_\la \mHaml \W_2 \cdot\big(\hJ_\la \mHaml\big)^2 \hJ_\la \W_2 \\
		& +2\ao c_2\ka\int\mHaml\W_2^0\cdot L\mAl\W_2^0 - 2\ao c_2\kb\int\mAl\mHaml\W_2^0\cdot L\W_2^0.
	\end{aligned}
	\end{equation}
	By the scaling \eqref{eq: relation Ham mHam} and definition \eqref{eq: w2}, we see 
	\begin{equation} \label{eq: quasilinear first two}
	\begin{aligned}
		& -\ao \int \hJ_\la \mHaml \W_2 \cdot\big(\hJ_\la \mHaml\big)^2 \W_2+ \bo \int \hJ_\la \mHaml \W_2 \cdot\big(\hJ_\la \mHaml\big)^2 \hJ_\la \W_2 \\
		& = \frac{1}{\la^{8}}
		\bigg[ -\ao\int\hJ\mHam\w_2\cdot(\hJ\mHam)^2\w_2 +\bo\int\hJ^2\mHam\w_2\cdot\big(\hJ\mHam\big)^2\hJ\w_2\bigg].
	\end{aligned}
	\end{equation} 
	To treat the first term of \eqref{eq: quasilinear expand}, we introduce the following lemma, whose proof is given in Appendix \ref{S: appendix B}.
	
	\begin{lemma}[Gain of two derivatives]
		\label{le: gain of 2d}
		Assume Appendix \ref{S: appendix A}. There exists a constant $d_1\,\in\,(0,1)$ such that for any vector $\vg$ under the Frenet basis $\big[\er,\etau,Q\big]$, there holds the inequality
		\begin{equation} \label{eq: gain of 2d}
			-\int\hJ\mHam\vg\cdot\big(\hJ\mHam\big)^2\vg \leq\, \frac{b(1-d_1)(|\ao|+|\bo|)}{2(\ao^2+\bo^2)} \big\|\hJ\mHam\vg\big\|_{L^2}^2 +b\,\db\,\big\|\mHam \vg\big\|_{L^2}^2,
		\end{equation}
		where $\db$ is the infinitesimal defined by \eqref{eq: b inf}.
	\end{lemma}
	Applying \eqref{eq: gain of 2d} with $\vg=\w_2$, and also \eqref{eq: ip bound 31}, we have for some constant $d_2=d_2(b^{\ast})\in(0,1)$ independent of $\W_2$ that
	\begin{equation} \label{eq: sch term estimate}
	\begin{aligned}
		-\ao\int\hJ\mHam\w_2\cdot(\hJ\mHam)^2\w_2 
		& \leq b(1-d_1)\big\|\hJ\mHam\w_2\big\|_{L^2}^2 +b\,\db\big\|\mHam\w_2\big\|_{L^2}^2 \\
		& \leq b(1-d_2)\E_4+ b\,\db\Ebb,
	\end{aligned}
	\end{equation}
	where we used $\ao(|\ao|+|\bo|)\leq 2(\ao^2+\bo^2)$. Then for the second term \eqref{eq: quasilinear expand}, we use \eqref{eq: hJ lambda} to separate it into
	\begin{equation} \label{eq: heat term expand} 
	\begin{aligned}
		&\quad-\int\hJ^2\mHam\w_2\cdot\mHam\big(\hJ\mHam\hJ\w_2\big)\\
		& = -\int\R^2\mHam\w_2\cdot\mHam\big[(\R\mHam\R)\w_2\big] \\
		& \quad -\int\R^2\mHam\w_2\cdot\mHam\Big[\Big(\R\mHam\Rw+\Rw\mHam\R+\Rw\mHam\Rw\Big)\w_2\Big] \\
		& \quad -\int\Big(\R\Rw\hs+\hs\Rw\R\hs+\hs\Rw^2\Big)\mHam\w_2\cdot\mHam\Big[\Big(\R\mHam\R\hs+\hs\R\mHam\Rw\hs+\hs\Rw\mHam\R\hs+\hs\Rw\mHam\Rw\Big)\w_2\Big].
	\end{aligned}
	\end{equation}
	By \eqref{eq: relation mHamp}, \eqref{eq: mHamp factori}, the first term on the RHS is actually
	\begin{equation} \label{eq: heat term 1}
	\begin{aligned}
		&\quad -\int\R^2\mHam\w_2\cdot\mHam\big[(\R\mHam\R)\w_2\big]
		= \int\R^2\mHam\w_2\cdot\mHam(\mHamp\w_2) \\
		& = \int\Big(\hs-\hs\mHamp\w_2 +2\onez\py\w_2^{\third}\phsh\ex\Big)\cdot\bigg[(\mHamp)^2\w_2 +2\onez\pyzy\Ham\w_2^{\first}\phsh\ez\bigg] \\
		& = -\int|\mA\mHamp\w_2|^2 +\int2\onez\py\w_2^{\third}\Ham^2\w_2^{\first}.
	\end{aligned}
	\end{equation}
	We claim the negative integral $-\int|\mA\mHamp\w_2|^2$ is the leading term of \eqref{eq: heat term expand}, and the other terms in \eqref{eq: heat term expand} are small errors in the sense that they can be bounded by $-\int|\mA\mHamp\w_2|^2$ plus $b^5\,\db\lgba^{-2}$. Indeed, in view of factorization \eqref{eq: mHamp factori}, we observe the LHS of \eqref{eq: heat term expand} can be reformulated as
	\begin{equation} \notag
	\begin{aligned}
		& \quad \int\hJ^2\mHam\w_2\cdot\mHam\big(\hJ\mHam\hJ\w_2\big) \\
		& = \int\mA\big(\hJ^2\mHam\w_2\big)\cdot\mA\big(\hJ\mHam\hJ\w_2\big) -\int\big(\hJ^2\mHam\w_2\big)^{\first}\cdot 2\onez\py\big(\hJ\mHam\hJ\w_2\big)^{\third} \\
		& \quad +\int \py\big(\hJ^2\mHam\w_2\big)^{\third}\cdot \py\big(\hJ\mHam\hJ\w_2\big)^{\third} +\int\big(\hJ^2\mHam\w_2\big)^{\third}\cdot 2\onez\pyzy\big(\hJ\mHam\hJ\w_2\big)^{\first}. 
	\end{aligned}
	\end{equation}
	The second order spacial derivative from $\mHam$ can act on either the second derivative of $\w_2$ ($\mHam\w_2$ and $\hJ\mHam\hJ\w_2$), or $\hw$ coming from $\hJ$ (recall $\hJ = (\ez+\hw)\wedge$). According to this the estimates can be splited into different types. The most involved situation is that one single derivative acts on $\hw$, and the other acts on $\mHam\w_2$ (or $\hJ\mHam\hJ\w_2$), where the whole integrand is basically given by $\py\hw$ times $\mHam\w_2$ (or $\hJ\mHam\hJ\w_2$) times a third derivative of $\w_2$. A key observation here is that the third derivative of $\w_2$ actually formulates $\mA\mHamp\w_2$, so that a straight application of Cauchy-Schwartz finishes the estimate. To see this, we estimate for instance the second term on the RHS of \eqref{eq: heat term 1}. Recalling \eqref{eq: w2}, we see
	\begin{equation} \notag
		\w_2^{\third} = (\hJ\mHam\w)^{\third} 
		= \h\alpha\Ham\beta -\h\beta\Ham\alpha.
	\end{equation}
	This together with \eqref{eq: ip bound 3}, \eqref{eq: ip bound 5} \eqref{eq: ip bound 11}, \eqref{eq: ip bound 15}, and the construction of the approximation solution implies 
	\begin{equation} \notag
	\begin{aligned}
		&\quad\int\big|2\onez\py\w_2^{\third}\Ham^2\w_2^{\first}\big| \\
		& = \int\Big|\A\bigg[O\Big(\frac{4}{1+y^2}\Big)\py\big(\h\alpha\Ham\beta -\h\beta\Ham\alpha\big)\bigg]\Big|\cdot\big|\A\Ham\w_2^{\first}\big| \\
		& \leq C\bigg(\|\py^2\hw^\perp\|_{L^\infty} +\Big\|\frac{\py\hw^\perp}{y}\Big\|_{L^\infty} +\Big\|\frac{\py\hw^\perp}{1+y^2}\Big\|_{L^\infty}\bigg) \sum_{i=0}^{4} \int\frac{|\py^i\w^\perp|}{1+y^{4-i}}\cdot|\mA\mHamp\w_2| \\
		& \leq b\lgba^{C_1} \sum_{i=0}^{4} \int\frac{|\py^i\w^\perp|}{1+y^{4-i}}\cdot|\A\mHamp\w_2| \\
		& \leq b\,\db\bflog + o(1)\int|\mA\mHamp\w_2|^2,
	\end{aligned}
	\end{equation}
	where the $o(1)$ is a positive constant that can be chosen sufficiently small. In other situations, we can control the terms directly by either $\|\py\hw\|_{L^{\infty}}^2b^4\lgba^{-2}$ or $\|\hw\|_{L^{\infty}}\int|\mA\mHamp\w_2|^2$, and thus the claim follows. Therefore we obtain
	\begin{equation} \label{eq: heat term estimate}
	\begin{aligned}
		\bo\int\hJ\mHam\w_2\cdot\big(\hJ\mHam\big)^2\hJ\w_2
		\leq -\frac{\bo}{2}\int\big|\mA\mHamp\w_2|^2 + b\,\db\Ebb.
	\end{aligned}
	\end{equation}
	Now for the second line of \eqref{eq: quasilinear expand}, from the definition of $k_i$ \eqref{eq: ks}, the coefficients satisfy $\ka(\ao^2-\bo^2)\geq 0, \kb(\bo^2-\ao^2)\geq 0$, and thus
	\begin{equation} \notag
	\lf\{\begin{aligned}
		& 2\ao c_2\ka=\frac{\ka(\ao^2-\bo^2)}{\kk\ao^2+\bo^2}\cdot\frac{4\ao^2}{\ao^2+\bo^2}\in[0,4], \\
		& -2\ao c_2\kb=\frac{\kb(\bo^2-\ao^2)}{\kk\ao^2 + \bo^2}\cdot\frac{4\ao^2}{\ao^2+\bo^2}\in[0,4].
	\end{aligned}\rg.
	\end{equation}
	These together with \eqref{eq: 43 negativity} ensure the negativity
	\begin{equation} \label{eq: negativity estimate}
		2\ao c_2\ka\int\mHaml\W_2^0\cdot L\mAl\W_2^0 \leq 0.
	\end{equation}
	In addition, from the decomposition of $\W_2^0$ \eqref{eq: W2 decomposition}, we have
	\begin{equation} \notag
	\begin{aligned}
		\int\mAl\mHaml\W_2^0\cdot L\W_2^0
		= -\frac{4b}{\la^8}\bigg(\hs-\hs\hs\int\mA\mHamp\w_2^1\cdot \frac{y\phsh\w_2^0}{(1+y^2)^2} +\int\mA\mHamp\w_2\cdot \frac{y\phsh\w_2^0}{(1+y^2)^2} \bigg),
	\end{aligned}
	\end{equation}
	where by \eqref{eq: ip bound 29}, \eqref{eq: ip bound 33}, the first integral in the parenthesis is bounded by
	\begin{equation} \notag
	\begin{aligned}
		& \Big|\int\mA\mHamp\w_2^1\cdot\frac{y\phsh\w_2^0}{(1+y^2)^2}\Big|
		\leq \int|\mHamp\w_2^1|\cdot \Big|\mAs\bigg(\frac{y\phsh\w_2^0}{(1+y^2)^2}\bigg)\Big| \\
		&\qquad\qquad \les \int\big|\R\mHam(\R^2\w_2^1)\big|\cdot \bigg(\frac{|\py\w_2^0|}{1+y^3} +\frac{|\w_2^0|}{y(1+y^3)}\bigg)
		\les \db\Ebb.
	\end{aligned}
	\end{equation}
	Moreover, from the positivity of $\bo$ and $b$, the second integral is bounded by
	\begin{equation} \notag
		\int\mA\mHamp\w_2\cdot \frac{y\phsh\w_2^0}{(1+y^2)^2}
		\leq \frac{\bo}{16b}\int\big|\mA\mHamp\w_2\big|^2 +O(b)\Ebb.
	\end{equation}
	Putting these bounds together, we see
	\begin{equation} \label{eq: absorb estimate}
		-2\ao c_2\kb\int\mAl\mHaml\W_2^0\cdot L\W_2^0
		\leq \frac{1}{\la^8} \bigg[\frac{\bo}{4}\int\big|\mA\mHamp\w_2\big|^2 +b\,\db\Ebb\bigg].
	\end{equation}
	Finally, collecting together \eqref{eq: quasilinear first two},  \eqref{eq: sch term estimate}, \eqref{eq: heat term estimate}, \eqref{eq: negativity estimate}, \eqref{eq: absorb estimate}, we obtain the control of the quadratic terms in \eqref{eq: mixed energy identity}: There holds
	\begin{equation} \label{eq: quadratic}
	\begin{aligned}
		& -\ao \int \hJ_\la \mHaml \W_2 \cdot\big(\hJ_\la \mHaml\big)^2 \W_2+ \bo \int \hJ_\la \mHaml \W_2 \cdot\big(\hJ_\la \mHaml\big)^2 \hJ_\la \W_2 \\
		& +2\ao c_2\ka\int\mHaml\W_2^0\cdot L\mAl\W_2^0 - 2\ao c_2\kb\int\mAl\mHaml\W_2^0\cdot L\W_2^0 \\
		& \leq  \frac{b\big(1-d_2\big)}{\la^{8}} \Ebb,
	\end{aligned}
	\end{equation}
	for some constant $d_2\in(0,1)$. Here we note that the integral $\int\big|\mA\mHamp\w_2\big|^2$ in \eqref{eq: absorb estimate} has been absorbed by the one in \eqref{eq: heat term estimate} with negative sign. 
	
	\medskip
	\step{3}{The $\Q_3$ term} According to the definition \eqref{eq: Q3}, we separate $\Q_3$ into four parts
	\begin{equation} \notag
		\Q_3=\Q_3^1 +\Q_3^2 +\Q_3^3 +\Q_3^4,
	\end{equation}
	where
	\begin{equation} \label{eq: Q3 separate}
	\begin{aligned}
		\Q_3^1 &= -\Big(b+\frac{\la_s}{\la}\Big)\int\R\mHaml\W_2^0\cdot \bigg[-\mHamlp\bigg(\frac{(\varl V)_\la}{\la^2r^2}\W^\perp\bigg) +\frac{(\varl V)_\la}{\la^2r^2}\R \W_2^0\,\bigg], \\
		\Q_3^2 &= \int\R\mHaml\W_2^0\cdot \bigg(\R\mHaml\R \big[\prt,\mHaml\big]\gamma\,\ez +\Rw\mHaml \R\big[\prt,\mHaml\big]\W \\
		&\qquad\qquad\qquad\quad +\prt\hW\wedge\mHaml\W_2^0 +\Rw\big[\prt,\mHaml\big]\W_2^0+\big[\prt,\hJ_\la\mHaml\big]\W_2^1\bigg), \\
		\Q_3^3 &= \int\R\mHaml\W_2^1\cdot \Big(\hJ_\la\mHaml\R \big[\prt,\mHaml\big]\W + \big[\prt,\hJ_\la\mHaml\big]\W_2\Big), \\
		\Q_3^4 &= \int\Rw\mHaml\W_2 \cdot\Big(\hJ_\la\mHaml\R\big[\prt,\mHaml\big]\W +\big[\prt,\hJ_\la\mHaml\big]\W_2\Big).
	\end{aligned}
	\end{equation}
	We estimate each one of these respectively. From the modulation equations \eqref{eq: modulation eq}, and the rough bound of the uncontrollable term \eqref{eq: uncontrollable term}, we have
	\begin{equation} \label{eq: Q31 estimate}
		|\Q_3^1| \les\frac{O(b^3)}{\la^{8}}\Ebb \les\frac{b\,\db}{\la^{8}}\Ebb.
	\end{equation}
	To treat $\Q_3^2$, we apply \eqref{eq: commutator2}, \eqref{eq: modulation eq} to get the brute estimate
	\begin{equation} \label{eq: commutator2 explain}
		\big[\prt,\mHaml\big]\W 
		= \frac{b}{\la^4} O\bigg(
			\frac{\w^\perp}{1+y^4} +\frac{\py\gamma\,\ex}{1+y^2} +\frac{1}{1+y^2}\Big(\py\alpha+\frac{\alpha}{y}\Big)\ez \bigg).
	\end{equation}
	Now we consider the $L^2$ bound of each term in the big parenthesis. By \eqref{eq: commutator2 explain}, \eqref{eq: ip bound 16}, the first term in the brace of $\Q_3^2$ is	
	\begin{equation} \notag
		\int\big|\R\mHaml\R\big[\prt,\mHaml\big]\gamma\phsh\ez\big|^2 
		\les \frac{b^2}{\la^{10}} \int\Big|\Ham\bigg(\frac{\py\gamma}{1+y^2}\bigg)\Big|^2 \les \frac{b^2\,\db}{\la^{10}}\Ebb.
	\end{equation}
	Similarly from \eqref{eq: commutator2}, \eqref{eq: hw L infinity estimate}, \eqref{eq: ip bound 3} the second term in the brace is bounded by
	\begin{equation} \notag
	\begin{aligned}
		\int\Big|\Rw\mHaml\R\big[\prt,\mHaml\big]\W\Big|^2 
		& \les \frac{b^2\,\db}{\la^{10}} \int
			\bigg[ \Big|\mHam\bigg(\frac{\w^\perp}{1+y^4}\bigg)\Big|^2 +\Big|\Ham\bigg(\frac{\py\gamma}{1+y^2}\bigg)\Big|^2 \bigg] \\
		& \les \frac{b^2\,\db}{\la^{10}}\Ebb.
	\end{aligned} 
	\end{equation}
	The rest of the terms in the parenthesis can be controlled in a similar fashion. Thus we have
	\begin{equation} \label{eq: Q32 estimate}
		|\Q_3^2|\les\big\|\R\mHaml\W_2^0\big\|_{L^2} \bigg[\frac{b^2\,\db}{\la^{10}}\Ebb\bigg]^{\frac{1}{2}}
		\les \frac{b\,\db}{\la^{8}}\Ebb.
	\end{equation}
	For $\Q_3^3$, we compare it with the LHS of \eqref{eq: uncontrollable small originate}. The latter, according to the proof of Lemma~\ref{pf: plain energy} (see \eqref{eq: uncontrollable small originate}, \eqref{eq: uncontrollable small originate 2}), has been splited into three separate parts in \eqref{eq: Q3 separate}. Specifically,
	\begin{equation} \label{eq: compare}
	\begin{aligned}
		&\quad \int\R\mHaml\W_2^0 \cdot \Big(\hJ_\la\mHaml\R_z\big[\prt,\mHaml\big]\W +\big[\prt,\hJ_\la\mHaml\big]\W_2\Big) \\
		& = \int\R\mHaml\W_2^0\cdot\Big[-\mHamlp\big(G\W^\perp\big)+G\R \W_2^0\,\Big] +\Q_3^1 +\Q_3^2.
	\end{aligned}
	\end{equation} 
	In view of the rough bound \eqref{eq: uncontrollable term} and the above computations on $\Q_3^1, \Q_3^2$, we have already estimated the integrals in \eqref{eq: compare} using Cauchy-Schwartz and Appendix~\ref{S: appendix A}. Note that the first factor of their integrands are always $\R\mHaml\W_2^0$, bounded by
	\begin{equation} \notag
		\int|\R\mHaml\W_2^0|^2 =\frac{1}{\la^6}\int|\R\mHam\w_2^0|^2 \les\frac{1}{\la^6}\Ebb.
	\end{equation}
	On the contrary, we review \eqref{eq: Q3 separate} and see the first factor of the integrand of $\Q_3^3$ is $\R\mHaml\W_2^1$ instead of $\R\mHaml\W_2^0$. This, by \eqref{eq: ip bound 33}, actually gives a better bounded than $\R\mHaml\W_2^0$:
	\begin{equation} \notag
		\int|\R\mHaml\W_2^1|^2 \leq\frac{1}{\la^6}\int|\mHam\w_2^1|^2 \les\frac{\db}{\la^6}\Ebb.
	\end{equation}
	The extra infinitesimal $\db$ bonus appearing here with the previous estimates yields the desired bound
	\begin{equation} \label{eq: Q33 estimate}
	\begin{aligned}
		|\Q_3^3| & \les \db
			\bigg(\Big|\hs\int\hs\R\mHaml\W_2^0\hs\cdot\hs \Big[\hs-\hs\hs\mHamlp\big(G\W^\perp \big)\hs+\hs G\R\W_2^0\Big]\Big| +|Q_3^1| +|Q_3^2|\bigg)\\
		& \les \frac{b\,\db}{\la^{8}}\Ebb.
	\end{aligned}
	\end{equation}
	Finally, for $\Q_3^4$, the smallness of $\|\hw\|_{L^\infty}$ \eqref{eq: hw L infinity estimate} implies the same bound as \eqref{eq: Q33 estimate} via similar analysis. Thus we have
	\begin{equation} \label{eq: Q3 estimate}
		|\Q_3| \les \frac{b\,\db}{\la^{8}}\Ebb.
	\end{equation}
	
	\medskip
	\step{4}{The term involving $\Q_1$}	Recall the definition of $\Q_1$ \eqref{eq: W2 equation}, we see the term involving $\Q_1$ in \eqref{eq: mixed energy identity} is actually given by
	\begin{equation} \label{eq: Q1 expand}
	\begin{aligned}
		&\quad \int\hJ_\la\mHaml\W_2\cdot\hJ_\la \mHaml\Q_1 \\
		& = \int\hJ_\la\mHaml\W_2\cdot\hJ_\la \mHaml\Big(\prt\hW\wedge\mHaml\W +\Rw\big[\prt,\mHaml\big]\W\Big) \\
		& = \frac{1}{\la^8}\hs\int\hJ\mHam\w_2\hs\cdot\hs \hJ\mHam\big(\prs\wzt\hs\wedge\hs\mHam\w\big) 
		+\frac{1}{\la^8}\hs\int\hJ\mHam\w_2\hs\cdot\hs \hJ\mHam\Big[\hs\big(\hs-\hs\ao\w_2\hs+\hs\bo\hJ\w_2\hs+\hs\f\big) \wedge\mHam\w\Big] \\
		&\quad +\int\hJ_\la\mHaml\W_2\cdot\hJ_\la \mHaml\Big(\hW\wedge\big[\prt,\mHaml\big]\W\Big).
	\end{aligned}
	\end{equation}
	By the explicit formula of $\wz$ given by Lemma~\ref{le: ap solution}, and the modulation equations \eqref{eq: modulation eq}, we have
	\begin{equation} \label{eq: partial s wzt}
	\begin{aligned}
		\prs\wzt 
		& = (\varl\chi)_{B_1}\bigg(\hs-\hs\frac{\prs B_1}{B_1}\bigg)\wz +\chi_{B_1}\prs\wz \\
		& = O\big(b\phsh\wz\one_{y\sim B_1}\big)+ \big(\prs\wz\one_{y\les B_1}\big)
		= b^2\,O\Big(y^3\one_{y\leq 1} +y\log y\,\one_{1\leq y\les B_1}\Big).
	\end{aligned}
	\end{equation}
	Applying this with \eqref{eq: partial s wzt}, \eqref{eq: JHw2 L2}, \eqref{eq: ip bound 6}, \eqref{eq: ip bound 18}, we obtain
	\begin{equation} \notag
	\begin{aligned}
		& \quad\; \Big|\int\hJ\mHam\w_2\cdot\hJ\mHam \big(\prs\wzt\wedge\mHam\w\big)\Big| \\
		& \les b^2 \bigg(\int|\hJ\mHam\w_2|^2\bigg)^{\frac{1}{2}} \Bigg(\sum_{i=0}^4\int\frac{|\py^i\w|^2|\log y|^2}{y^2(1+y^{4-2i})}\Bigg)^{\frac{1}{2}} \les b\,\db\Ebb.
	\end{aligned}
	\end{equation}
	Also, by the double wedge formula, there holds
	\begin{equation} \notag
	\begin{aligned}
		\w_2\wedge\mHam\w 
		& = \Big((\ez+\hw)\wedge\mHam\w\Big)\wedge\mHam\w \\
		& = (\ez\cdot\mHam\w)\,\mHam\w-|\mHam\w|^2\ez+O\big(|\hw|\phsh|\mHam\w|^2\big).
	\end{aligned}
	\end{equation}
	We apply \eqref{eq: ip bound 3}, \eqref{eq: ip bound 16}, \eqref{eq: ip bound 18}, \eqref{eq: ip bound 24}, \eqref{eq: ip bound 25} and it follows that
	\begin{equation} \notag
		\Big|\int\hJ\mHam\w_2\cdot\hJ\mHam \big(\w_2\wedge\mHam\w\big)\Big| \les b\,\db\Ebb.
	\end{equation}
	The term involving $\bo\hJ\w_2$ in \eqref{eq: Q1 expand} can be estimated similarly, and the one with $f$ can be treated as in Lemma~\ref{le: morawetz}. Now for last line of \eqref{eq: Q1 expand}, we use \eqref{eq: commutator2 explain} with \eqref{eq: ip bound 3}, \eqref{eq: ip bound 16} to get 
	\begin{equation} \notag
		\Big|\int\hJ_{\la}\mHaml\W_2\cdot\hJ_{\la}\mHaml\Big(\hW\wedge\big[\prt,\mHaml\big]\W\Big)\Big|
		\les \frac{b\,\db}{\la^8}\Ebb.
	\end{equation}
	These yield that
	\begin{equation} \label{eq: Q1 estimate}
		\Big| \int \hJ_\la \mHaml \W_2 \cdot \hJ_\la \mHaml \Q_1 \Big| \les \frac{b\,\db}{\la^{8}}\Ebb.
	\end{equation}
	
	\medskip
	\step{5}{The integral involving $\F$} From the definition \eqref{eq: F}, there holds the estimate:
	\begin{equation} \label{eq: hJ mHam square f}
		\int\hJ_\la\mHaml\W_2\cdot\big(\hJ_\la\mHaml\big)^2\F 
		\les \frac{1}{\la^{8}} \Ebb^{\frac{1}{2}}\bigg( \int\big|(\hJ\mHam)^2\f\big|^2 \bigg)^{\frac{1}{2}}.	
	\end{equation}
	We claim the following bound:
	\begin{equation} \label{eq: hJ mHam square f expand}
		\int|(\hJ\mHam)^2\mmodt|^2 +\int|(\hJ\mHam)^2\errzt|^2 + \int|(\hJ\mHam)^2\rest|^2 
		\les b\,\db\Ebb.
	\end{equation}
	In view of \eqref{eq: f term}, we inject \eqref{eq: hJ mHam square f expand} into \eqref{eq: hJ mHam square f} and then the desired bound for the $\F$ integral follows. Collecting this with \eqref{eq: energy morawetz size}, \eqref{eq: quadratic}, \eqref{eq: Q1 estimate}, \eqref{eq: Q3 estimate}, we conclude the proof of Proposition~\ref{pro: mixed energy estimate}. Now it remains to show \eqref{eq: hJ mHam square f expand}.
	
	\medskip
	\step{6}{Contribution of $\mmodt$} We see from \eqref{eq: mHam mmodt} that
	\begin{equation} \notag
	\begin{aligned}
		\hJ\mHam\mmodt 
		& = \chib\Bigg(\frac{a_s}{\ao^2+\bo^2}\begin{bmatrix} -\bo\\ \ao\\ 0\end{bmatrix}\varl\phi -\frac{b_s+b^2+a^2}{\ao^2+\bo^2}\begin{bmatrix} \ao\\ \bo\\ 0\end{bmatrix}\varl\phi\Bigg)\\
		&\quad +U(t)\phsh\bigg(b\phsh y^{3}\one_{y\leq 1} +\frac{b\phsh y}{\lgba}\one_{1\leq y\leq 6B_0} +\frac{\log y}{y}\one_{y\sim B_1} +\frac{1}{y^3}\one_{y\gtrsim B_1}\bigg) \\
		&\quad +U(t)\phs\bigg[\hw\wedge O\bigg(\one_{y\leq 1} +\frac{\log y}{y}\one_{1\leq y\leq 2B_1} +\frac{1}{y^3}\one_{y\gtrsim B_1} \bigg)\bigg].
	\end{aligned}
	\end{equation}
	Applying $\hJ\mHam$ again, we find in the region of $0\leq y\leq B_1$ that the first line on the above RHS vanishes, and hence
	\begin{equation} \notag
	\begin{aligned}
		\big(\hJ\mHam\big)^2\mmodt
		& = U(t)\bigg(b\phsh y\one_{y\leq 1} +\frac{b}{y\lgba}\one_{1\leq y\leq 6B_0} +\frac{\log y}{y^3}\one_{y\sim B_1} +\frac{1}{y^5}\one_{y\gtrsim B_1}\bigg) \\
		&\quad + U(t)\phs \hJ\mHam\bigg[\hw\wedge O\bigg(\one_{y\leq 1} +\frac{\log y}{y}\one_{1\leq y\leq 2B_1} +\frac{1}{y^3}\one_{y\gtrsim B_1} \bigg)\bigg].
	\end{aligned}
	\end{equation}
	For the first part above, a straightforward computation with \eqref{eq: U bound} yields
	\begin{equation} \notag
		U(t)^2 \hs\hs\int\hs\hs\bigg(b\phsh y\one_{y\leq 1} \hs+\hs\frac{b}{y\lgba}\one_{1\leq y\leq 6B_0} \hs+\hs\frac{\log y}{y^3}\one_{y\sim B_1} \hs+\hs\frac{1}{y^5}\one_{y\gtrsim B_1}\hs\hs\bigg)^2\hs\hs
		\les b^2\bigg(\frac{\E_4}{\log M}\hs+\hs\frac{b^4}{\lgba^2}\bigg),
	\end{equation}
	while the estimate for the second part requires a separation by $\hw=\wzt+\w^\perp+\gamma\phsh\ez$ and further applications of the interpolation bounds in Appendix~\ref{S: appendix A}. The resulting bound is identical. The details are left to readers. This yields the estimate for $\mmodt$.
	
	\medskip
	\step{7}{Contribution of $\errzt$} According to \eqref{eq: hJ lambda}, we let
	\begin{equation} \notag
	\hJ\mHam\errzt = \Po + \Pt, \quad\mbox{where}\quad
	\lf\{\begin{aligned} 
		& \Po = \R\mHam\phsh\errzt, \\
		& \Pt = \Rw\mHam\phsh\errzt.
	\end{aligned}\rg.
	\end{equation}
	By \eqref{eq: mHam func}, the explicit expressions are
	\begin{equation} \label{eq: po}
		\Po = \begin{bmatrix}
			-\hs\Ham\errzt^{\second} \\
			 \Ham\errzt^{\first} \\ 0 \\
		\end{bmatrix} -2\onez\phsh\py\errzt^{\third}\phsh\ey,
	\end{equation}
	\begin{equation} \label{eq: pt}
	\begin{aligned}
		\Pt 	
		& = \bigg\{ \h\beta\bigg[-\Delta\errzt^{\third} + 2\onez\pyzy\errzt^{\first}\bigg] - \h\gamma\Ham\errzt^{\second} \bigg\}\phsh\ex \\
		&\quad +\bigg\{ \h\gamma\Big[\Ham\errzt^{\first} -2\onez\py \errzt^{\third}\Big] -\h\alpha\bigg[-\Delta \errzt^{\third} +2\onez\pyzy\errzt^{\first}\bigg] \bigg\}\phsh\ey \\
		&\quad + \bigg\{ \h\alpha\Ham\errzt^{\second} -\h\beta\Big[\Ham\errzt^{\first} -2\onez\py\errzt^{\third} \Big] \bigg\}\phsh\ez.
	\end{aligned}
	\end{equation}
	Using \eqref{eq: localized weighted bound 2}, \eqref{eq: localized weighted bound 4}, \eqref{eq: localized weighted bound 6}, we have
	\begin{equation} \notag
	\begin{aligned}
		\int|\mHam\Po|^2
		\les & \int\big|\Ham^2\errzt^{\first}\big|^2 + \int\big|\Ham^2 \errzt^{\second}\big|^2 \\
		& +\int\Big|\frac{\big(\py+\tfrac{1}{y}\big)\Ham\errzt^{\second}}{1+y^2}\Big|^2 +\sum_{1\leq i\leq 3}\int\frac{|\py^{i}\errzt^{\third}|^2}{(1+y^{2})y^{6-2i}} \les \frac{b^6}{\lgba^2}, 
	\end{aligned}
	\end{equation}
	which together with \eqref{eq: hw L infinity estimate} yields
	\begin{equation} \notag
		\int|\hJ\mHam\Po|^2 \les \big(1+\|\hw\|_{L^{\infty}}\big)\int|\mHam\Po|^2 \les \frac{b^6}{\lgba^2}.
	\end{equation}
	Similarly, \eqref{eq: hw L infinity estimate} implies
	\begin{equation} \label{eq: Pt no hJ}
		\int|\hJ\mHam\Pt|^2 \les \int|\mHam\Pt|^2.
	\end{equation}
	Thus it suffices to control $\mHam\Pt$ directly. For this, we   compute its components by \eqref{eq: pt}, and estimate them respectively. For instance, let us consider for $y\leq 1$ the following term coming from the first component of $\mHam\Pt$:
	\def\be{\h\beta}
	\def\ff{\errzt^{\third}}
	\def\gg{\errzt^{\first}}
	\begin{equation} \label{eq: a term in Pt}
	\begin{aligned}
		&\quad \Ham\bigg\{\h\beta\bigg[-\Delta\errzt^{\third} + 2\onez\pyzy\errzt^{\first}\bigg]\bigg\} \\
		& \les \big|\Ham\big(\h\beta\phsh\Delta\errzt^{\third}\big)\big| +\Big|\Ham\bigg[\h\beta\big(1\hs+\hs O(y^2)\big)\pyzy\errzt^{\first}\bigg]\Big| \\
		& \les \frac{y|\Ham\be|}{1+\lf|\log y\rg|}\Bigg( \sum_{1\leq i\leq 2}\frac{|\py^{i}\ff|}{y^{4-i}} +\sum_{0\leq i\leq 1}\frac{|\py^{i}\gg|}{y^{3-i}}\Bigg) + y\bigg(|\py\be|+\frac{|\be|}{y}\bigg) \sum_{0\leq i\leq 3}\frac{|\py^{i}\gg|}{y^{3-i}}.
	\end{aligned}
	\end{equation}
	Using \eqref{eq: localized weighted bound 1}, \eqref{eq: localized weighted bound 2}, \eqref{eq: ip bound 11}, \eqref{eq: ip bound 13}, we have
	\begin{equation} \notag
	\begin{aligned}
		&\quad \int_{y\leq 1}\Big|\Ham\bigg\{\h\beta\bigg[-\Delta\errzt^{\third} + 2\onez\pyzy\errzt^{\first}\bigg]\bigg\}\Big|^2 \\
		&\les \Big\|\frac{\Ham\be}{y(1+\lf|\log y\rg|)}\Big\|_{L^{\infty}(y\leq 1)}^2 \Bigg(\sum_{1\leq i\leq 2}\int\frac{|\py^{i}\ff|^2}{y^{8-2i}} +\sum_{0\leq i\leq 1}\int\frac{|\py^{i}\gg|^2}{y^{6-2i}}\Bigg) \\
		&\quad +\bigg(\|\py\be\|_{L^{\infty}(y\leq 1)}^2 +\Big\|\frac{\be}{y}\Big\|_{L^{\infty}(y\leq 1)}^2\bigg) \sum_{0\leq i\leq 3}\int\frac{|\py^{i}\gg|^2}{y^{6-2i}} 
		\les \frac{b^6}{\lgba^2}.
	\end{aligned}
	\end{equation}
	The estimates of \eqref{eq: a term in Pt} for $y\geq 1$ can be derived in a similar fashion. Other terms in $\mHam\Pt$ can be also treated in the same way. Therefore we obtain
	\begin{equation} \notag
		\int|\mHam\Pt|^2 \les \frac{b^6}{\lgba^2},
	\end{equation}
	which together with \eqref{eq: Pt no hJ} yields the bound for $\Pt$. This concludes the estimate for $\errzt$.
	
	\medskip	
	\step{8}{Contribution of $\rest$} From the decomposition \eqref{eq: rest expand}, and \eqref{eq: hw L infinity estimate}, we see
	\begin{equation} \label{eq: rest expand 2}
		|(\hJ\mHam)^2\rest| \les |\mHam\R\mHam\rest_1| +|\mHam\Rw\mHam\rest_1| +|\mHam\R\mHam\rest_2| +|\mHam\Rw\mHam\rest_2|.
	\end{equation}
	Using the estimate for $\mHam\rest_1$ \eqref{eq: mHam R1 expand}, we compute by brute force that
	\begin{equation} \notag
	\begin{aligned}
		\mHam\R\mHam\rest_1 & = O\big(\Ham(\mHam\rest_1)\big) +O\bigg(\frac{\py\mHam\rest_1}{1+y^2}\bigg) +O\bigg(\frac{\mHam\rest_1}{y(1+y^2)}\bigg) \\
		& \les \Bigg[\sum_{0\leq i\leq 2}\bigg(\frac{|\py^{i}\Ham\w^{\perp}|}{y^{2-i}} +\frac{|\py^{i}\Delta\gamma|}{y^{2-i}}\bigg) +\sum_{0\leq i\leq 3}\frac{|\py^{i}\w|}{y^{3-i}(1+y^2)} \Bigg] \\
		&\qquad\qquad \times\bigg(b^{2}\phsh y^3\one_{y\leq1}\hs +\hs\frac{b^{\frac{3}{2}}}{\lgba}\one_{y\sim B_1}\hs +\hs\frac{b}{y}\one_{y\gtrsim B_1}\bigg),
	\end{aligned}
	\end{equation}
	where $b^2\phsh y^3\one_{y\leq 1}$ eliminates the possible singularity at the origin caused by $|\py^{i}\w|/y^{3-i}$. From \eqref{eq: ip bound 3}, \eqref{eq: ip bound 5}, \eqref{eq: ip bound 16}, all these terms can be controlled, and thus
	\begin{equation} \notag
		\int|\mHam\R\mHam\rest_1|^2\les b\,\db\Ebb.
	\end{equation}
	The estimate for $\mHam\Rw\mHam\rest_1$ is more involved, but in view of the smallness of derivatives of $\hw$ indicated by \eqref{eq: ip bound 15}, \eqref{eq: ip bound 21}, \eqref{eq: ip bound 22}, \eqref{eq: ip bound 24}, it could be handled similarly. Therefore we have
	\begin{equation} \notag
		\int|(\hJ\mHam)^2\rest_1|^2\les b\,\db\Ebb.
	\end{equation}
	Then for the terms involving $\rest_2$ in \eqref{eq: rest expand 2}, we recall from \eqref{eq: rest expand} that
	\begin{equation} \notag
	\begin{aligned}
		\rest_2 
		& = -\bo\hJ\big(\w\wedge\mHam\wzt\big) -\bo\w\wedge\big(\wzt\wedge\mHam\wzt\big) \\ 
		&\quad -\bo\w\wedge\big(\w\wedge\mHam\wzt\big) +\vart_s Z\R\w.
	\end{aligned}
	\end{equation}
	The first two terms in $\rest_2$ are concerned with $\mHam\wzt$, and thus easy to treat. Indeed, by \eqref{eq: wzt brute estimate}, $\mHam\wzt$ would eventually become an extra $b^2$ bonus before the $\E_4$ bound, which results in the bound $b^5\db/\lgba^2$. Now we treat the third term in $\rest_2$. By brute computations, and \eqref{eq: ip bound 3}, \eqref{eq: ip bound 11}, \eqref{eq: ip bound 13}, \eqref{eq: ip bound 16}, \eqref{eq: ip bound 22}, we have 
	\begin{equation} \notag
	\begin{aligned}
		\int\big|\mHam\R & \mHam\big(\w\hs\wedge(\w\hs\wedge\hs\mHam\wzt)\big)\big|^2 
		\les \sum_{0\leq i_1+i_2+i_3\leq 4}\int\frac{|\py^{i_1}\w|^2|\py^{i_2}\w|^2|\py^{i_3}\mHam\wzt|^2}{y^{2(4-i_1+i_2+i_3)}} \\
		& \les b^2 \hspace{-.3em}\sum_{0\leq i_1+i_2\leq 4}\int\frac{|\py^{i_1}\w|^2|\py^{i_2}\w|^2}{y^{2(3-i_1-i_2)}} \bigg(\one_{y\leq 1} +\frac{\log y}{y^2}\one_{1\leq y\leq 2B_1}\bigg)^2 \\
		& \les b^2\lgba^4 \bigg(\big\|\py\w\big\|_{L^{\infty}}^2+\Big\|\frac{\w}{y}\Big\|_{L^{\infty}}^2\bigg) \sum_{0\leq i\leq 4}\int \frac{|\py^i\w|^2}{y^{2(2-i)}(1+y^4)(1+\lf|\log y\rg|^2)} \\
		& \quad +b^2\lgba^4\int\frac{|\py^2\w|^4}{(1+y^4)(1+\lf|\log y\rg|^2)} \\
		& \les b^3 \Ebb \les b\,\db\Ebb.
	\end{aligned}
	\end{equation}
	Finally for the phase term $\mHam\R\mHam(Z\R\w)$ in $\rest_2$, to deal with the possible singularity at the origin, we split it into
	\begin{equation} \label{eq: phase singularity}
	\begin{aligned}
		\mHam\R\mHam(Z\R\w) 
		& = \mHam\R\mHam\big((Z-1)\R\w\big) +\mHam\R\mHam(\R\w) \\
		& = -\mHam\mHamp\big((Z-1)\w\big) -\mHam\mHamp\w.
	\end{aligned}
	\end{equation} 
	For the first term in \eqref{eq: phase singularity}, we see
	\begin{equation} \notag
	\begin{aligned}
		\mHam\mHamp\big((Z-1)\w\big) 
		& = \Ham^2\big((Z-1)\alpha\big)\ex 
			+\Ham^2\big((Z-1)\beta\big)\ey \\
		&\quad +2\onez\pyzy\Ham\big((Z-1)\alpha\big)\ez,
	\end{aligned}
	\end{equation}
	where the singularity can only arise from the first two components, so we compute using the asymptotics $Z-1=-2y^2+O(y^4)$ for $y\leq 1$ that
	\begin{equation} \notag
	\begin{aligned}
		\Ham^2\big((Z-1)\alpha\big)
		& \sim \Ham^2(y^2\alpha) + O\big(\Ham^2(y^4\alpha)\big) \\
		& = y^2\Ham^2\alpha -4y\big(\py\Ham\alpha +\Ham\py\alpha\big) \\
		&\quad -8\Ham\alpha +8\py^2\alpha +\frac{4\py\alpha}{y} +O\big(\Ham^2(y^4\alpha)\big).
	\end{aligned}
	\end{equation}
	The singular terms are actually
	\begin{equation} \notag
	\begin{aligned}
		-4y\big(\py\Ham\alpha +\Ham\py\alpha\big) +\frac{4\py\alpha}{y}
		& = -4y\bigg[2\Ham\py\alpha +\py\bigg(\frac{V}{y^2}\bigg)\bigg] \\
		& = 8y\bigg(\py^3\alpha +\frac{\py^2\alpha}{y}\bigg) 
			-4y\bigg[\,\frac{2V}{y^2}\py\alpha +\py\bigg(\frac{V}{y^2}\bigg)\alpha\,\bigg],
	\end{aligned}
	\end{equation}
	where by $V-1=4y^2+O(y^4)$, we have
	\begin{equation} \notag
	\begin{aligned}
		\frac{2V}{y^2}\py\alpha +\py\bigg(\frac{V}{y^2}\bigg)\alpha
		& = \frac{2(V-1)}{y^2}\py\alpha -\py\bigg(\frac{V-1}{y^2}\bigg)\alpha +\frac{2(Z-1)}{y^2}\alpha -\frac{2\A\alpha}{y^2} \\
		& = O\big(\py\alpha\big) +O(\alpha) -\frac{2\A\alpha}{y^2}.
	\end{aligned}
	\end{equation}
	From \eqref{eq: ip bound 10}, the last term here admits estimate
	\begin{equation} \label{eq: singularity by A}
		\int_{y\leq 1}\frac{|\A\alpha|^2}{y^4} \les C(M)\E_4,
	\end{equation}
	which helps control the singularity. For second term in \eqref{eq: phase singularity}, we can treat it in the same way, and the singularity may come from the third component, that is
	\begin{equation} \notag
	\begin{aligned}
		\bigg(\py+\frac{Z}{y}\bigg)\Ham\alpha 
		& = -\py^3\alpha-\frac{1+Z}{y}\py^2\alpha +\frac{V-1}{y^2}\py\alpha -\frac{Z-1}{y^2}\py\alpha +\py\bigg(\frac{V-1}{y^2}\bigg)\alpha \\
		&\quad +\frac{(Z-1)(V-1)}{y^3}\alpha +\frac{2(Z-1)}{y^3}\alpha +\frac{V-1}{y^3}\alpha +\frac{\py\alpha}{y^2} -\frac{Z}{y^3}\alpha.
	\end{aligned}
	\end{equation}
	The only singular terms are the last two terms, which still constitute $-\A\alpha/y^2$, and thus are bounded by \eqref{eq: singularity by A} again. Other non-singular terms in \eqref{eq: phase singularity} can be controlled by $C(M)\E_4$ using \eqref{eq: ip bound 3}. Then by the modulation equations \eqref{eq: modulation eq}, we have
	\begin{equation} \notag
		\int\big|\mHam\R\mHam\big(\vart_s Z\R\w\big)\big|^2
		\les \big(|a|+U(t)\big)^2 \Ebb \les b^2\Ebb.
	\end{equation}
	The last term in \eqref{eq: rest expand 2} share the same bound as above due the smallness of $\hw$. We omit the details. Therefore we obtain
	\begin{equation}
		\int|(\hJ\mHam)^2\rest|^2 \les b\,\db\Ebb,
	\end{equation}
	which concludes the estimate for $\rest$, and ends the proof.
	\end{proof}

	\subsection{Closing the bootstrap bounds}
	\label{SS: closing the bootstrap}
	In this subsection, we close the bootstrap argument by proving Proposition \ref{pro: trapped regime}. In particular, we apply Proposition \ref{pro: mixed energy estimate} to show the refined bounds \eqref{eq: refined pointwise a b}--\eqref{eq: refined pointwise energy 3}. 
	
	\begin{proof}{Proposition {\rm\ref{pro: trapped regime}}} \label{pf: trapped regime}
	
	\medskip
	\step{1}{The refined bound for $b$} The modulation equations \eqref{eq: modulation eq} implies the asymptotic
	\begin{equation} \label{eq: b asymptotics}
		b(s) = \frac{1}{s} + O\bigg(\frac{1}{s\log s}\bigg).
	\end{equation}
	This implies for $s_0\gg 1$ chosen sufficiently large, $b(s)$ decays over the rescaled time. In view of the self-similar transformation \eqref{eq: var sub}, we see for $t\geq 0$ that
	\begin{equation} \notag
		0 <b(t)\leq b(0) <b^\ast.
	\end{equation}
	Thus the refined bound \eqref{eq: refined pointwise a b} for $b(t)$ holds for $K>2$.

	\medskip
	\step{2}{The refined bound for $a$} \label{sp: a initial data} We define the function
	\begin{equation} \notag
		\kappa(s) = \frac{2a(s)\lf|\lgb(s)\rg|}{b(s)}.
	\end{equation}
	In view of the assumption \eqref{eq: initial b a}, we see $ \kappa(s_0) \in \iset=[-\tfrac{1}{2}, \tfrac{1}{2}]$ ($t=0$ when $s=s_0$), and the refined bound \eqref{eq: refined pointwise a b} for $a$ is equivalent to $\kappa(s)\in \isetst = [-1, 1]$ for all $s\in[s_0, +\infty)$. To show the latter, we compute the equation for $\kappa$ using \eqref{eq: modulation eq}:
	\begin{equation} \label{eq: the derivative of kappa}
	\begin{aligned}
		\frac{d}{ds}\kappa(s)
		& = \frac{2a_s\lgba}{b} -\frac{2ab_s\lgba}{b^2}\bigg( 1 + \frac{1}{\lgba} \bigg) \\
		& = - \frac{2\lgba}{b}\bigg[\frac{2ab}{\lgba}+O\bigg(\frac{b^2}{\sqrt{\log M}\lgba}\bigg)\bigg] \\
		&\quad +\frac{2a\lgba}{b^2} \bigg(1+\frac{1}{\lgba}\bigg) \bigg[b^2\bigg(1+\frac{2}{\lgba}\bigg) +O\bigg(\frac{b^2}{\sqrt{\log M}\lgba}\bigg)\bigg] \\
		& = \kappa\phsh b\big(1+o(1)\big) +O\bigg(\frac{b}{\sqrt{\log M}}\bigg).
	\end{aligned}
	\end{equation}
	From the largeness of $M$, we see $\kappa$ is monotone increasing/decreasing near $\kappa=1,-1$, implying that if $\kappa$ approaches the boundary of $\isetst$, it will escape from $\isetst$. We denote the moment when $\kappa$ first leave by $\sstar>0$, called the exit time. And the desired bound \eqref{eq: refined pointwise a b} holds if we prove $\sstar=+\infty$. Let us consider the following subsets of $\iset$:
	\begin{equation} \notag
	\begin{aligned}
		& \iset_+ = \big\{\kappa(s_0)\in \iset: \phsh s^{\ast}\in(s_0,+\infty)\mbox{ such that }\kappa(s^{\ast})=1\big\}, \\
		& \iset_- = \big\{\kappa(s_0)\in \iset: \phsh s^{\ast}\in (s_0,+\infty)\mbox{ such that }\kappa(s^{\ast})=-1\big\}.
	\end{aligned}
	\end{equation}
	Now the problem can be categorized as three cases: $\iset_{+}$ empty, $\iset_{-}$ empty, and both of them being nonempty. In the first case, for any $\kappa(s_0)\in \iset$, we have $\kappa(s)<1$ for all $s\in[s_0,+\infty)$. We choose 
	\begin{equation} \label{eq: choose of a 1}
		\kappa(s_0)>\frac{C}{\sqrt{\log M}},
	\end{equation}
	for $C>0$ sufficiently large. Then by \eqref{eq: the derivative of kappa}, $\kappa$ is monotone increasing, so that $\kappa(s)\geq\kappa(s_0)>-1$, and thus $s^\ast=+\infty$. An analogous statement holds for the second case. In the last case, by the $\CCC^1$ dependence of the solutions on the initial data,  we know $\iset_{+}, \iset_{-}$ are open subsets of $\iset$, and obviously disjoint. This implies, from an elementary topological argument on the connectedness, that $\iset\neq\iset_{+}\cup\iset_{-}$. Consequently, 
	\begin{equation} \label{eq: choose of a 2}
		\exists\,\kappa(s_0)\in\iset\setminus(\iset_{+}\cup\iset_{-})\neq\emptyset,
	\end{equation}
	satisfying $\sstar=+\infty$, which is what we need. In summary, we can always find suitable initial data $a(t=0)$ satisfying \eqref{eq: initial b a} such that the refined bound \eqref{eq: refined pointwise a b} for $a(t)$ holds, concluding the proof of \eqref{eq: refined pointwise a b}.
	
	\medskip
	\step{3}{The refined bound for $\E_1$} For any given vector under the Frenet basis 
	\begin{equation} \notag
		\xx = \h\alpha\phsh \er +\h\beta\phsh \etau +(1+\h\gamma)\phsh Q,
	\end{equation}
	satisfying the constraint \eqref{eq: component constraint}, the corresponding Dirichlet energy is given by 
	\begin{equation} \notag
	\begin{aligned}
		\E(\xx) 
		& = \int|\nabla\xx|^2 = -\int\xx\cdot\big(\Delta\xx+|\nabla Q|^2\xx\big) +\int|\nabla Q|^2|\xx|^2 \\
		& = \big(\h\alpha,\Ham\h\alpha\big) +\big(\h\beta,\Ham\h\beta\big) +\big(-\hs\Delta\h\gamma,\h\gamma\big) \\
		&\quad + 2\int\onez\bigg(-\h\alpha\phsh\py\h\gamma +\h\gamma\pyzy\h\alpha\bigg) +\E(Q), 
	\end{aligned}
	\end{equation}
	where we haved used \eqref{eq: part linear}. By Lemma \ref{le: localized ap solution}, straightforward computations show for any $s\in[s_0,+\infty)$ that 
	\begin{equation} \notag
		\big|\E(Q+\wzt) -\E(Q)\big|(s) \les \sqrt{b}.
	\end{equation}
	Then, in view of \eqref{eq: mHam func}, the Dirichlet energy of the solution $u$ is
	\begin{equation} \notag
	\begin{aligned}
		\E(u)(s)
		& = \E(Q+\wzt) +\big(\w,\mHam\wzt\big) +\big(\wzt,\mHam\w\big) +\big(\w,\mHam\w\big) \\
		& = \E(Q) +\big(\alpha,\Ham\alpha\big) +\big(\beta,\Ham\beta\big) +O(\sqrt{b}),
	\end{aligned}
	\end{equation}
	where the cross terms are controlled by $O(\sqrt{b})$ via the bootstrap bounds. Also, the dissipative property \eqref{eq: energy dissipative} of the Dirichlet energy implies
	\begin{equation} \label{eq: dissipative 2}
		\big(\alpha,\Ham\alpha\big) +\big(\beta,\Ham\beta\big)
		\leq \big(\alpha,\Ham \alpha\big)(s_0) +\big(\beta,\Ham\beta\big)(s_0) +O(\sqrt{b}).
	\end{equation}
	Moreover, owing to the orthogonality \eqref{eq: orthogonality}, we have the coercivity estimate
	\begin{equation} \label{eq: coercivity estimate}
		\big(\alpha,\Ham\alpha\big) +\big(\beta,\Ham\beta\big) \geq C(M)\E_1(s),
	\end{equation}
	for some universal constant $C(M)>0$ independent of $K$, and for any $s\in[s_0,+\infty)$. Its opposite direction follows directly from integration by parts
	\begin{equation} \label{eq: integration by parts}
		\big(\alpha,\Ham\alpha\big) +\big(\beta,\Ham\beta\big) \leq\E_1(s).
	\end{equation}
	Now putting together \eqref{eq: dissipative 2}, \eqref{eq: coercivity estimate}, \eqref{eq: integration by parts}, we obtain
	\begin{equation} \notag
	\begin{aligned}	
		\E_1(s) \les\big(\alpha,\Ham\alpha\big) +\big(\beta,\Ham\beta\big) \les \E_1(0) +O(\sqrt{b}) \les\db.
	\end{aligned}
	\end{equation}
	Note that the implicit constant in this inequality does not depend on $K$. Therefore by choosing $K$ large enough, the refined $\E_1$ bound \eqref{eq: refined pointwise energy 1} is proved.
	
	\medskip
	\step{4}{The refined bound of $\E_4$} A direct integration of \eqref{eq: energy estimate} with respect to $t$ shows
	there exists constant $d_2 \in (0,1)$ and some constant $C>0$ independent of $M$ such that
	\begin{equation} \label{eq: E4 estimate formula}
		\E_4(t)
		\leq \frac{\la(t)^6}{\la(0)^6} \E_4(0) + \big( 2(1-d_2)K + C \big) \la(t)^6\int_0^t \frac{b^5}{\la^8\lgba^2} d\sigma.
	\end{equation}
	To treat the first part on the RHS, we let $C_1, C_2 >0$ be two universal constants large enough compared with the implicit constant in \eqref{eq: modulation eq}, and define
	\begin{equation} \notag
		\eta_1 = 2-\frac{C_1}{\sqrt{\log M}}, \quad
		\eta_2 = 2+\frac{C_2}{\sqrt{\log M}}.
	\end{equation}
	Using the modulation equations~\eqref{eq: modulation eq}, we compute the following derivative:
	\begin{equation} \notag
	\begin{aligned}
		\frac{d}{ds}& \bigg( \frac{b\lgba^{\eta_i}}{\la} \bigg)
		= \frac{\lgba^{\eta_i}}{\la}\bigg(b_s-\frac{b\la_s}{\la} +\frac{\eta_i b_s}{\lgba}\bigg) \\
		& = \bigg(1-\frac{\eta_i}{\lgba}\bigg) \frac{\lgba^{\eta_i}}{\la} \bigg[ b_s + \bigg(1+\frac{\eta_i}{\lgba} + O\bigg(\frac{1}{\lgba^2}\bigg)\bigg) b^2 \bigg]
		\,\begin{cases} \leq 0,\;\;i=1,\\ \geq 0,\;\;i=2, \end{cases}
	\end{aligned}
	\end{equation}
	where we see the choice of $\eta_i$ is decisive to the sign. By direct integrations, it implies
	\begin{equation} \label{eq: b lambda relation}
		\frac{b(0)}{\la(0)}\lf|\frac{\lgb(0)}{\lgb(t)}\rg|^{\eta_2}
		\leq \frac{b(t)}{\la(t)}
		\leq \frac{b(0)}{\la(0)} \lf|\frac{\lgb(0)}{\lgb(t)}\rg|^{\eta_1},
	\end{equation}
	from which we have
	\begin{equation} \label{eq: E4 estimate formula term1}
		\frac{\la(t)^6}{\la(0)^6} \E_4(0) 
		\leq \frac{b(t)^6 \lf|\lgb(t)\rg|^{6 \eta_2}}{b(0)^6 \lf|\lgb(0)\rg|^{6 \eta_2}} \E_4(0) \leq b(t)^5 
		\ll \frac{b(t)^4}{\lf|\lgb(t)\rg|^2}.
	\end{equation}
	For the second part on the RHS of \eqref{eq: E4 estimate formula}, it follows from \eqref{eq: modulation eq} that
	\begin{equation} \notag
	\lf\{\begin{aligned}
		& b = -\frac{\la_s}{\la}+O(b^3) \leq -\big(1+2b(0)^2\big)\la \la_t, \\
		& -b_t = -\frac{b_s}{\la^2} \leq \bigg(1+\frac{2}{\lf| \lgb(0)\rg|}\bigg) \frac{b^2}{\la^2}, \\
		& 4+\frac{1}{\lgba} \leq 4+\frac{1}{\lf|\lgb(0)\rg|},
	\end{aligned}\rg.
	\end{equation}
	using which we compute
	\begin{equation} \notag
	\begin{aligned}
		\int_0^t \frac{b^5}{\la^8 \lgba^2} &\, d\sigma
		\leq -\big(1+2b(0)^2\big) \int_0^t \frac{\la_t}{\la^7} \bflog d\sigma \\
		& = \frac{1}{6} \big(1+2b(0)^2\big)	\bigg[ \frac{b^4}{\la^6 \lgba^2}\bigg|_0^t - \int_0^t\frac{b_t b^3}{\la^6 \lgba^2}\bigg(4+\frac{1}{\lgba}\bigg) d\sigma \bigg] \\
		& \leq \frac{1}{6} \big(1+2b(0)^2\big) \bigg[ \frac{b^4}{\la^6\lgba^2} \bigg|_0^t \\
		&\qquad\qquad + \lf(1+\frac{2}{\lf|\lgb(0)\rg|}\rg) \lf(4+\frac{1}{\lf|\lgb(0)\rg|}\rg) \int_0^t \frac{b^5}{\la^8\lgba^2} d\sigma \bigg],
	\end{aligned}
	\end{equation}
	and hence
	\begin{equation} \label{eq: E4 estimate formula term2}
	\begin{aligned}
		\int_0^t \frac{b^5}{\la^8 \lgba^2} d\sigma
		& \leq \bigg(\frac{1}{2}+O\bigg(\frac{1}{\lf|\lgb(0)\rg|}\bigg)\bigg) \frac{b^4}{\la^6\lgba^2} \bigg|_0^t \\
		& \leq \bigg( \frac{1}{2} + O\bigg(\frac{1}{\lf|\lgb(0)\rg|}\bigg)\bigg) \frac{b^4}{\la^6 \lgba^2}.
	\end{aligned}
	\end{equation}
	Injecting \eqref{eq: E4 estimate formula term1}, \eqref{eq: E4 estimate formula term2} into \eqref{eq: E4 estimate formula}, we obtain
	\begin{equation} \notag
		\E_4(t) \leq \big((1-d_2) K+C\big)\frac{b(t)^4}{\lf|\lgb(t)\rg|^2},
	\end{equation}
	for some constant $C>0$ independent of $K,M$. By choosing $K>0$ large enough, we see there exist $0<\eta <1$ such that $(1-d_2)K+C \leq K(1-\eta)$, and thus the desired bound \eqref{eq: refined pointwise energy 3} follows.
	
	\medskip
	\step{5}{The refined bound of $\E_2$} A direct interpolation between the bounds of $\E_1, \E_4$ is insufficient to show \eqref{eq: refined pointwise energy 2}. Thus we treat it attentively. By \eqref{eq: hv eq}, \eqref{eq: part linear}, we see
	\begin{equation}
		u\wedge\Delta u 
		= u\wedge\big(\Delta u+|\nabla Q|^2u\big) = \hJ\mHam\hw,
	\end{equation}
	from which we compute:
	\begin{equation} \label{eq: E2 comp 1}
	\begin{aligned}
		\frac{1}{2}\frac{d}{dt} \int\big|u\wedge\Delta u\big|^2 
		& = -\bo\int (u\wedge\Delta u)\cdot\big[u\wedge(u\wedge\Delta u)\big]\wedge\Delta u \\
		& \quad +\ao\int(u\wedge\Delta u)\cdot u\wedge\Delta(u\wedge\Delta u) \\
		& \quad -\bo\int(u\wedge\Delta u)\cdot u\wedge\Delta\big[u\wedge(u\wedge\Delta u)\big] \\
		& = \frac{1}{\la^4}\; \bigg\{ \bo\int\hJ\mHam\hw\cdot(\hJ^2\mHam\hw)\wedge\mHam\hw \\
		& \qquad\qquad -\ao\int\hJ\mHam\hw \cdot(\hJ\mHam)^2\hw +\bo\int\hJ\mHam\hw\cdot\hJ\mHam(\hJ^2\mHam\hw)\bigg\}.
	\end{aligned}
	\end{equation}
	Note that the LHS is a rough equivalent of $\E_2$ disregarding the $\lambda$ power
	\begin{equation} \label{eq: lower bound to E2}
		\la^2\hs\hs\hs\int\hs\lf|u\wedge\Delta u\rg|^2
		= \hs\hs\int\hs\big|\hJ\mHam\hw\big|^2 
		= \hs\hs\int\hs\big|\hJ\mHam\w\big|^2 \hs+\hs O\bigg(\hs\int\big|\hJ\mHam\wzt\big|^2\hs\bigg) 
		= \E_2 + O\big(b^2\lgba^2\hs\big).
	\end{equation}
	In addition, from \eqref{eq: pointwise energy 2} and Appendix~\ref{S: appendix A}, we can estimate each term on the RHS of \eqref{eq: E2 comp 1}. Indeed, by \eqref{eq: wzt brute estimate}, \eqref{eq: ip bound 11}, \eqref{eq: ip bound 15}, \eqref{eq: ip bound 21}, the first term is bounded by
	\begin{equation}
	\begin{aligned}
		\int \hJ\mHam\hw\cdot(\hJ^2\mHam\hw)\wedge\mHam\hw 
		& = \int\mHam\hw\cdot\big(\hJ\mHam\hw\wedge\hJ^2\mHam\hw\big) 
			= \int\big(\mHam\hw\cdot(Q+\hw)\big)\big|\hJ\mHam\hw\big|^2 \\
		& \leq \|\mHam\hw\|_{L^\infty} \Big(\|\hJ\mHam\wzt\|_{L^2}^2 + \|\w_2\|_{L^2}^2 \Big) \\
		& \les b\Big(b^2\lgba^2+Kb^2\lgba^6\Big)
		\les K b^3 \lgba^6.
	\end{aligned} \notag
	\end{equation}
	Using Lemma \ref{le: gain of 2d}, \eqref{eq: ip bound 31}, \eqref{eq: ip bound 34}, we estimate the second one
	\begin{equation}
	\begin{aligned}
		- \int \hJ \mHam \hw \cdot(\hJ\mHam)^2 \hw
		& = - \int\hJ\mHam\wzt\cdot(\hJ\mHam)^2 \wzt -\int\w_2\cdot(\hJ\mHam)^2 \wzt\\
		& \quad -\int \hJ \mHam \wzt \cdot \hJ \mHam \w_2 - \int \w_2 \cdot \hJ \mHam \w_2\\
		& \leq b\, \|\mHam\wzt\|_{L^2}^2 + \|\mHam\hJ\w_2\|_{L^2} \phsh\|\hJ\mHam\wzt\|_{L^2} \\
		& \quad + \|\mHam\wzt\|_{L^2} \phsh \|\mHam\w_2\|_{L^2} + \|\w_2\|_{L^2} \phsh \|\hJ\mHam\w_2\|_{L^2} \\
		& \les b^3\lgba^2 + 2\sqrt{K}b^3 +K b^3 \lgba^2
		\les K b^3 \lgba^2.
	\end{aligned} \notag
	\end{equation}
	Similarly, from \eqref{eq: ip bound 27}, \eqref{eq: ip bound 34}, the third term is bounded by
	\begin{equation}
	\begin{aligned}
		\int\hJ\mHam\hw \cdot \hJ\mHam(\hJ^2\mHam\hw)
		& \leq \Big(\|\mHam\wzt\|_{L^2}+\|\mHam\w\|_{L^2}\Big) \Big(\|\mHam(\hJ^2\mHam\wzt)\|_{L^2} +\|\mHam(\hJ\w_2)\|_{L^2}\Big)  \\
		& \les \Big(b\lgba+\sqrt{\E_2}\Big) \frac{\sqrt{K}b^2}{\lgba}
		\les K b^3 \lgba^2.
	\end{aligned} \notag
	\end{equation}
	We insert these bounds into \eqref{eq: E2 comp 1} to obtain
	\begin{equation} \notag
		\frac{d}{dt} \int\big|u\wedge\Delta u\big|^2
		\les \frac{Kb^3 \lgba^2}{\la^4}.
	\end{equation}
	Integrating this in time from $0$ to $t$ and applying \eqref{eq: lower bound to E2}, we have 
	\begin{equation}
		\E_2(t) \les b(t)^2\lf|\lgb(t)\rg|^2 +\la(t)^2 b(0)^{10} +K\la(t)^2 \int_0^t \frac{b(\sigma)^3 \lf|\log b(\sigma)\rg|^3}{\la(\sigma)^4} \d\sigma.
	\label{eq: E2 comp 2}
	\end{equation}
	From \eqref{eq: b lambda relation}, there holds
	\begin{equation} \notag
		\la(t)^2 b(0)^{10}
		\leq b(0)^{10} \bigg( b(t)\frac{\la(0)}{b(0)} \bigg|\frac{\lgb(t)}{\lgb(0)}\bigg|^{\eta_2} \bigg)^2 
		\les b(t)^2 \lf|\lgb(t)\rg|^5.
	\end{equation}
	Moreover, from \eqref{eq: modulation eq}, we have $b^2\les -b_s$, which together with \eqref{eq: b lambda relation} gives	
	\begin{equation} \notag
	\begin{aligned}
		\int_0^t \frac{b^3\lgba^3}{\la^4} d\sigma
		& \les -\int_0^t \frac{b b_s\lgba^3}{\la^4} d\sigma 
		\les -\frac{b(0)^2\lf|\log b(0)\rg|^{2\eta_1}}{\la(0)^2} \int_0^t \frac{b_t}{b\lgba^{2\eta_1-3}} d\sigma \\ 
		& \les \frac{b(0)^2\lf|\log b(0)\rg|^{2}}{\la(0)^2}
		\les \frac{b^2\lgba^{2\eta_2}}{\la^2\lf|\log b(0)\rg|^{2\eta_2-2}} \les \frac{b^2\lgba^{5}}{\la^2}.
	\end{aligned}
	\end{equation}
	Injecting these into \eqref{eq: E2 comp 2}, we obtain
	\begin{equation} \notag
		\E_2 \les K b^2 \lgba^5 \leq \frac{K}{2} b^2\lgba^6,
	\end{equation}
	where, in view of \eqref{eq: b asymptotics}, the last inequality holds if $b(0)$ is chosen small enough. This is the refined bound \eqref{eq: refined pointwise energy 2}, and thus ends the proof.	
	\end{proof}
	\begin{remark}
		{\rm (i)} In the above proof we have actually specified the constants $M,K$. To be explicit, we first let $M\gg 1$ be a large constant compared with those implicit constants from earlier estimates. Next the constant $b^{\ast} = b^{\ast}(M)$ is chosen to be a small upper bound of $b(t)$, making $b(t)$ an higher-order infinitesimal relative to $\frac{1}{\log M}$. Then we set $K\gg 1$ in comparison with other unnamed constants, such that the bootstrap bounds \eqref{eq: pointwise energy 1}--\eqref{eq: pointwise energy 3} and \eqref{eq: pointwise a b} for $b(t)$  hold. {\rm (ii)} We note that the initial bounds \eqref{eq: initial b a}, \eqref{eq: initial energy} imply an open initial data set of $b, \w$ stable under small perturbation, while $a(0)$ requires precise selection relying on $b(0), \w(0)$, through a topological argument {\rm (}see \eqref{eq: choose of a 1}, \eqref{eq: choose of a 2}{\rm )}. This actually yields a codimension one initial data set of $u$, as mentioned in Theorem~\ref{th: main th}.
	\end{remark}
	
	\section{Description on the singularity formation}
	\label{S: des on singularity formation}
	We are now ready to prove the Theorem~\ref{th: main th}. In this subsection, we prove the finite time blowup, show the convergence of the phase $\vart$, give a sharp description on the blowup speed $\la$, and prove the strong convergence of the excess energy.
	
	\begin{proof}{Theorem {\rm\ref{th: main th}}}
	By Proposition \ref{pro: trapped regime}, $\w$ has been trapped in the regime of \eqref{eq: pointwise a b}, \eqref{eq: pointwise energy 1}--\eqref{eq: pointwise energy 3}. Thus $u$ satisfying \eqref{eq: decomposition} exists for $s\in[s_0,+\infty)$. Now we suppose the lifespan of $u$ is $T\leq+\infty$, and study the corresponding asymptotic behavior.
	
	\medskip
	\step{1}{Finite time blowup} We apply \eqref{eq: b lambda relation} and $\la_s/\la \les -b$ to obtain
	\begin{equation} \notag
	\begin{aligned}
		\frac{d}{dt} \sqrt{\la}
		= \frac{\la_s}{2\la^2 \sqrt{\la}} \les -\frac{b}{\lambda\sqrt{\lambda}} \les -\frac{1}{\sqrt{b}} \bigg( \frac{b(0)}{\la(0)} \bigg|\frac{\lgb(0)}{\lgb}\bigg|^{\eta_2} \bigg)^{\frac{3}{2}}.
	\end{aligned}
	\end{equation}
	By the smallness and monotonicity of $b$ \eqref{eq: b asymptotics}, there exists constant $C(u_0)>0$ such that
	\begin{equation} \notag
		\frac{d}{dt} \sqrt{\la} < -C(u_0) <0,
	\end{equation}
	which together with \eqref{eq: b lambda relation} implies
	\begin{equation} \label{eq: the blowup time}
		T <+\infty, \quad\mbox{with}\quad \la(T)=b(T)=0.
	\end{equation}
	Thus the scaling $\la(t)\to 0$ as $t \to T$, and the solution $u$ blows up at $T$.
	
	\medskip
	\def\bd{b^{-\delta}}
	\def\chibd{\chi_{\B}}
	\def\varpb{\varp_{\delta}}
	\def\ati{\ti{a}}
	\step{2}{Convergence of the phase} As illustrated in Remark~\ref{re: non convergence of theta}, the modulation equations \eqref{eq: modulation eq} on parameter $a$ is too crude to show the convergence of the phase. We claim the following refined bound with additional logarithmic gain
	\begin{equation} \label{eq: refined bound a}
		|a(t)|\les C(\delta) \frac{b(t)}{\lf|\log b(t)\rg|^{\frac{3}{2}}},
	\end{equation}
	for some universal small constant $\delta>0$ and a large constant $C(\delta)$. Applying \eqref{eq: refined bound a} with \eqref{eq: b asymptotics}, we obtain
	\begin{equation} \notag
		\lim_{s\to +\infty}\vart(s) -\vart(s_0) 
		\leq \int_{s_0}^{+\infty} |a(s)|\,ds \leq \int_{s_0}^{+\infty} \frac{ds}{s(\log s)^{\frac{3}{2}}} < +\infty,
	\end{equation}
	which implies the convergence of the phase
	\begin{equation} \notag
		\vart(t) \to \vart(T)\in \RR, \;\;\mbox{as}\;\; t\to T.
	\end{equation}
	It remains to prove \eqref{eq: refined bound a}. To this end, we choose the varying scale 
	\begin{equation} \notag
		\B=\bd,
	\end{equation}
	for sufficiently small constant $\delta>0$, and repeat the computations as in Proposition~\ref{pro: modulation equation}. More precisely, we define the direction
	\begin{equation} \label{eq: varpb}
		\varpb = \chibd\varl\phi\begin{bmatrix}
			\ao\\ \bo\\ 0
		\end{bmatrix},
		\quad\mbox{with}\quad 
		\lf\{\begin{aligned}
			&\, \|\varpb\|_{L^2}\sim \sqrt{\lgba}, \\
			&\, (\varp_{1,0},\mHamp\varpb) \sim \lgba, \\
			&\, (\varp_{0,1},\mHamp\varpb) =0.
		\end{aligned}\rg.
	\end{equation}
	We take the scalar products of \eqref{eq: w equation} with $\mHamp\varpb$, and the resulting equation is
	\begin{equation} \notag
	\begin{aligned}
		0 & = \big(\prs\w,\mHamp\varpb\big) -\frac{\la_s}{\la}\big(\varl\w,\mHamp\varpb\big) +\big(\ao\hJ\mHam\w-\bo\hJ^2\mHam\w,\mHamp\varpb\big) \\
		& \quad +\big(\mmodt,\mHamp\varpb\big) +\big(\errzt,\mHamp\varpb\big) +\big(\rest,\mHamp\varpb\big).
	\end{aligned}
	\end{equation}
	Using Appendix~\ref{S: appendix A}, we have
	\begin{equation} \notag
		\Big|\frac{\la_s}{\la}\Big|
		\,\big|\big(\varl\w,\mHamp\varpb\big)\big| \les b^{1-C\delta}\bigg(\int\frac{|\mHamp(\varl\w)|^2}{(1+y^4)(1+\lf|\log y\rg|^2)}\bigg)^{\frac{1}{2}} \les b^{1-C\delta}\sqrt{\E_4}.
	\end{equation}
	Recalling \eqref{eq: modulation comp 2}, \eqref{eq: modulation comp 3}, we see the linear terms are bounded by
	\begin{equation} \notag
		\big|\big(\hJ\mHam\w,\mHamp\varpb\big)\big| +\big|\big(\hJ^2 \mHam\w,\mHamp\varpb\big)\big|
		\les \sqrt{\lgba} \sqrt{\E_4},
	\end{equation}
	where the $\sqrt{\lgba}$ comes from the $L^2$ norm of $\varpb$. Moreover, from \eqref{eq: localized mod vec}, \eqref{eq: modulation eq}, and also the orthogonality of $\varp_{0,1}$ and $\varpb$ \eqref{eq: varpb}, we observe
	\begin{equation} \notag
		\big(\mmodt,\mHamp\varpb\big) = \big(\mmod,\mHamp\varpb\big) = a_s\big(\varp_{1,0},\mHamp\varpb\big) + O\big(b^{1-C\delta}U(t)\big).
	\end{equation}
	The estimate for $\errzt$ follows from \eqref{eq: localized weighted bound 3}, and the bound for $\rest$ can be obtained by simply repeating the computation \eqref{eq: modulation comp expand 4}, \eqref{eq: modulation comp 4}, with an extra $b^{1+C\delta}$ smallness from the $L^2$ norm of $\mHamp\varpb$. Combining these bounds, we obtain
	\begin{equation} \label{eq: b delta modulation}
		a_s\big(\varp_{1,0},\mHamp\varpb\big)+ \big(\prs\w,\mHamp\varpb\big) 
		\les \frac{b^2}{\lgba^{\frac{1}{2}}}.
	\end{equation}
	Now we let 
	\begin{equation} \notag
		\ati = a +\frac{(\w,\mHamp\varpb)}{(\varp_{1,0},\mHamp\varpb)} 
		= a +O\bigg(b^{-C\delta}\frac{b^2}{\lgba}\bigg).
	\end{equation} 
	and its derivative follows from \eqref{eq: b delta modulation}:
	\begin{equation} \notag
		\ati_s = a_s+ \frac{(\prs\w,\mHamp\varpb)}{(\varp_{1,0},\mHamp\varpb)}
		+ O\bigg(\frac{b^{3-C\delta}}{\lgba^2}\bigg) \les \frac{b^2}{\lgba^{\frac{3}{2}}}.
	\end{equation}
	Integrating this from $s=+\infty$ where $\ati=0$ to the present time $s$ using \eqref{eq: b asymptotics}, we obtain the bound
	\begin{equation} \notag
		|\ati| \les \frac{b}{\lgba^{\frac{3}{2}}}.
	\end{equation}
	This yields the corresponding bound for $a$ \eqref{eq: refined bound a}.

	\medskip
	\step{3}{Derivation of the blowup speed} We recall the modulation equations~\eqref{eq: modulation eq} that
	\begin{equation} \notag
		b_s = -b^2 \bigg(1+\frac{2}{\lgba}\bigg) + O\bigg(\frac{b^2}{\sqrt{\log M} \lgba} \bigg),
	\end{equation}
	From the rough asymptotics \eqref{eq: b asymptotics}, we assume $b = 1/s + f/s^2$ with $|f| \ll s$, and $s$ sufficiently large to derive the refined estimate
	\begin{equation} \notag
		b(s) = \frac{1}{s} -\frac{2}{s\log s} +O\bigg(\frac{1}{s(\log s)^2}\bigg).
	\end{equation}
	Moreover, from $b+\la_s/\la = O(b^3)$, we see the concentration scale $\la$ admits
	\begin{equation} \label{eq: lambda on s}
		\frac{\la_s}{\la} = -\frac{1}{s} +\frac{2}{s\log s} +O\bigg( \frac{1}{s(\log s)^2} \bigg),
	\end{equation}
	from which we have
	\begin{equation} \notag
		\frac{1}{s} = C(u_0)\big(1+o(1)\big)\frac{\la}{\lf|\log\la\rg|^2}.
	\end{equation}
	Using $-1/s \sim \la \la_t$ by \eqref{eq: lambda on s}, we get
	\begin{equation} \label{eq: the derivative of lambda}
		\la_t = -C(u_0)\big(1+o(1)\big)\frac{1}{\lf|\log\la\rg|^2}.
	\end{equation}
	Integrating this from  $t$ to $T$, the blowup speed \eqref{eq: lambda} follows
	\begin{equation} \label{eq: blowup speed}
		\la(t) = C(u_0)\big(1+o(1)\big)\frac{(T-t)}{\lf|\log(T-t)\rg|^2}.
	\end{equation}
	Combining \eqref{eq: the derivative of lambda} with \eqref{eq: blowup speed}, and applying $b+\la_s/\la =O(b^3)$ again, we obtain
	\begin{equation} \label{eq: b t asymptotics}
		b(t) = C(u_0)\big(1+o(1)\big)\frac{(T-t)}{\lf|\log(T-t)\rg|^4}.
	\end{equation}

	\medskip
	\step{4}{Convergence of the excess energy} From \eqref{eq: decomposition} we have the decomposition
	\begin{equation} \notag
		u =e^{\vart\R}Q_{\lambda} + \ti{u}, 
		\quad\mbox{where}\quad \ti{u}:=e^{\vart\R}\hv_{\lambda}.
	\end{equation}
	Owing to the dissipative Dirichlet energy \eqref{eq: energy dissipative} and the energy-critical scaling, the $\dot{H}^1$ norm of $\ti{u}$ is bounded
	\begin{equation} \notag
		\| \nabla \ti{u} \|_{L^2}
		\leq \|\nabla u\|_{L^2} +\|\nabla(e^{\vart\R}Q_{\lambda})\|_{L^2} 
		\leq \|\nabla u_0\|_{L^2} +\|\nabla Q\|_{L^2} \les C(u_0).
	\end{equation}
	Moreover, we have the $\dot{H}^2$ bound
	\begin{equation} \label{eq: ti u H2}
		\|\Delta\ti{u}\|_{L^2} \les C(u_0).
	\end{equation}
	Indeed, from \eqref{eq: part linear}, \eqref{eq: mHam}, \eqref{eq: ip bound 27}, we have
	\begin{equation} \label{eq: Lap ti u}
	\begin{aligned}
		\|\Delta \ti{u}\|_{L^2} = \frac{1}{\la^2} \int |\Delta \hv|^2
		\les \frac{1}{\la^2}\bigg(\int|\mHam\hw|^2 +\int\frac{|\hw|^2}{1+y^8} \bigg) 
		\les \frac{1}{\la^2} \big(\E_2+b^2\lgba^2\big).
	\end{aligned}
	\end{equation}
	The inequality \eqref{eq: E2 comp 2} with the explicit asymptotics \eqref{eq: blowup speed}, \eqref{eq: b t asymptotics} gives 
	\begin{equation} \notag
		\E_2 \les b^2\lgba^2 +\la^2b(0)^{10} +\la^2\int_0^t\frac{dt}{(T-t)\lf|\log(T-t)\rg|^2} 
		\les b^2\lgba^2 +\la^2.
	\end{equation}
	Injecting this into \eqref{eq: Lap ti u} with an application of \eqref{eq: b lambda relation} yields the desired $\dot{H}^2$ boundedness. By a simple localization process, this boundedness yields a strong convergence of $\nabla \ti{u}$ outside the blowup point (the origin). More precisely, there exists $u^{\ast} \in \dot{H}^1$, such that for any $R>0$ there holds
	\begin{equation} \notag
		\|\nabla u-\nabla u^{\ast}\|_{L^2(|x|>R)}
		\sim \|\nabla\ti{u}-\nabla u^{\ast}\|_{L^2(|x|>R)} \to 0,
		\quad\mbox{as}\quad t \to T,
	\end{equation}
	which yields \eqref{eq: singularity formation}. Now the convergence \eqref{eq: singularity formation} together with the $\dot{H}^2$ bound \eqref{eq: ti u H2} gives \eqref{eq: propagation of regularity}. This concludes the proof of Theorem \ref{th: main th}.
	\end{proof}
	\begin{appendix}
		
		\section{Coercivity estimates and interpolation estimates}
		\label{S: appendix A}
		In this appendix, we list some results on the coercivity of the Hamiltonians $\Ham, \Ham^2$, and also the interpolation estimates used in the proof of Proposition \ref{pro: modulation equation}, \ref{pro: mixed energy estimate}. The complete proof can be found in \cite{2011Merle_SchMapBlowup}. 
		\begin{lemma}[Coercivity of $\Ham$ {\rm \cite{2011Merle_SchMapBlowup}}] \label{le: Ham coercivity} 
		Let $M\geq 1$ be a large enough universal constant. Let $\varp_M$ be given by \eqref{eq: varp M}. Then there exists a universal constant $C(M)>0$ such that for all radially symmetric function $u\in H^1$ satisfying
		\begin{equation} \notag
			\int\frac{|u|^2}{y^4(1+\lf|\log y\rg|)^2} +\int\big|\py(\A u)\big|^2<+\infty.
		\end{equation}
		and the orthogonality conditions
		\begin{equation} \notag
			(u,\varp_M)=0,
		\end{equation}
		there holds
		\begin{equation} \label{eq: Ham coercivity}
		\begin{aligned}
			&\int_{y\geq 1}\frac{|\py^2 u|^2}{1+\lf|\log y\rg|^2}
			+\int\frac{|\py u|^2}{y^2(1+\lf|\log y\rg|)^2} +\int\frac{|u|^2}{y^4(1+\lf|\log y\rg|)^2} \\
			&\quad\leq C(M)\bigg[ \int\frac{|\A u|^2}{y^2(1+\lf|\log y\rg|)^2} +\int\big|\py(\A u)\big|^2\bigg]
			\les C(M)\int|\Ham u|^2.
		\end{aligned}
		\end{equation}
		\end{lemma}
		\begin{lemma}[Coercivity of $\Ham^2$ {\rm \cite{2011Merle_SchMapBlowup}}] \label{le: Ham2 coercivity} 
		Assume the conditions in Lemma \ref{le: Ham coercivity}. Then there exists a universal constant $C(M)>0$ such that for all radially symmetric function $u$ satisfying
		\begin{equation} \notag
		\begin{aligned}
			&\int|\py\A\Ham u|^2 +\int\frac{|\A\Ham u|^2}{y^2(1+y^2)} +\int\frac{|\Ham u|^2}{y^4(1+\lf|\log y\rg|)^2}\\
			&\qquad +\int\frac{|u|^2}{y^4(1+y^4)(1+\lf|\log y\rg|)^2} +\int\frac{(\py u)^2}{y^2(1+y^4)(1+\lf|\log y\rg|)^2}<+\infty.
		\end{aligned}
		\end{equation}
		and the orthogonality conditions
		\begin{equation} \notag
			(u,\varp_M)=0,\quad (\Ham u,\varp_M)=0,
		\end{equation}
		there holds
		\begin{equation} \label{eq: Ham2 coercivity}
		\begin{aligned}
			&\int\frac{|\Ham u|^2}{y^4(1+\lf|\log y\rg|)^2}
			+\int\frac{|\py\Ham u|^2}{y^2(1+\lf|\log y\rg|)^2} +\int\frac{|\py^4 u|^2}{(1+\lf|\log y\rg|)^2} \\
			&\quad +\int\frac{|\py^3 u|^2}{y^2(1+\lf|\log y\rg|)^2} +\int\frac{|\py^2 u|^2}{y^4(1+\lf|\log y\rg|)^2} \\
			&\quad +\int\frac{|\py u|^2}{y^2(1+y^4)(1+\lf|\log y\rg|)^2} +\int\frac{|u|^2}{y^4(1+y^4)(1+\lf|\log y\rg|)^2} \leq C(M)\int|\Ham^2 u|^2.
		\end{aligned}
		\end{equation}
		\end{lemma}
		We recall the notations
		\begin{equation} \notag
			\w=\begin{bmatrix}
				\alpha\\ \beta\\ \gamma
			\end{bmatrix}=\w^\perp+\gamma\,\ez,
			\qquad \w_2=\hJ\mHam\w^\perp=\w_2^0+\w_2^1.
		\end{equation}
		and assume the bootstrap bounds \eqref{eq: pointwise energy 1}--\eqref{eq: pointwise energy 3}, then the following interpolation bounds are the consequences of the coercivity estimate \eqref{eq: Ham coercivity}, \eqref{eq: Ham2 coercivity} and the regularity of $\w$ at the origin ensured by the smoothness of the LL flow $u$.
		\begin{lemma}[Interpolation bounds for $\w^\perp$ {\rm \cite{2011Merle_SchMapBlowup}}]
			\label{le: bounds for wp}
			There holds:
			\begin{align}
			& \int\hs\hs\frac{|\w^\perp|^2}{y^4(1\hs+\hs y^4)(1\hs+\hs\lf|\log y\rg|^2)} +\hs\int\hs\hs\frac{|\py^i \w^\perp|^2}{y^2(1\hs+\hs y^{6-2i})(1\hs+\hs\lf|\log y\rg|^2)} \les C(M)\E_4, \quad 1\leq i\leq 3,
				\label{eq: ip bound 3}\\
			& \int_{|y|\geq 1} \frac{|\py^i \w^\perp|^2}{(1+y^{4-2i})(1+\lf|\log y\rg|^2)}
			\les C(M) \E_2, \quad 1 \leq i \leq 2,
			\label{eq: ip bound 4}\\
			& \int_{|y|\geq 1} \frac{(1+ \lf| \log y\rg|^C)|\py^i \w^\perp|^2}{y^2(1+\lf|\log y\rg|^2)(1+y^{6-2i})}
				\les b^4\lgba^{C_1(C)}, \quad 0\leq i\leq 3,
				\label{eq: ip bound 5}\\
			& \int_{y\geq 1}\frac{1+\lf|\log y\rg|^C}{y^2(1+\lf|\log y\rg|^2)(1+y^{4-2i})}|\py^i\w^\perp|^2 \les b^3\lgba^{C_1(C)},  \quad 0\leq i\leq 2,
				\label{eq: ip bound 6}\\
			& \int_{|y|\geq 1}|\py\mHam\w^\perp|^2\les b^3\lgba^6,
				\label{eq: ip bound 7}\\
			& \big\|\w^\perp\big\|_{L^\infty}\les \db,
				\label{eq: ip bound 8}\\
			& \lf\| \mA\w^\perp \rg\|_{L^\infty}^2
			\les b^2\lgba^9,
			\label{eq: ip bound 9}\\
			& \int_{y \leq 1} \frac{|\mA \w^\perp|^2}{y^6(1+\lf|\log y\rg|^2)}
			\les C(M)\E_4,
			\label{eq: ip bound 10}\\
			& \lf\|\frac{\mA\w^\perp}{y^2(1\hs+\hs\lf|\log y\rg|)}\rg\|_{L^\infty(y\leq 1)}^2 \hs\hs+\lf\|\frac{\Delta\mA\w^\perp}{1\hs+\hs\lf|\log y\rg|}\rg\|_{L^\infty(y \leq 1)}^2 \hs\hs+\lf\|\frac{\mHam\w^\perp}{y(1\hs+\hs\lf|\log y\rg|)}\rg\|_{L^\infty(y\leq 1)}^2 \hs\hs\les b^4,
			\label{eq: ip bound 11}\\
			& \lf\| \frac{|\Ham \alpha| + |\Ham \beta|}
			{y(1+ \lf|\log y\rg|)}\rg\|_{L^\infty(y\leq 1)}^2 \les b^4,
			\label{eq: ip bound 12}\\
			& \lf\| \frac{\w^\perp}{y} \rg\|_{L^\infty(y\leq 1)}^2
			+ \lf\| \py \w^\perp \rg\|_{L^\infty(y\leq 1)}^2
			\les b^4,
			\label{eq: ip bound 13}\\
			& \lf\|\frac{\w^\perp}{y}\rg\|_{L^\infty(y\geq 1)}^2 +\lf\|\py\w^\perp\rg\|_{L^\infty(y\geq 1)}^2 
				\les b^2\lgba^8,
				\label{eq: ip bound 14}\\
			& \lf\|\frac{\w^\perp}{1+y^2}\rg\|_{L^\infty}^2 +\lf\|\frac{\py\w^\perp}{1+y}\rg\|_{L^\infty}^2 +\lf\|\pr_{y}^2 \w^\perp \rg\|_{L^\infty(y \geq 1)}^2 \les C(M)\,b^2\lgba^2.
				\label{eq: ip bound 15}
			\end{align}
		\end{lemma}
		\begin{lemma}[Interpolation bounds for $\gamma$ {\rm \cite{2011Merle_SchMapBlowup}}]
			\label{le: bounds for gamma}
			There holds:
			\begin{align}
			& \int\frac{|\gamma|^2}{y^6(1+y^2)(1+\lf|\log y\rg|^2)} +\int\frac{|\py\gamma|^2}{y^4(1+y^{4-2i})(1+\lf|\log y\rg|^2)} \notag\\
			& \quad +\int\frac{|\py^i\gamma|^2}{y^2(1+y^{6-2i})(1+ \lf|\log y\rg|^2)} \les \db\Ebb, \quad 2\leq i\leq 3,
				\label{eq: ip bound 16}\\
			& \int_{y \geq 1} \frac{1 + \lf| \log y \rg|^C}{y^4(1+ \lf| \log y\rg|^2)(1+ y^{4-2i})}|\py^i \gamma|^2
			\les b^4 \lgba^{C_1(C)}, \quad 0 \leq i\leq 2,
			\label{eq: ip bound 17}\\
			& \int_{y\geq 1} \frac{1+\lf|\log y\rg|^C}{y^{6-2i}(1+\lf|\log y\rg|^2)}|\py^i\gamma|^2 \les b^3\lgba^{C_1(C)}, \quad 0\leq i\leq2,
				\label{eq: ip bound 18}\\
			& \int\frac{|\A\py\gamma|^2}{y^4(1+\lf|\log y\rg|^2)}\les \db\Ebb,
				\label{eq: ip bound 19}\\
			& \lf\|\gamma\frac{1+|y|}{|y|}\rg\|_{L^\infty}
			\les \db,
			\label{eq: ip bound 20}\\
			& \lf\| \frac{(1+|y|) \gamma}{y^2} \rg\|_{L^\infty}^2
			+ \lf\| \py \gamma \rg\|_{L^\infty}^2
			\les b^2 \lgba^8,
			\label{eq: ip bound 21}\\
			& \lf\| \frac{\gamma}{y(1+y)} \rg\|_{L^\infty}^2
			+ \lf\| \frac{\py \gamma}{y} \rg\|_{L^\infty}^2
			\les C(M)b^3 \lgba^2.
			\label{eq: ip bound 22}\\
			& \int | \Delta \gamma |^2
			\les \db \E_2 + b^2\lgba^2,
			\label{eq: ip bound 23}\\
			& \|\Delta\gamma\|_{L^\infty(y\geq1)}^2 \les b^3\lgba^8,
				\label{eq: ip bound 24}\\
			& \int|\w^\perp|^2|\Delta^2\gamma|^2 +\int_{y\geq1}|\Delta^2\gamma|^2 \les \db\Ebb.
				\label{eq: ip bound 25}
			\end{align}
		\end{lemma}
		\begin{lemma}[Interpolation bounds for $\w_2$ {\rm \cite{2011Merle_SchMapBlowup}}]
			\label{le: bounds for w2}
			There holds:
			\begin{align}
			& \int |\w_2|^2 =
			\E_2 + O\big( b^2\lgba^2 + \db \E_2 \big),
			\label{eq: ip bound 26}\\
			& \int |\mHam \w|^2
			\les \E_2 + b^2\lgba^2,
			\label{eq: ip bound 27}\\
			& \int \frac{|\mHam \w|^2}{(1+y^4)(1+ \lf| \log y\rg|^2)}
			\les C(M) \E_4,
			\label{eq: ip bound 28}\\
			& \int \frac{|\w_2^0|^2}{(1+y^4)(1+ \lf| \log y\rg|^2)}
			\les C(M) \E_4,
			\label{eq: ip bound 29}\\
			& \int \frac{|\w_2^1|^2}{(1+y^4)(1+ \lf| \log y\rg|^2)}
			\les \db\lf( \E_4 + \bflog \rg),
			\label{eq: ip bound 30}
			\end{align}
			\begin{align}
			& \int | \mHam \w_2 |^2
			\les C(M)\lf( \E_4 + \bflog \rg),
			\label{eq: ip bound 31}\\
			& \int \big| \hJ \mHam \w_2 \big|^2
			\les \E_4 + \bflog,
			\label{eq: ip bound 32}\\
			& \int \big| \mHam \w_2^1 \big|^2
			+ \int \lf| \R \mHam \big( \R^2 \w_2^1 \big) \rg|^2
			\les \db\lf( \E_4 + \bflog \rg),
			\label{eq: ip bound 33}\\
			& \int |\mHam \hJ \w_2 |^2
			\les C(M) \bflog.
			\label{eq: ip bound 34}
			\end{align}
		\end{lemma}
		
		\section{Proof of Lemma \ref{le: gain of 2d}}
		\label{S: appendix B}
		This appendix is devoted to prove Lemma \ref{le: gain of 2d}. The proof is basically algebraic. The key here is to explore the structure of the LHS of \eqref{eq: gain of 2d}. 
		
		\begin{proof}{Lemma \ref{le: gain of 2d}}
		
		\medskip
		\step{1}{Gain of two derivatives} Let
		\begin{equation} \notag
			\av=\alpha \er +\beta \etau +\gamma\phsh Q, \quad
			\vg=\begin{bmatrix}
				\alpha\\ \beta\\ \gamma
			\end{bmatrix},
		\end{equation}
		be a decomposition of the vector $\av$ under the Frenet basis, with its components $(\alpha, \beta, \gamma)$   functions of the radial variable $y$ satisfying the constraint
		\begin{equation} \notag
			\alpha^2 + \beta^2 + (1+\gamma)^2 = 1. 
		\end{equation}
		To ease the notations, we introduce the vectors
		\begin{equation} \notag
			Z_1=(1+Z)\ey,\quad
			Z_2=\frac{Z}{y}\ez-(1+Z)\ex,
		\end{equation}
		based on which we apply Lemma \ref{le: frenet basis} and obtain
		\begin{equation} \label{v w frenet basis}
			\py\av=\py \vg+Z_1\wedge\vg, 
			\quad \pr_{\tau}\av=Z_2\wedge \vg.
		\end{equation}
		We also recall the double wedge formula
		\begin{equation} \notag
			a\wedge(b\wedge c)=(a\cdot c)b-(a\cdot b)c.
		\end{equation}
		In view of \eqref{eq: part linear}, \eqref{eq: mHam func}, we compute
		\begin{equation} \label{eq: b compute 1}
		- \vg \cdot \hJ \mHam \vg
		= - \vg \cdot (\ez + \hw) \wedge \mHam \vg
		= \av\cdot \hv \wedge (\Delta \av+ |\nabla Q|^2 \av )
		= \av\cdot \big( \hv \wedge \Delta \av \big).
		\end{equation}
		which together with the action of the Laplacian in the Frenet basis gives
		\begin{equation} \notag
		\begin{aligned}
		\int\hs\av\cdot \big( \hv \wedge \Delta \av \big)
		= \int\hs\Delta\av\cdot \big( \av\wedge \hv \big) 
		= \int\hs\av\cdot \Delta \big( \av\wedge \hv \big)
		= \int\hs\av\cdot \big( \Delta \av \wedge \hv + 2 \nabla \av \wedge \nabla \hv \big),
		\end{aligned}
		\end{equation}
		and thus 
		\begin{equation} \label{eq: b compute 2}
			\int\hs\av\cdot \big( \hv \wedge \Delta \av \big)
			= \int\hs\av\cdot \big( \nabla \av \wedge \nabla \hv \big).
		\end{equation}
		Combining \eqref{v w frenet basis}, \eqref{eq: b compute 1},
		\eqref{eq: b compute 2}, we let $\vg=\hJ\mHam\vg$, and compute  the LHS of \eqref{eq: gain of 2d}:
		\begin{equation} \label{eq: b identity}
		\begin{aligned}
		- \int \hJ \mHam \vg \cdot (\hJ \mHam)^2 \vg
		= & \int \hJ \mHam \vg \cdot \py \hJ \mHam \vg \wedge \py \hw
		+ \int \hJ \mHam \vg \cdot \big(Z_1\wedge \hJ \mHam \vg\big) \wedge \py \hw \\
		& + \int \hJ \mHam \vg \cdot \py \hJ \mHam \vg \wedge \big(Z_1\wedge(\ez+\hw)\big) \\
		& + \int \hJ \mHam \vg \cdot\big(Z_1\wedge \hJ\mHam\vg\big) \wedge \big(Z_1\wedge(\ez+\hw)\big) \\
		& + \int \hJ \mHam \vg \cdot \big(Z_2\wedge \hJ \mHam \vg\big) \wedge \big(Z_2\wedge(\ez+\hw)\big),
		\end{aligned}
		\end{equation}
		which gives a two-derivatives gain. The normalization of $(\ez+\hw)$ and the structures of $Z_1, Z_2$ produce some cancellations in the RHS of \eqref{eq: b identity}. To see this, we Let
		\begin{equation}
			\vg_2 = \mHam\vg,
		\end{equation}
		then the first term on the RHS of \eqref{eq: b identity} can be simplified to
		\begin{equation} \notag
		\begin{aligned}
		& \int \hJ\varg_2 \cdot \py \hJ\varg_2\wedge\py\hw \\
		= & \int \hJ \varg_2 \cdot \big[ \py \hw \wedge \varg_2
		+ ( \ez + \hw ) \wedge \py \varg_2 \big] \wedge \py \hw \notag\\
		= & \int \big( \hJ \varg_2 \cdot \varg_2 \big) \big( \py \hw \cdot \py \hw \big)
		- \int \big(\hJ \varg_2 \cdot \py \hw \big)\big(\varg_2 \cdot \py \hw \big) \notag\\
		& + \int \big( \hJ \varg_2 \cdot \py \varg_2 \big)\big( ( \ez + \hw ) \cdot \py \hw \big)
		- \int \big( \hJ \varg_2 \cdot ( \ez + \hw ) \big)\big( \py \varg_2 \cdot \py \hw \big) \notag\\
		= & - \int \big( \hJ \varg_2 \cdot \py \hw \big) \big( \varg_2 \cdot \py \hw \big).
		\end{aligned}
		\end{equation}
		The second term on the RHS of \eqref{eq: b identity} is
		\begin{equation} \notag
		\begin{aligned}
			& \int \hJ \varg_2 \cdot \big(Z_1\wedge \hJ \varg_2\big) \wedge \py \hw \\
			= & \int \hJ \varg_2 \cdot \Big[ \hJ \varg_2 \big( Z_1 \cdot \py \hw \big) - Z_1 \big( \hJ \varg_2 \cdot \py \hw \big) \Big] \\
			= & \int(1+Z)\big|\hJ\varg_2\big|^2(\ey\cdot\py\hw) -\int(1+Z)\big(\ey\cdot\hJ\varg_2\big) \big(\py\hw\cdot\hJ\varg_2\big).
		\end{aligned}
		\end{equation}
		Moreover, the third term is
		\begin{equation} \notag
		\begin{aligned}
			& \int \hJ \varg_2 \cdot \py \hJ \varg_2 \wedge \big(Z_1\wedge(\ez+\hw)\big) \\
			= & \int (\hJ\mHam\vg \cdot Z_1) \big( \py \hJ \varg_2 \cdot ( \ez + \hw ) \big) - \big(\hJ\mHam\vg\cdot(\ez+\hw)\big) \big( \py \hJ \varg_2 \cdot \ey \big) \\
			= & \int ( 1 + Z ) \big( \hJ \varg_2 \cdot \ey \big) \Big[ \big( \py \hw \wedge \varg_2 + \hJ \py \varg_2 \big) \cdot ( \ez + \hw ) \Big] \\
			=& -\int(1+Z)\big(\ey\cdot\hJ\varg_2\big) \big(\py\hw\cdot\hJ\varg_2\big).
		\end{aligned}
		\end{equation}
		We observe
		\begin{equation}
			Z_1\wedge \ez=(1+Z)\ex, \quad Z_2\wedge \ez=(1+Z)\ey,
		\end{equation}
		 then the last two terms of \eqref{eq: b identity} can be reformulated as
		\begin{equation}
		\begin{aligned}
		& \int \hJ \varg_2 \cdot \big(Z_1\wedge \hJ \varg_2\big) \wedge\big(Z_1\wedge(\ez+\hw)\big)
		+ \int \hJ \varg_2 \cdot\big(Z_2\wedge \hJ \varg_2\big) \wedge \big(Z_2\wedge(\ez+\hw)\big) \\
		= & - \int \big(\hJ \varg_2 \cdot Z_1 \big) \big( \hJ \varg_2 \cdot Z_1 \wedge \ez \big)
		+ \int \hJ \varg_2 \cdot \big(Z_1\wedge \hJ \varg_2\big) \wedge \big(Z_1\wedge \hw\big) \\
		& - \int \big( \hJ \varg_2 \cdot Z_2 \big) \big( \hJ \varg_2 \cdot Z_2 \wedge \ez \big) 
		+ \int \hJ \varg_2 \cdot \big(Z_2\wedge \hJ \varg_2\big) \wedge \big(Z_2\wedge \hw\big) \\
		= & - \int \frac{Z( 1 + Z )}{y} \big(\ez\cdot\hJ\varg_2\big)
		\big(\ey\cdot\hJ\varg_2\big) \\
		& + \int \hJ \varg_2 \cdot \big(Z_1\wedge \hJ \varg_2\big) \wedge \big(Z_1\wedge \hw\big)
		+ \int \hJ \varg_2 \cdot \big(Z_2\wedge \hJ \varg_2\big) \wedge \big(Z_2\wedge \hw\big)
		\end{aligned}\notag
		\end{equation}
		Injecting these computations into \eqref{eq: b identity}, we have the following refined formula
		\begin{equation} \label{eq: b identity 2}
		\begin{aligned}
			&\quad -\int\hJ\mHam\vg\cdot(\hJ\mHam)^2\vg\\
			&= -\int\big(\hJ\mHam\vg\cdot\py\hw\big) \big(\mHam\vg\cdot\py\hw\big) +\int\onez\big(\ey\cdot\py\hw\big)|\hJ\mHam\vg|^2 \\
			&\quad -2\int\onez\big(\ey\cdot\hJ\mHam\vg\big) \big(\py\hw\cdot\hJ\mHam\vg\big) -\int\frac{Z\onez}{y}\big(\ez\cdot\hJ\mHam\vg\big) \big(\ey\cdot\hJ\mHam\vg\big) \\
			&\quad +\int\hJ\mHam\vg\cdot \big(Z_1\wedge\hJ\mHam\vg\big) \wedge\big(Z_1\wedge\hw\big) +\int\hJ\mHam\vg\cdot \big(Z_2\wedge\hJ\mHam\vg\big) \wedge\big(Z_2\wedge\hw\big).	
		\end{aligned}
		\end{equation}
		
		\medskip
		\step{2}{Extraction of the leading terms} To derive precise bounds for the RHS of \eqref{eq: b identity 2}, let us introduce the following decomposition of $\hw,\wzt$ (This is different from decomposition~\eqref{eq: wz decomposition}):
		\begin{equation} \label{eq: hw decomp}
			\hw = \wzt^1+\wzt^2+\w, \quad\mbox{with}\quad
			\lf\{\begin{aligned}	
				&\wzt^1 = b\phsh\varpt_{0,1} = b\phsh\chib\varp_{0,1}, \\
				&\wzt^2 = \wzt - \wzt^1.
			\end{aligned}\rg.
		\end{equation}
		By \eqref{eq: ip bound 13}, \eqref{eq: ip bound 21}, the first term on the RHS of \eqref{eq: b identity 2} is bounded by
		\begin{equation} \notag
			\Big|\hs\int\big(\hJ\mHam\vg\cdot\py\hw\big)\big(\mHam \vg\cdot\py\hw\big)\Big|
			\leq\,\|\py\hw\|_{L^{\infty}}^2 \|\mHam\vg\|_{L^2}^2 \les b\,\db\,\|\mHam\vg\|_{L^2}^2.
		\end{equation}
		According to \eqref{eq: hw decomp}, we apply \eqref{eq: ip bound 15}, \eqref{eq: ip bound 22}, and $|a|\leq b\lgba^{-1}$ to compute the second term		
		\begin{equation} \notag
		\begin{aligned}
			& \int\onez\big(\ey\cdot\py\hw\big)|\hJ\mHam\vg|^2 \\
			=&\, b\int\onez\,\py\varpt_{0,1}^{\second}\,|\hJ\mHam \vg|^2 +\int\onez\Big(\ey\cdot\py(\wzt^2+\w)\Big)|\hJ\mHam\vg|^2 \\
			=&\, b\int\onez\,\py\varpt_{0,1}^{\second}\,|\hJ\mHam\vg|^2 +O\Big(b\,\db\,\|\hJ\mHam\vg\|_{L^2}^2\Big).
		\end{aligned}
		\end{equation}
		Similarly, the third term on the RHS of \eqref{eq: b identity 2} is
		\begin{equation} \notag
		\begin{aligned}
			& -2\int\onez\big(\ey\cdot\hJ\mHam\vg\big)\big(\py\hw\cdot\hJ\mHam\vg\big)\\
			=& -2b\int\onez\big(\ey\cdot\hJ\mHam\vg\big) \big(\py\varpt_{0,1}\cdot\hJ\mHam\vg\big) \\
			& -2\int\onez\big(\ey\cdot\hJ\mHam\vg\big) \Big(\py(\wzt^2+\w)\cdot\hJ\mHam\vg\Big) \\
			\leq & -2b\int\onez\,\py\varpt_{0,1}^{\first}\,\big(\ex\cdot\hJ\mHam\vg\big)\big(\ey\cdot\hJ\mHam\vg\big) \\
			& -2b\int\onez\,\py\varpt_{0,1}^{\second}\,\big(\ey\cdot\hJ\mHam\vg\big)^2 +O\Big(b\,\db\,\|\hJ\mHam\vg\|_{L^2}^2\Big).
		\end{aligned}
		\end{equation}
		Moreover, from \eqref{eq: hJ} we note the simple identities
		\begin{equation} \notag
		\begin{aligned}
			\ez\cdot\hJ\mHam\vg &=\ez\cdot(\ez+\hw)\wedge\mHam\vg
				=\ez\cdot\hw\wedge\mHam\vg, \\
			\ex\cdot\mHam\vg &=\ey\cdot(\ez\wedge\mHam\vg) =\ey\cdot\hJ\mHam\vg -\ey\cdot\Rw\mHam\vg, \\
			\ey\cdot\mHam\vg &=-\ex\cdot(\ez\wedge\mHam\vg) 
				=-\ex\cdot\hJ\mHam\vg +\ey\cdot\Rw\mHam\vg,
		\end{aligned}
		\end{equation}
		from which the fourth term on the RHS of \eqref{eq: b identity 2} is bounded by
		\begin{equation} \notag
		\begin{aligned}
			& -\int\frac{Z\onez}{y}\big(\ez\cdot\hJ\mHam\vg\big) \big(\ey\cdot\hJ\mHam\vg\big) \\
			\leq & -b\hs\hs\int\hs\frac{Z\onez}{y} \big(\ez\cdot \varpt_{0,1}\hs\wedge\hs\mHam\vg\big) \big(\ey\cdot\hJ\mHam\vg\big) 
			+\hs\int\hs O\bigg(\frac{|\wzt^2|+|\w|}{y(1+y^2)}\bigg) \,|\hJ\mHam\vg|^2 \\
			\leq &\, b\hs\int\frac{Z\onez}{y}\,\varpt_{0,1}^{\first}\, \big(\ex\cdot\hJ\mHam\vg\big) \big(\ey\cdot\hJ\mHam\vg\big)
			+b\hs\int\frac{Z\onez}{y}\,\varpt_{0,1}^{\second}\,\big(\ey\cdot\hJ\mHam\vg\big)^2 \\
			& +O\bigg(b\phsh\bigg\|\frac{\chib T_1}{y(1+y^2)}\bigg\|_{L^{\infty}}\|\hw\|_{L^{\infty}}\hs\bigg)\,\|\hJ\mHam\vg \|_{L^2}^2
			+O\bigg(\,\bigg\|\frac{|\wzt^2|+|\w|}{y(1+y^2)}\bigg\|_{L^{\infty}}\bigg)\,\|\hJ\mHam\vg\|_{L^2}^2 \\
			\leq &\, b\hs\int\frac{Z\onez}{y}\,\varpt_{0,1}^{\first}\, \big(\ex\cdot\hJ\mHam\vg\big) \big(\ey\cdot\hJ\mHam\vg\big) \\
			& +b\hs\int\frac{Z\onez}{y}\,\varpt_{0,1}^{\second}\,\big(\ey\cdot\hJ\mHam\vg\big)^2 
			+O\Big(b\,\db\,\|\hJ\mHam\vg\|_{L^2}^2\Big).
		\end{aligned}
		\end{equation}
		Furthermore, we compute the fifth term on the RHS of \eqref{eq: b identity 2}:
		\begin{equation} \notag
		\begin{aligned}
			& \int \hJ\mHam\vg\cdot\big(Z_1\wedge \hJ\mHam\vg\big)\wedge\big(Z_1\wedge\hw\big) \\
			=&\, \int\big(Z_1\cdot\hJ\mHam\vg\big) \big(\hw\cdot Z_1\wedge\hJ\mHam\vg\big) \\
			=&\, b\int\onez^2 \Big(\varpt_{0,1}\cdot \ey\wedge\hJ\mHam\vg\Big) \big(\ey\cdot\hJ\mHam\vg\big) \\
			& + \int\onez^2 \big((\wzt^2+\w)\cdot \ey\wedge\hJ\mHam\vg\big) \big(\ey\cdot\hJ\mHam\vg\big) \\
			\leq &\, b\int\onez^2\,\varpt_{0,1}^{\first}\,\big(\ey\cdot\hJ\mHam\vg\big) \big(\ez\cdot\hJ\mHam\vg\big) +O\Big(b\,\db\,\|\hJ\mHam\vg\|_{L^2}^2\Big).
		\end{aligned}
		\end{equation}
		Finally, for the last term on the RHS of \eqref{eq: b identity 2}, we have
		\begin{equation} \notag
		\begin{aligned}
			&\quad \int\hJ\mHam\vg\cdot\big(Z_2\wedge\hJ\mHam \vg\big) \wedge\big(Z_2\wedge\hw\big) \\
			& = \int\big(Z_2\cdot\hJ\mHam\vg\big) \big(\hw\cdot Z_2\wedge\hJ\mHam\vg\big) \\
			& = b\int\big(Z_2\cdot\hJ\mHam\vg\big) \Big(\big(\varpt_{0,1}\wedge Z_2\big)\cdot\hJ\mHam\vg\Big) +\int\big(Z_2\cdot\hJ\mHam\vg\big)\Big(\big((\wzt^2+\w)\wedge Z_2\big)\cdot \hJ\mHam\vg\Big) \\
			& \leq b\int\frac{Z\onez}{y}\,\varpt_{0,1}^{\first}\, \big(\ex\cdot\hJ\mHam\vg\big) \big(\ey\cdot\hJ\mHam\vg\big) \\
			&\quad -b\int\frac{Z\onez}{y}\,\varpt_{0,1}^{\second}\, \big(\ex\cdot\hJ\mHam\vg\big)^2 
			+\rest_{\vg} +O\Big(b\,\db\,\|\hJ\mHam\vg\|_{L^2}^2\Big),
		\end{aligned}
		\end{equation}
		where $\rest_{\vg}$ consists of the terms involving $(\ez\cdot\hJ\mHam\vg)$: 
		\begin{equation} \notag
		\begin{aligned} 
			\rest_{\vg} &= -b\hs\int\bigg(\frac{Z}{y}\bigg)^2\,\varpt_{0,1}^{\first}\,\big(\ez\hs\cdot\hs\hJ\mHam\vg\big) \big(\ey\hs\cdot\hs\hJ\mHam\vg\big) 
			+ b\hs\hs\int\frac{Z\onez}{y}\,\varpt_{0,1}^{\second}\, \big(\ez\hs\cdot\hs\hJ\mHam\vg\big)^2 \\
			&\quad +b\int\bigg(\frac{Z}{y}\bigg)^2\,\varpt_{0,1}^{\second}\, \big(\ez\hs\cdot\hs\hJ\mHam\vg\big) \big(\ex\hs\cdot\hs\hJ\mHam\vg\big) 
			-b\int\onez^2\,\varpt_{0,1}^{\first}\, \big(\ez\hs\cdot\hs\hJ\mHam\vg\big) \big(\ex\hs\cdot\hs\hJ\mHam\vg\big).
		\end{aligned}
		\end{equation}
		Combing these computations with \eqref{eq: b identity 2}, and  applying the definition of function $Z$ \eqref{eq: lambda phi Z} the operator $A$ \eqref{eq: Ham factori}, we obtain the following inequality
		\begin{equation} \label{b ineq 3 estimate}
		\begin{aligned}
			&\quad -\int\hJ\mHam\vg\cdot(\hJ\mHam)^2\vg \\
			&\leq b\int\onez\A\phsh\varpt_{0,1}^{\second}\Big[\big(\ey\cdot\hJ\mHam\vg\big)^2-\big(\ex\cdot\hJ\mHam\vg\big)^2 -\big(\ez\cdot\hJ\mHam\vg\big)^2 \Big] \\
			&\quad + 2b\int\onez\A\phsh\varpt_{0,1}^{\first}\big(\ex\cdot\hJ\mHam\vg\big)\big(\ey\cdot\hJ\mHam\vg\big) \\
			&\quad +b\int\frac{V}{y^2}\,\varpt_{0,1}^{\second}\big(\ex\cdot\hJ\mHam\vg\big)\big(\ez\cdot\hJ\mHam\vg\big)
			-b\int\frac{V}{y^2}\,\varpt_{0,1}^{\first}\big(\ey\cdot\hJ\mHam\vg\big)\big(\ez\cdot\hJ\mHam\vg\big).
		\end{aligned}
		\end{equation}
	
		\medskip
		\step{3}{Upper bound on the quadratic terms} Recalling \eqref{eq: Ham factori}, for $\ti{T}_1=\chib T_1$, we claim that there exists a constant $d_1\in (0,1)$ such that
		\begin{equation} \label{AT1 1+Z}
			\forall y>0, \quad 0\leq\onez\A\ti{T}_1 
			\leq\frac{1}{2}(1-d_1).
		\end{equation}
		We prove this inequality for $T_1$, but the same result follows for $\ti{T}_1$ immediately. First we note from \eqref{eq: A reformulation} that
		\begin{equation} \notag
			\frac{1}{y\varl\phi} \py\big(y\varl\phi\, \A T_1\big)
			= \As(\A T_1) = \varl\phi,
		\end{equation}
		which yields the explicit expression of $\onez\A T_1$:
		\begin{equation} \label{eq: onez AT_1}
			\onez\A T_1 = \frac{1+Z}{y \varl\phi} \int_0^y\tau \varl\phi^2(\tau) d\tau
			= \frac{2\log(1+y^2)}{y^2} -\frac{2}{1+y^2}.
		\end{equation}
		To show \eqref{AT1 1+Z}, we define the functions
		\begin{equation} \notag
			f_1(y) = 2\log(1+y)-\frac{2y}{1+y}, \quad
			f_2(y) = f_1(y) -\frac{y}{2}.
		\end{equation}
		By direct computations, we see 
		\begin{equation} \notag
			f_1(0)=f_2(0)=0, \quad\mbox{and}\quad
			\forall\phsh y>0,\quad f_1'(y) > 0 \geq f_2'(y),
		\end{equation}
		where $f_2'(y)=0$ holds for $y=1$ only. This actually implies
		\begin{equation} \notag
			\forall y>0, \quad
			0<\frac{2\log(1+y^2)}{y^2}-\frac{2}{1+y^2}<\frac{1}{2},
		\end{equation}
		which together with \eqref{eq: onez AT_1} yields \eqref{AT1 1+Z}. Consequently, the quadratic terms in \eqref{b ineq 3 estimate} can be estimated. By the explicit formula of $\varp_{0,1}$ \eqref{eq: varp function}, we have
		\begin{equation} \label{b compute 3}
		\begin{aligned}
			&\quad b\int\onez\A\phsh\varpt_{0,1}^{\second}\Big[\big(\ey\cdot\hJ\mHam\vg\big)^2-\big(\ex\cdot\hJ\mHam\vg\big)^2 -\big(\ez\cdot\hJ\mHam\vg\big)^2\Big]  \\
			& \leq \frac{b|\ao|}{\ao^2+\bo^2} \,\big\|\onez\A\ti{T}_1 \big\|_{L^{\infty}} \int|\hJ\mHam\vg\big|^2 
			\leq \frac{b(1-d_1)|\ao|}{2(\ao^2+\bo^2)}\, \|\hJ\mHam\vg\|_{L^2}^2,
		\end{aligned}
		\end{equation}
		and also
		\begin{equation} \label{b compute 4}
		\begin{aligned}
			&\quad 2b\int\onez\A\phsh\varpt_{0,1}^{\first}\big(\ex\cdot\hJ\mHam\vg\big)\big(\ey\cdot\hJ\mHam\vg\big) \\
			& \leq \frac{2b|\bo|}{\ao^2+\bo^2} \,\big\|\onez\A\ti{T}_1 \big\|_{L^{\infty}} \int|\hJ\mHam\vg|^2
			\leq \frac{b(1-d_1)|\bo|}{2(\ao^2+\bo^2)}\, \|\hJ\mHam\vg\|_{L^2}^2.
		\end{aligned}
		\end{equation}
		For the last line of \eqref{b ineq 3 estimate}, from the smallness of $\w$ \eqref{eq: ip bound 8}, \eqref{eq: ip bound 20}, we observe
		\begin{equation} \notag
			|\ez\cdot\hJ\mHam\vg| = |\ez\cdot\hw\wedge\mHam\vg|
			\leq |\hw|\,|\mHam\vg| \les \db|\mHam\vg|,
		\end{equation}
		and hence
		\begin{equation} \label{b compute 5}
		\begin{aligned}
			&\quad b\int\frac{V}{y^2}\,\varpt_{0,1}^{\second}\big(\ex\cdot\hJ\mHam\vg\big)\big(\ez\cdot\hJ\mHam\vg\big)
			-b\int\frac{V}{y^2}\,\varpt_{0,1}^{\first}\big(\ey\cdot\hJ\mHam\vg\big)\big(\ez\cdot\hJ\mHam\vg\big) \\
			& \les  b\,\db \bigg\|\frac{V}{y^2}\,\ti{T}_1\bigg\|_{L^{\infty}} \int|\hJ\mHam\vg|\cdot|\mHam\vg| 
			\les b\,\db \|\mHam\vg\|_{L^2}^2.
		\end{aligned}
		\end{equation}
		Eventually, we inject \eqref{b compute 3}, \eqref{b compute 4}, \eqref{b compute 5} into \eqref{b ineq 3 estimate}, and then \eqref{eq: gain of 2d} follows.
		\end{proof}
	\end{appendix}
	
	\bibliographystyle{plain}
	
\end{document}